\documentclass[11pt]{article}
\usepackage{bbm}
\usepackage{geometry}
\usepackage{color}
\usepackage{caption}
\usepackage{amsfonts}
\usepackage{latexsym, amssymb, amsmath, amscd, amsthm, amsxtra}
\usepackage{mathtools}
\usepackage{enumerate}
\usepackage[all]{xy}
\usepackage{mathrsfs}
\usepackage{fancyhdr}
\usepackage{listings}
\usepackage{hyperref}
\usepackage{makecell}
\usepackage{tikz}

\def\11{\mathbbm{1}}
\def\Gc{\mathcal{G}}

\def\ER{Erd\H{o}s-R\'enyi\ }

\def\K{\mathcal{K}}

\def\Qb{\mathbb Q}

\newcommand{\Pb}{\mathbb P}

\newcommand\defby{\overset{\triangle}{=}}

\newcommand{\doublesetminus}
{\mathbin{\setminus\mkern-5mu\setminus}}
\newcommand{\doublesymdiff}{%
  \mathbin{\text{\raisebox{0.20ex}{\scalebox{0.60}{$\triangle$}}\hspace{-0.70em}$\triangle$}}%
}

\newtheorem{thm}{Theorem}[section]

\newtheorem{defn}[thm]{Definition}

\newtheorem{claim}[thm]{Claim}

\newtheorem{proposition}[thm]{Proposition}

\newtheorem{lemma}[thm]{Lemma}
\newtheorem{cor}[thm]{Corollary}

\newtheorem{remark}[thm]{Remark}

\numberwithin{equation}{section}

\newenvironment{definition}[1][Definition]{\begin{trivlist}
\item[\hskip \labelsep {\bfseries #1}]}{\end{trivlist}}

\hypersetup{colorlinks=true,linkcolor=blue}

\begin{document}
\title{A computational transition for detecting correlated stochastic block models by low-degree polynomials}
	
\author{ Guanyi Chen \thanks{School of Mathematical Sciences, Peking University.} \and Jian Ding $^{*}$ \and Shuyang Gong $^{*}$ \and Zhangsong Li $^{*}$ }

\maketitle

\begin{abstract}
Detection of correlation in a pair of random graphs is a fundamental statistical and computational problem that has been extensively studied in recent years. In this work, we consider a pair of correlated (sparse) stochastic block models $\mathcal{S}(n,\tfrac{\lambda}{n};k,\epsilon;s)$ that are subsampled from a common parent stochastic block model $\mathcal S(n,\tfrac{\lambda}{n};k,\epsilon)$ with $k=O(1)$ symmetric communities, average degree $\lambda=O(1)$, divergence parameter $\epsilon$, and subsampling probability $s$. 

For the detection problem of distinguishing this model from a pair of independent \ER graphs with the same edge density $\mathcal{G}(n,\tfrac{\lambda s}{n})$, we focus on tests based on \emph{low-degree polynomials} of the entries of the adjacency matrices, and we determine the threshold that separates the easy and hard regimes. More precisely, we show that this class of tests can distinguish these two models if and only if $s> \min \{ \sqrt{\alpha}, \frac{1}{\lambda \epsilon^2} \}$, where $\alpha\approx 0.338$ is the Otter's constant and $\frac{1}{\lambda \epsilon^2}$ is the Kesten--Stigum threshold. Combining a reduction argument in \cite{Li25+}, our hardness result also implies low-degree hardness for partial recovery and detection (to independent block models) when $s< \min \{ \sqrt{\alpha}, \frac{1}{\lambda \epsilon^2} \}$. Finally, our proof of low-degree hardness is based on a conditional variant of the low-degree likelihood calculation.
\end{abstract}

\tableofcontents

\section{Introduction}

In this paper, we consider a pair of correlated sparse stochastic block models with a constant number of symmetric communities, defined as follows. For convenience, denote by $\operatorname{U}_n$ the collection of unordered pairs $(i,j)$ with $1\le i\neq j\le n$.

\begin{defn}[Stochastic block model] {\label{def-SBM}}
    Given an integer $n\ge 1$ and three parameters $k \in \mathbb{N}, \lambda>0, \epsilon \in (0,1)$, we define a random graph $G$ as follows: 
    \begin{itemize}
    \item Sample a labeling $\sigma_* \in [k]^{n} = \{ 1,\ldots,k \}^{n}$ uniformly at random; 
    \item For every distinct pair $(i,j) \in \operatorname{U}_n$, we let $G_{i,j}$ be an independent Bernoulli variable such that $G_{i,j} = 1$ (which represents that there is an undirected edge between $i$ and $j$) with probability $\frac{(1+(k-1)\epsilon)\lambda}{n}$ if $\sigma_*(i) = \sigma_*(j)$ and with probability $\frac{(1-\epsilon)\lambda}{n}$ if $\sigma_*(i) \neq \sigma_*(j)$. 
    \end{itemize}
    In this case, we say that $G$ is sampled from a stochastic block model $\mathcal S(n, \tfrac{\lambda}{n}; k, \epsilon)$.
\end{defn}

\begin{defn}[Correlated stochastic block models] {\label{def-correlated-SBM}}
    Given an integer $n\ge 1$ and four parameters $k \in \mathbb{N}, \lambda>0, \epsilon,s \in (0,1)$, for $(i,j) \in \operatorname{U}_n$ let $J_{i,j}$ and $K_{i,j}$ be independent Bernoulli variables with parameter $s$. In addition, let $\pi_*$ be an independent uniform permutation of $[n]=\{1,\dots,n\}$. Then, we define a triple of correlated random graphs $(G,A,B)$ such that $G$ is sampled from a stochastic block model $\mathcal{S}(n,\tfrac{\lambda}{n};k,\epsilon)$, and conditioned on the realization of $G$ (note that we identify a graph with its adjacency matrix),
    \[
    A_{i,j} = G_{i,j}J_{i,j}, B_{i,j} = G_{\pi_*^{-1}(i),\pi_*^{-1}(j)} K_{i,j} \,.
    \]
    We denote the joint law of $(\sigma_*,\pi_*,G,A,B)$ as $\Pb_{*,n}:=\Pb_{*,n,\lambda;k,\epsilon;s}$, and we denote the marginal law of $(A,B)$ as $\Pb_n:=\Pb_{n,\lambda;k,\epsilon;s}$.
\end{defn}
Two basic problems regarding correlated stochastic block models are as follows: (1) the detection problem, i.e., testing $\mathbb{P}_n$ against $\mathbb{Q}_n$ where $\mathbb{Q}_n$ is the law of two independent \ER graphs on $[n]$ with edge density $\tfrac{\lambda s}{n}$; (2) the recovery problem, i.e., recovering the latent matching $\pi_*$ and the latent community labeling $\sigma_*$ from $(A,B) \sim \Pb_n$. Our focus is on understanding the power and limitations of {\em computationally efficient} tests, that is, tests that can be performed by polynomial-time algorithms. In light of the lack of complexity-theoretic tools to prove computational hardness of {\em average-case} problems such as the one under consideration (where the input is random), currently the leading approaches for {\em demonstrating} hardness are based on either average-case reductions which formally relate different average-case problems to each other (see, e.g., \cite{BBH18, BB20} and references therein) or based on unconditional lower bounds against restricted classes of algorithms (see e.g. \cite{BPW18, Gamarnik21}). 

Our main result establishes a sharp computational transition for algorithms restricted to {\em low-degree polynomial tests}. This is a powerful class of tests that include statistics such as small subgraph counts. It is by now well-established that these low-degree tests are useful proxies for computationally efficient tests, in the sense that the best-known polynomial-time algorithms for a wide variety of high-dimensional testing problems are captured by the low-degree class; see e.g. \cite{Hopkins18, KWB22}.

\begin{thm}[Computational detection threshold for low-degree polynomials, informal] {\label{MAIN-THM-detection}}
    With the observation of a pair of random graphs $(A,B)$ sampled from either $\Pb_n$ or $\Qb_n$, we have the following (below degree-$\omega(1)$ means that degree tends to infinity as $n\to\infty$).
    \begin{enumerate}
    \item[(1)] When $s>\tfrac{1}{\lambda\epsilon^2}$ or $s>\sqrt{\alpha}$ where (throughout the paper) $\alpha\approx 0.338$ is the Otter's constant, for any $D_n\to\infty$ and $D_n=o\big(\tfrac{\log n}{\log \log n}\big)$ there is an algorithm $\mathsf{Alg}$ based on degree-$D_n$ polynomials that successfully distinguishes $\Pb_n$ and $\Qb_n$ in the sense that (we write $\mathsf{Alg}$ outputs $\mathbb P_n$/$\mathbb Q_n$ if $\mathsf{Alg}$ decides the sample is from $\mathbb P_n$/$\mathbb Q_n$)
    \begin{equation}{\label{eq-success}}
        \mathbb P_n(\mathsf{Alg}\mbox{ outputs } \mathbb Q_n) + \mathbb Q_n(\mathsf{Alg}\mbox{ outputs } \mathbb P_n)  = o(1)\,.
    \end{equation}
    In addition, this algorithm has running time $n^{2+o(1)}$.
    \item[(2)] When $s<\min\{ \sqrt{\alpha}, \frac{1}{\lambda \epsilon^2} \}$, there is evidence suggesting that all algorithms based on degree-$O(n^{o(1)})$ polynomials fail to distinguish $\Pb_n$ and $\Qb_n$. See Theorem~\ref{main-thm-detection-lower-bound} for a precise statement.
    \end{enumerate}
\end{thm}
\begin{remark}
    We briefly remark on the information-theoretic side of the testing problem under consideration. In \cite{MSS23, MSS24+} (which extended previous results in \cite{MNS15, Massoulie14}), it was shown that when $\epsilon^2 \lambda s>C_*$ where $C_*=C_*(k)\leq 1$ is the reconstruction-on-tree threshold, it is information-theoretically possible to detect a single block model $\mathcal S(n,\tfrac{\lambda s}{n};k,\epsilon)$ against $\mathcal G(n,\tfrac{\lambda s}{n})$. In addition, in a series of works \cite{WXY23, MWXY23, DD23a, Feng25+} it was shown that when $s>\min\{ \sqrt{\alpha}, \tfrac{1}{\sqrt{\lambda}} \}$ where $\alpha$ is the Otter's constant, it is information-theoretically possible to detect a pair of correlated \ER graphs $\mathcal G(n,\tfrac{\lambda}{n};s)$ (which means that these two graphs are subsampled from a parent \ER graph $\mathcal G(n,\tfrac{\lambda}{n})$ with subsampling probability $s$) against two independent \ER graphs $\mathcal G(n,\tfrac{\lambda s}{n})$. It seems plausible that the feasibility results in \cite{DD23a, Feng25+} can be extended to the setting of testing $\mathcal S(n,\tfrac{\lambda}{n};k,\epsilon;s)$ against two independent $\mathcal G(n,\tfrac{\lambda s}{n})$, which would then lead to the following: the testing problem between $\Pb_n$ and $\Qb_n$ is feasible at least when
    \begin{align*}
        s > \min\big\{ \tfrac{C_*(k)}{\epsilon^2 \lambda}, \sqrt{\alpha}, \tfrac{1}{\sqrt{\lambda}} \big\} \,.
    \end{align*}
    Determining the exact information-theoretic threshold for this testing problem  remains an intriguing open question. 
\end{remark}

\begin{remark}
    Building on Theorem~\ref{MAIN-THM-detection} and a natural strengthening of the low-degree conjecture, a recent work \cite{Li25+} provides evidences that the related \emph{partial matching recovery} (i.e., recovering a positive fraction of the coordinates of $\pi_*$) problem is computationally impossible when $s<\min\{ \sqrt{\alpha}, \tfrac{1}{\lambda \epsilon^2} \}$. 
\end{remark}
\begin{remark}
    It is natural to consider the related problem of testing $\Pb_n$ against $\widetilde{\Qb}_n$, where $\widetilde{\Qb}_n$ is the law of two independent stochastic block models $\mathcal S(n,\tfrac{\lambda s}{n};k,\epsilon)$, and we remark that Theorem~\ref{MAIN-THM-detection} is helpful also in understanding the computational feasibility of this related testing problem. The upper bound (i) in Theorem~\ref{MAIN-THM-detection} is the relatively easy part of our result, which essentially generalizes the ideas in \cite{MWXY21+, MWXY23} for the detection problem in correlated \ER graphs. However, this requires a modified proof in the SBM case and it is plausible that our proof can be generalized to give an efficient algorithm that strongly detects $\Pb_n$ from $\widetilde{\Qb}_n$ when $s>\sqrt{\alpha}$. Furthermore, the lower bound (ii) is the most technical part and it also leads to several corollaries in the testing problem between $\Pb_n$ and $\widetilde{\Qb}_n$. Let us focus on the regime where $s< \min\{\sqrt{\alpha}, \tfrac{1}{\lambda \epsilon^2}\}$ such that marginally each graph is an SBM below the Kesten-Stigum threshold. When $k \leq 4$, it was shown in \cite{MNS15, MSS23} that $\widetilde{\Qb}_n$ is contiguous with $\Qb_n$ (although their result focuses on a single graph, extending it to two \emph{independent} graphs is straightforward). Thus, our result provides evidence suggesting that all the three measures
    \begin{align*}
        \Pb_n \,, \quad \Qb_n \,, \quad \widetilde{\Qb}_n
    \end{align*}
    cannot be distinguished by efficient algorithms. For general $k \in \mathbb N$, it was shown in \cite{Li25+} that $\widetilde{\Qb}_n$ is contiguous with $\Qb_n$ in the algorithmic sense (see Definition~2.3 therein), provided (in addition) the average degree $\lambda s$ is sufficiently large. Thus, in this case our result also suggests that all these three measures cannot be distinguished by efficient algorithms as long as $s<\sqrt{\alpha}$. 
\end{remark}

\begin{remark}
    For the hypothesis testing problem between $\Pb_n$ (a pair of correlated SBMs) and $\widetilde{\Qb}_n$ (a pair of independent SBMs), we point out that when marginally each graph is above the Kesten-Stigum threshold, it seems that the (computational) detection threshold should be strictly below $\sqrt{\alpha}$, as in this case it is possible to recover the community labels for a positive fraction of vertices, and thus (intuitively) it is possible to break the Otter's threshold by counting colored trees. This was further supported by a recent work \cite{CDGL25+}, which has shown that in the case of $k=2$, strong detection between $\Pb_n$ and $\widetilde{\Qb}_n$ is achievable when $s\geq \sqrt{\alpha}-\delta$ and $\lambda>\Delta:=\Delta(\epsilon,\delta)$ where $\delta$ is a sufficiently small constant and $\Delta$ is a sufficiently large constant that depends on $\epsilon,\delta$. However, rigorous analysis in the general case seems of substantial challenge and we leave it for future work.
\end{remark}

\subsection{Backgrounds and related works}

$\textbf{Community detection.}$ Introduced in \cite{HLL83}, the stochastic block model is a canonical probabilistic generative model for networks with community structure and as a result has received extensive attention over the past decades. In particular, it serves as an essential benchmark for studying the behavior of clustering algorithms on average-case networks (see, e.g., \cite{SN97, BC09, RCY10}). In the past few decades, extensive efforts have been dedicated toward understanding the statistical and computational limits of various inference tasks for this model, including exact community recovery in the logarithmic degree region \cite{AS15, ABH16} and community detection/weak community recovery in the constant degree region. Since the latter case is closely related to our work, we next review progress on this front, largely driven by a seminal paper in statistical physics \cite{DKMZ11} where the authors predicted that: (i) for all $k \in \mathbb N$, it is possible to use efficient algorithms to detect communities better than random if $\epsilon^2 \lambda>1$; (ii) for $k \leq 4$ it is information theoretically impossible to detect better than random if $\epsilon > 0$ and $\epsilon^2 \lambda<1$; (iii) for $k \geq 5$, it is information theoretically possible to detect better than random for some $\epsilon>0$ with $\epsilon^2 \lambda<1$, but not in a computationally efficient way (that is to say, statistical-computational gap emerges for $k\geq 5$).

The threshold $\epsilon^2 \lambda=1$, known as the Kesten--Stigum (KS) threshold, represents a natural threshold for the trade off between noise and signals. It was first discovered in the context of the broadcast process on trees in \cite{KS66}. Recent advancements have verified (i) by analyzing related algorithms based on non-backtracking or self-avoiding walks \cite{Massoulie14, MNS18, BLM15, BDG+16, AS16, AS18}. Moreover, the information-theoretic aspects of (ii) and (iii) were established in a series of works \cite{BMNN16, AS18, CS21+, MSS23, MSS24+}. Regarding the computational aspect of (iii), compelling evidence was provided in \cite{HS17} suggesting the emergence of a statistical-computational gap.

$\textbf{Correlated random graphs.}$ Graph matching (also known as graph alignment) refers to finding the vertex correspondence between two graphs such that the total number of common edges is maximized. It plays an essential role in various applied fields such as computational biology \cite{SXB08, VCL+15}, social networking \cite{NS08, NS09}, computer vision \cite{BBM05, CSS06} and natural language processing \cite{HNM05}. From a theoretical perspective, perhaps the most widely studied model is the correlated \ER graph model \cite{PG11}, where the observations are two \ER graphs with correlated pairs of edges through a latent vertex bijection $\pi_*$. 

Recent research has focused on two important and entangling issues for this model: the information threshold (i.e., the statistical threshold) and the computational phase transition. On the one hand, collective efforts from the community have led to rather satisfying understanding on information thresholds for the problem of correlation detection and vertex matching \cite{CK16, CK17, HM23, GML20+, WXY22, WXY23, DD23a, DD23b}. On the other hand,  in extensive works including \cite{PG11, YG13, LFP14, KHG15, FQMK+16, SGE17, BCL+19, DMWX21, FMWX23a, FMWX23b, BSH19, CKMP20, DCK+19, MX20, GM20, GML20+, MRT21, MRT23, MWXY21+, GMS22+, MWXY23, DL22+, DL23+}, substantial progress on algorithms were achieved and the state-of-the-art can be summarized as follows: in the sparse regime, efficient matching algorithms are available when the correlation exceeds the square root of Otter’s constant (which is approximately 0.338) \cite{GML20+, GMS22+, MWXY21+, MWXY23}; in the dense regime, efficient matching algorithms exist as long as the correlation exceeds an arbitrarily small constant \cite{DL22+, DL23+}. Roughly speaking, the separation between the sparse and dense regimes mentioned above depends on whether the average degree grows polynomially or sub-polynomially. In addition, while proving the hardness of typical instances of the graph matching problem remains challenging even under the assumption of $\operatorname{P}\neq  \operatorname{NP}$, evidence based on the analysis of a specific class known as low-degree polynomials from \cite{DDL23+} indicates that the state-of-the-art algorithms may essentially capture the correct computational thresholds.

$\textbf{Correlated stochastic block models.}$ The study of correlated stochastic block models originated in \cite{LSFPCVPP15,OGE16}, serving as a framework to understand the interplay between community recovery and graph matching. Previous results on this model focus mainly on the logarithmic degree region, where their interest is to study the interplay between the exact community recovery and the exact graph matching \cite{RS21, GRS22, YC23+, YSC23, CR24, RZ25}. In particular, \cite{GRS22} showed that in this regime there are indeed subtle interactions between these two inference tasks, since one can recover the community (although not necessarily by efficient algorithms) even when neither the exact community recovery in a single graph nor the exact matching recovery in \ER graphs is possible. This line of inquiry was further extended to multiple correlated SBMs by \cite{RZ25}.

In this work, however, we are interested in the related detection problem where the average degree is a constant. The goal of our work is to understand how side information in the form of multiple correlated stochastic block models affects the threshold given by single-community detection or correlation detection. As shown by Theorem~\ref{MAIN-THM-detection}, somewhat surprisingly, it seems that such side information cannot be exploited by efficient algorithms in this particular region.

$\textbf{Low-degree tests.}$ Our hardness result is based on the study of a specific class of algorithms known as {\em low-degree polynomials}. Somewhat informally, the idea is to study degree-$D$ multivariate polynomials in the input variables whose real-valued output separates (see Definition~\ref{def-strong-separation}) samples from the planted and null distributions. The idea to study this class of tests emerged from the line of works \cite{BHK+19, HS17, HKP+17, Hopkins18}; see also \cite{KWB22} for a survey. Tests of degree $O(\log n)$ are generally taken as a proxy for polynomial-time tests, as they capture many leading algorithmic approaches such as spectral methods, approximate message passing and small subgraph counts.

There is now a standard method for obtaining low-degree testing bounds based on the low-degree likelihood ratio (see \cite[Section~2.3]{Hopkins18}), which boils down to finding an orthonormal basis of polynomials with respect to the null distribution and computing the expectations of these basis polynomials under the planted distribution. However, our setting is more subtle because the second moment of the low-degree likelihood ratio diverges due to some rare ``bad'' events under the planted distribution. We therefore need to carry out a conditional low-degree argument in which the planted distribution is conditioned on some ``good'' event. 

Conditional low-degree arguments of this kind have appeared before in a few instances \cite{BAH+22, CGH+22, DMW23+, DDL23+}, but our argument differs in a technical level. Prior works \cite{BAH+22, CGH+22} chose to condition on an event that would seem to make direct computations with the orthogonal polynomials rather complicated; to overcome this, they bounded the conditional low-degree likelihood ratio in an indirect way by first relating it to a certain ``low-overlap'' second moment (also called the Franz-Parisi criterion in \cite{BAH+22}). In addition, in \cite{DMW23+} the authors overcame this issue by conditioning on an event that only involves a {\em small} part of the variables and then bounding the conditional expectation by its first moment. However, in this problem we do not know how to apply these two approaches as dealing with a random permutation that involves all $n$ coordinates seems of substantial and novel challenge. In contrast, our approach is based on \cite{DDL23+}, where the idea was to carefully analyze conditional expectations and use sophisticated cancellations under the conditioning. Still, even when $\epsilon=0$ (i.e., when there is no community signal) our result gives a sharper result compared to \cite{DDL23+} as we are able to rule out all polynomials with degree $n^{o(1)}$ but \cite{DDL23+} can only rule out polynomials with degree $e^{o(\sqrt{\log n})}$. In addition, compared to \cite{DDL23+}, this work provides an approach that we believe is more robust and overcomes several technical issues that arise in this specific setting (see Section~\ref{sec-contribution} for further discussions).

\subsection{Our contributions}
\label{sec-contribution}
While the hardness aspect of the present work can be viewed as a follow-up work of \cite{DDL23+}, we do think that we have overcome significant challenges in this particular setting, as we elaborate next. 

\

\noindent (i) In prior works, the failure of direct low-degree likelihood calculations is typically due to an event that occurs with vanishing probability; specifically, in \cite{DDL23+} this ``bad'' event is the emergence of graphs with atypical high edge density. However, in our setting the low-degree likelihood calculation blows up due to \emph{two} conceptually different events: one is the occurrence of dense subgraphs and the other is the occurrence of small cycles. Worse still, the later event occurs \emph{with positive probability}. A possible approach to address this challenge is to develop an analog of the small subgraph conditioning method for this context. To be more precise, we need to carefully count small cycles in the graph and account for their influence on the low-degree likelihood ratio. Consequently, rather than conditioning on a typical event with probability $1-o(1)$ (as in \cite{BAH+22, CGH+22, DMW23+, DDL23+}), we need to condition on an event {\em with positive probability}, which will make the calculation of conditional probabilities and expectations even more complicated.

\

\noindent (ii) Although it is tempting to directly work with the conditional measure discussed in (i), calculating the conditional expectation seems of substantial challenge. The techniques developed in \cite{DDL23+} rely on the {\em independence between edges} in the unconditioned model (in the parent graph). However, in our setting even in the parent graph the edges are correlated due to the latent community labeling, and the conditioning further affects the measure over this labeling. To address this, instead of working directly with the conditional measure, we need to work with a carefully designed measure that is {\em statistically indistinguishable} from the conditional measure, yet simplifies the computation of conditional expectations.

\

\noindent (iii) From a technical standpoint, the work sharpens several key estimates developed in \cite{DDL23+}. Specifically, the methods in \cite{DDL23+} involve several combinatorial estimates on enumerations of graphs with certain properties, where relatively coarse bounds sufficed due to the simplicity of the conditioned event. However, in this work, the event we condition on is more involved for aforementioned reasons. As a result, such enumeration estimates in this work become substantially more delicate, which presents a significant technical challenge in our proof. This refinement enables us to rule out all polynomials with degree $n^{o(1)}$, suggesting that any algorithm capable of breaking the threshold $\min\{ \sqrt{\alpha}, \frac{1}{\epsilon^2 \lambda} \}$ must have sub-exponential running time. See Section~\ref{sec:prelim-graphs} of the appendix for a more detailed discussion on how these estimates are handled.

\subsection{Notations}\label{sec-notation}
In this subsection, we record a list of notations that we shall use throughout the paper. Recall that $\Pb_n,\Qb_n$ are two probability measures on pairs of random graphs on $[n]=\{1,\ldots,n\}$. Denote $\mathfrak{S}_n$ the set of permutations over $[n]$ and denote $\mu$ the uniform distribution on $\mathfrak S_n$. We will use the following notation conventions on graphs.
\begin{itemize}
    \item {\em Labeled graphs}. Denote by $\mathcal{K}_n$ the complete graph with vertex set $[n]$ and edge set $\operatorname{U}_n$. For any graph $H$, let $V(H)$ denote the vertex set of $H$ and let $E(H)$ denote the edge set of $H$. We say $H$ is a subgraph of $G$, denoted by $H\subset G$, if $V(H) \subset V(G)$ and $E(H) \subset E(G)$. Define the excess of the graph $\tau(H)=|E(H)|-|V(H)|$.
    \item {\em Induced subgraphs}. For a graph $H=(V,E)$ and a subset $A \subset V$, define $H_A =(A,E_A)$ to be the vertex-induced subgraph of $H$ in $A$. Define $H_{\setminus A} = (V,E_{\setminus A})$ to be the subgraph of $H$ obtained by deleting all edges within $A$. Note that $E_A \cup E_{\setminus A} = E$. For a graph $G=G(V,E)$ and an edge set $E_0 \subset E$, define the edge-induced subgraph $G_0=(V_0,E_0)$, where $V_0$ is the collection of $v \in V$ such that $v$ is the endpoint of some $e_0 \in E_0$. 
    \item {\em Isolated vertices}. For $u \in V(H)$, we say $u$ is an isolated vertex of $H$, if there is no edge in $E(H)$ incident to $u$. Denote $\mathcal I(H)$ the set of isolated vertices of $H$. For two graphs $H,S$, we denote $H \ltimes S$ if $H \subset S$ and $\mathcal I(S) \subset \mathcal I(H)$, and we denote $H \Subset S$ if $H \subset S$ and $\mathcal I(H)=\emptyset$. For any graph $H \subset \mathcal K_n$, let $\widetilde H$ be the subgraph of $H$ induced by $V(H) \setminus \mathcal I(H)$. 
   \item {\em Graph intersections and unions}. For $H,S \subset \mathcal{K}_n$, denote by $H \cap S$ the graph with vertex set given by $V(H) \cap V(S)$ and edge set given by $E(H)\cap E(S)$. Denote by $S \cup H$ the graph with vertex set given by $V(H) \cup V(S)$ and edge set $E(H) \cup E(S)$. In addition, denote by $S\Cap H$, $S \doublesetminus H$ and $S\doublesymdiff H$ the graph induced by the edge set $E(S)\cap E(H)$, $E(S)\setminus E(H)$ and $ E(S)\triangle E(H)$, respectively (in particular, these induced graphs have no isolated points).
    \item {\em Paths.} We say a triple $P=(u,v,H)$ (where $u,v \in [n]$ and $H$ is a subgraph of $\mathcal K_n$) is a path with endpoints $u,v$ (possibly with $u=v$), if there exist distinct $w_1, \ldots, w_m \neq u,v$ such that $V(H)=\{ u,v,w_1,\ldots,w_m \}$ and $E(H)=\{ (u,w_1), (w_1,w_2) \ldots, (w_m,v) \}$. We say $P$ is a simple path if its endpoints $u \neq v$. We denote $\operatorname{EndP}(P)$ as the set of endpoints of a path $P$. Note that when $H$ is a cycle, for all $u \in V(H)$ we have $P_u=(u,u,H)$ is a path with endpoint $\{ u \}$.
    \item {\em Cycles and independent cycles.} We say a subgraph $H$ is an $m$-cycle if $V(H)=\{ v_1, \ldots, v_m \}$ and $E(H)=\{ (v_1,v_2), \ldots, (v_{m-1},v_m), (v_m,v_1) \}$. For a subgraph $K \subset H$, we say $K$ is an independent $m$-cycle of $H$, if $K$ is an $m$-cycle and no edge in $E(H)\setminus E(K)$ is incident to $V(K)$. Denote by $\mathtt{C}_m(H)$ the set of $m$-cycles of $H$ and denote by $\mathcal{C}_m(H)$ the set of independent $m$-cycles of $H$. For $H\subset S$, we define $\mathfrak{C}_{m}(S,H)$ to be the set of independent $m$-cycles in $S$ whose vertex set is disjoint from $V(H)$. Define $\mathfrak{C}(S,H)=\cup_{m\geq 3} \mathfrak{C}_m(S,H)$.
    \item {\em Leaves.} A vertex $u \in V(H)$ is called a leaf of $H$, if the degree of $u$ in $H$ is $1$; denote $\mathcal L(H)$ as the set of leaves of $H$. 
    \item {\em Graph isomorphisms and unlabeled graphs.} Two graphs $H$ and $H'$ are isomorphic, denoted by $H\cong H'$, if there exists a bijection $\pi:V(H) \to V(H')$ such that $(\pi(u),\pi(v)) \in E(H')$ if and only if $(u,v)\in E(H)$. Denote by $\mathcal H$ the isomorphism class of graphs; it is customary to refer to these isomorphic classes as unlabeled graphs. Let $\operatorname{Aut}(H)$ be the number of automorphisms of $H$ (graph isomorphisms to itself). 
\end{itemize}

For two real numbers $a$ and $b$, we let $a \vee b = \max \{ a,b \}$ and $a \wedge b = \min \{ a,b \}$. We use standard asymptotic notations: for two sequences $a_n$ and $b_n$ of positive numbers, we write $a_n = O(b_n)$, if $a_n<Cb_n$ for an absolute constant $C$ and for all $n$ (similarly we use the notation $O_h$ is the constant $C$ is not absolute but depends only on $h$); we write $a_n = \Omega(b_n)$, if $b_n = O(a_n)$; we write $a_n = \Theta(b_n)$, if $a_n =O(b_n)$ and $a_n = \Omega(b_n)$; we write $a_n = o(b_n)$ or $b_n = \omega(a_n)$, if $a_n/b_n \to 0$ as $n \to \infty$. In addition, we write $a_n \circeq b_n$ if $a_n = [1+o(1)] b_n$. For a set $\mathsf A$, we will use both $\# \mathsf A$ and $|\mathsf A|$ to denote its cardinality. For two probability measures $\mathbb P$ and $\mathbb Q$, we denote the total variation distance between them by $\operatorname{TV}(\mathbb P,\mathbb Q)$.

\subsection{Organization of this paper}
The rest of this paper is organized as follows. In Section~\ref{sec:LDP-framework} we rigorously state the low-degree framework for the detection problem under consideration. In Section~\ref{sec:detection-upper-bound} we propose an algorithm for detection and give a theoretical guarantee when $s>\min\{ \sqrt{\alpha},\frac{1}{\lambda\epsilon^2} \}$, which implies Part (i) of Theorem~\ref{MAIN-THM-detection}. In Section~\ref{sec:detection-lower-bound} we prove low-degree hardness for detection when $s<\min\{ \sqrt{\alpha}, \frac{1}{\lambda\epsilon^2} \}$, which implies Part (ii) of Theorem~\ref{MAIN-THM-detection}. Several technical results are postponed to the appendix to ensure a smooth flow of presentation.

\section{The low-degree polynomial framework}{\label{sec:LDP-framework}}

Inspired by the sum-of-squares hierarchy, the low-degree polynomial method offers a promising framework for establishing computational lower bounds in high-dimensional inference problems. This approach focuses primarily on analyzing algorithms that evaluate collections of polynomials with moderate degrees. The exploration of this category of algorithms is driven by research in high-dimensional hypothesis testing problems \cite{BHK+19, HS17, HKP+17, Hopkins18}, with an extensive overview provided in \cite{KWB22}. This low-degree framework has subsequently been extended to study random optimization and constraint satisfaction problems.

The approach of low-degree polynomials is appealing partly because it has yielded tight hardness results for a wide range of problems. Prominent examples include detection and recovery problems such as planted clique, planted dense subgraph, community detection, sparse-PCA and tensor-PCA (see \cite{HS17, HKP+17, Hopkins18, KWB22, SW22, DMW23+, BKW19, MW21+, DKW22, BAH+22, MW22, DMW23+, KMW24+}), optimization problems such as maximal independent sets in sparse random graphs \cite{GJW20, Wein22}, and constraint satisfaction problems such as random $k$-SAT \cite{BH22}. In the remaining of this paper, we will focus on applying this framework in the context of detection for correlated stochastic block models.
 
More precisely, to probe the computational threshold for testing between two sequences of probability measures $\Pb_n$ and $\Qb_n$, we focus on low-degree polynomial algorithms (see, e.g., \cite{Hopkins18, KWB22, DDL23+}). Let $\mathbb{R}[A,B]_{\leq D}$ denote the set of multivariate polynomials in the entries of $(A,B)$ with degree at most $D$. With a slight abuse of notation, we will often say ``a polynomial'' to mean a sequence of polynomials $f=f_n \in \mathbb{R}[A,B]_{\leq D}$, one for each problem size $n$; the degree $D=D_n$ of such a polynomial may scale with $n$. To study the power of a polynomial in testing $\Pb_n$ against $\Qb_n$, we consider the following notions of strong separation and weak separation defined in \cite[Definition~1.6]{BAH+22}.

\begin{defn}{\label{def-strong-separation}}
    Let $f \in \mathbb{R}[A,B]_{\leq D}$ be a polynomial.
    \begin{itemize}
        \item We say $f$ strongly separates $\Pb_n$ and $\Qb_n$ if as $n \to \infty$ 
        \[
        \sqrt{ \max\big\{ \operatorname{Var}_{\Pb_n}(f(A,B)), \operatorname{Var}_{\Qb_n}(f(A,B)) \big\} } = o\big( \big| \mathbb{E}_{\Pb_n}[f(A,B)] - \mathbb{E}_{\Qb_n}[f(A,B)] \big| \big) \,;
        \]
        \item We say $f$ weakly separates $\Pb_n$ and $\Qb_n$ if as $n \to \infty$ 
        \[
        \sqrt{ \max\big\{ \operatorname{Var}_{\Pb_n}(f(A,B)), \operatorname{Var}_{\Qb_n}(f(A,B)) \big\} } = O\big( \big| \mathbb{E}_{\Pb_n}[f(A,B)] - \mathbb{E}_{\Qb_n}[f(A,B)] \big| \big) \,.
        \]
    \end{itemize}
\end{defn}
See \cite{BAH+22} for a detailed discussion of why these conditions are natural for hypothesis testing. In particular, according to Chebyshev's inequality, strong separation implies that we can threshold $f(A,B)$ to test $\Pb_n$ against $\Qb_n$ with vanishing type I and type II errors (i.e., \eqref{eq-success} holds). Our first result confirms the existence of a low-degree polynomial that achieves strong separation in the ``easy'' region.

\begin{thm}{\label{main-thm-detection-upper-bound}}
    Suppose that we observe a pair of random graphs $(A,B)$ sampled from either $\Pb_n$ or $\Qb_n$ with $s>\sqrt{\alpha} \wedge \frac{1}{\lambda \epsilon^2}$. Then for any $\omega(1)\leq D_n\leq o\big(\frac{\log n}{\log \log n}\big)$ there exists a degree-$D_n$ polynomial that strongly separates $\Pb_n$ and $\Qb_n$. In addition, there exists an algorithm based on this polynomial that runs in time $n^{2+o(1)}$ and successfully distinguishes $\mathbb P_n$ from $\mathbb Q_n$ in the sense of \eqref{eq-success}.
\end{thm}

We now focus on the ``hard'' region and hope to give evidence on computational hardness for this problem. While it is perhaps most natural to provide evidence that no low-degree polynomial achieves strong separation for $\Pb_n$ and $\Qb_n$ in this region, this approach runs into several technical problems. In order to address this, we instead provide evidence on a modified testing problem, whose computational complexity is no more than that of the original problem. To this end, we first present a couple of lemmas as a preparation.
\begin{lemma}\label{lem-condition-positive-events}
    Assume that an algorithm $\mathcal A$ can distinguish two probability measures $\mathbb P_n$ and $\mathbb Q_n$ with probability $1-o(1)$ (i.e., in the sense of \eqref{eq-success}). Then for any positive constant $c>0$ and any sequence of events $\mathcal E_n$ such that $\mathbb{P}_n(\mathcal E_n) \geq c$, the algorithm $\mathcal A$ can distinguish $\mathbb P_n(\cdot \mid \mathcal E_n)$ and $\mathbb Q_n$ with probability $1-o(1)$.
\end{lemma}
\begin{proof}
    Suppose that we use the convention that $\mathcal A$ outputs 0 if it decides the sample is from $\mathbb Q_n$ and outputs 1 if it decides the sample is from $\mathbb P_n$. Then,
    \begin{equation}{\label{eq-translate-success}}
        \mathbb P_n\big( \mathcal A(\mbox{input} )=0 \big)=o(1), \quad \mathbb Q_n \big( \mathcal A(\mbox{input})=0 \big) = 1-o(1) \,.
    \end{equation}
    This shows that
    \[
    \mathbb P_n\big( \mathcal A(\mbox{input} )=0 \mid \mathcal E_n \big) \leq \frac{ \mathbb P_n( \mathcal A(\mbox{input} )=0 ) }{ \mathbb P_n(\mathcal E_n) } =o(1) \,, 
    \]
    which yields the desired result.
\end{proof}
\begin{lemma}\label{lem-replace-by-TVo1-measures}
    Assume that an algorithm $\mathcal A$ can distinguish two probability measures $\mathbb P_n$ and $\mathbb Q_n$ with probability $1-o(1)$ (i.e., in the sense of \eqref{eq-success}). Then for any sequence of probability measures $\mathbb P_n'$ such that $\operatorname{TV}(\mathbb P_n,\mathbb P_n')=o(1)$, the algorithm $\mathcal A$ can distinguish $\mathbb P_n'$ and $\mathbb Q_n$ with probability $1-o(1)$.
\end{lemma}
\begin{proof}
    By \eqref{eq-translate-success}, we have that 
    \[
    \Pb_n'\big( \mathcal A(\mbox{input} )=0 \big) \leq \mathbb P_n( \mathcal A(\mbox{input} )=0 ) + \operatorname{TV}(\Pb_n,\Pb'_n) =o(1) \,, 
    \]
    which yields the desired result.
\end{proof}

Now we can state our result in the ``hard'' region as follows.

\begin{thm}{\label{main-thm-detection-lower-bound}}
    Suppose that we observe a pair of random graphs $(A,B)$ sampled from either $\Pb_n$ or $\Qb_n$ with $s<\sqrt{\alpha} \wedge \frac{1}{\lambda \epsilon^2}$. Then there exists a sequence of events $\mathcal E_n$ and a sequence of probability measures $\mathbb P_n'$ such that the following hold:
   \begin{enumerate}
        \item[(1)] $\mathbb{P}_{n}(\mathcal E_n) \geq c$ for some constant $c=c(\lambda,k,\epsilon)>0$.
        \item[(2)] $\operatorname{TV}(\mathbb P_n(\cdot \mid \mathcal E_n), \mathbb{P}_n') \to 0$ as $n \to \infty$.
        \item[(3)] There is no degree-$n^{o(1)}$ polynomial that can strongly separate $\mathbb{P}_n'$ and $\mathbb Q_n$. 
    \end{enumerate}
\end{thm}
\begin{proof}[Proof of Theorem~\ref{MAIN-THM-detection}]
    Part (1) of Theorem~\ref{MAIN-THM-detection} follows from Theorem~\ref{main-thm-detection-upper-bound}; Part (2) of Theorem~\ref{MAIN-THM-detection} follows by combining Theorem~\ref{main-thm-detection-lower-bound} with Lemmas~\ref{lem-condition-positive-events} and \ref{lem-replace-by-TVo1-measures}, as we explain below. We emphasize that Theorem~\ref{main-thm-detection-lower-bound} does not rigorously prove that all degree-$n^{o(1)}$ polynomials fail to achieve strong separation for the \emph{original} testing problem between $\Pb_n$ and $\Qb_n$. However, Lemmas~\ref{lem-condition-positive-events} and \ref{lem-replace-by-TVo1-measures} imply that the detection problem between $\Pb'_n$ and $\Qb_n$ is \emph{not harder} than the detection problem between $\Pb_n$ and $\Qb_n$. Thus, as it is widely accepted that (see e.g., \cite[Section~3.3]{Wein25+}) the inability of degree-$D$ polynomial to achieve strong separation serves as a compelling evidence that no algorithm with running time $n^{D/\log n}$ achieves strong detection, our theorem serves as an evidence for the computational hardness of testing $\Pb'_n$ and $\Qb_n$ (and thus also serves as an evidence for the computational hardness of testing $\Pb_n$ against $\Qb_n$).
\end{proof}

In the subsequent sections of this paper, we will keep the values of $n,\lambda,k,\epsilon$ and $s$ fixed, and for the sake of simplicity we will omit subscripts involving these parameters without further specification. In particular, we will simply denote $\Pb_{*,n}, \Pb_n, \Qb_n, \operatorname{U}_n, \Pb'_n, \mathcal E_n$ as $\Pb_{*},\Pb, \Qb, \operatorname{U}, \Pb', \mathcal E$.

\section{Correlation detection via counting trees}{\label{sec:detection-upper-bound}}

In this section we prove Theorem~\ref{main-thm-detection-upper-bound}. From \cite{MNS15, HS17}, the results hold when $s> \frac{1}{\lambda\epsilon^2}$ since we can distinguish $\Pb$ and $\Qb$ by simply using one graph $A$. It remains to deal with the case $\frac{1}{\lambda \epsilon^2} \geq s>\sqrt{\alpha}$. As we shall see the following polynomials will play a vital role in our proof.
\begin{defn}{\label{def-phi-S1,S2}}
    For two graphs $S_1,S_2\subset \mathcal K_n$, define the polynomial $\phi_{S_1,S_2}$ associated with $S_1,S_2$ by 
    \begin{equation}{\label{eq-def-f-K1K2}}
        \phi_{S_1,S_2} \big( \{A_{i,j}\},\{B_{i,j}\} \big) =\big(\tfrac{\lambda s}{n}(1-\tfrac{\lambda s}{n})\big)^{-\frac{|E(S_1)|+|E(S_2)|}{2}} \prod_{(i,j)\in E(S_1)} \Bar{A}_{i,j} \prod_{(i,j)\in E(S_2)} \Bar{B}_{i,j} \,,
    \end{equation}
    where $\Bar{A}_{i,j}=A_{i,j}-\tfrac{\lambda s}{n}, \Bar{B}_{i,j}=B_{i,j}-\tfrac{\lambda s}{n}$ for all $(i,j) \in \operatorname{U}$. In particular, $\phi_{\emptyset,\emptyset}\equiv 1$.
\end{defn}
As implied by \cite{MWXY21+, DDL23+}, it is straightforward that $\{ \phi_{S_1,S_2} : S_1,S_2 \Subset \mathcal K_n \}$ is an orthonormal basis under the measure $\Qb$ in the sense that
\begin{equation}{\label{eq-standard-orthogonal}}
    \mathbb{E}_{\Qb}\big[ \phi_{S_1,S_2} \phi_{S_1',S_2'} \big]= \mathbf{1}_{ \{ (S_1,S_2)=(S_1',S_2') \} } \,.
\end{equation}
Next, denote by $\mathcal{T}=\mathcal T_{\aleph_n}$ the set of all unlabeled trees with $\aleph=\aleph_n$ edges, where
\begin{equation}{\label{eq-choice-K}}
    \omega (1) \leq \aleph_n \leq o\Big( \frac{\log n}{\log \log n} \Big) \,.
\end{equation}
It was known in \cite{Otter48} that 
\begin{equation}{\label{eq-num-trees}}
    \lim_{n \longrightarrow \infty} |\mathcal{T}_{\aleph_n}|^{\frac{1}{\aleph_n}} = \tfrac{1}{\alpha} \,,
\end{equation}
where we recall that $\alpha \approx 0.338$ is the Otter's constant. Define
\begin{equation}{\label{eq-def-f-T}}
    f_{\mathcal{T}}(A,B) = \sum_{\mathbf H \in \mathcal{T}} \frac{ s^\aleph \operatorname{Aut}(\mathbf H) (n-\aleph-1)!}{n!} \sum_{ S_1,S_2 \cong \mathbf H } \phi_{S_1,S_2} (A,B) \,.
\end{equation}
Recall Definition~\ref{def-phi-S1,S2}. Observe that if we drop the centering in \eqref{eq-def-f-K1K2} (i.e., if we replace $\Bar{A}_{i,j}, \Bar{B}_{i,j}$ with $A_{i,j},B_{i,j}$ in \eqref{eq-def-f-K1K2}), then each summation item in \eqref{eq-def-f-T} is simply the product of the number of copies of $\mathbf H$ in the two graphs (module a constant factor), and thus \eqref{eq-def-f-T} can be viewed as counting trees in the ``centered'' graphs. This statistic was first introduced by \cite{MWXY21+}. We will show that strong separation is possible via tree counting under the assumption $\tfrac{1}{\lambda \epsilon^2} \geq s > \sqrt{\alpha}$, as incorporated in the following proposition.

\begin{proposition}{\label{prop-first-second-moment-f-T}}
    Assume that $\tfrac{1}{\lambda \epsilon^2} \geq s > \sqrt{\alpha}$. We have the following results:
    \begin{enumerate}
        \item[(1)] $\frac{ \operatorname{Var}_{\Qb}[f_{\mathcal{T}}] }{ (\mathbb{E}_{\Pb}[f_{\mathcal{T}}])^2 } = o(1)$ and $\mathbb E_{\Qb}[f_{\mathcal T}]=0$;
        \item[(2)] $\frac{ \operatorname{Var}_{\Pb}[f_{\mathcal{T}}] }{ (\mathbb{E}_{\Pb}[f_{\mathcal{T}}])^2 } = o(1)$.
    \end{enumerate}
    Thus, $f_{\mathcal T}$ strongly separates $\Pb$ and $\Qb$.
\end{proposition}

\begin{remark}{\label{rmk-algorithm}}
    As discussed in Definition~\ref{def-strong-separation}, Proposition~\ref{prop-first-second-moment-f-T} implies that the testing error satisfies
    $$ 
    \Qb( f_{\mathcal{T}}(A,B)\ge \tau ) + \Pb( f_{\mathcal{T}}(A,B)\le \tau )= o(1) \,, 
    $$
    where the threshold $\tau$ is chosen as $\tau = C\mathbb{E}_{\Pb} [f_{\mathcal{T}}(A,B)]$ for any fixed constant $0<C<1$. In addition, the statistics $f_{\mathcal T}$ can be approximated in $n^{2+o(1)}$ time by color coding, as incorporated in \cite[Algorithm~1]{MWXY21+}. We omit further details here since the proof of the validity of this approximation algorithm remains basically unchanged.
\end{remark}

The rest of this section is devoted to the proof of Proposition~\ref{prop-first-second-moment-f-T} (which then yields Theorem~\ref{main-thm-detection-upper-bound} in light of Remark~\ref{rmk-algorithm}). Our proof extends the methodology introduced in \cite{MWXY21+} and a key technical challenge arises from the additional estimation errors inherent in our setting. Specifically, in the relevant parameter regimes, the latent community partition cannot be exactly recovered. As a result, the edge-indicator variables in the centered subgraph counts cannot be precisely centered, necessitating careful handling of these errors in our analysis.

\subsection{Estimation of the first moment}{\label{subsec:est--1st-moment}}

In this section, we will provide a uniform bound on $\mathbb{E}_{\Pb} [\phi_{S_1,S_2}]$, which will lead to the proof of Item (1) in Proposition~\ref{prop-first-second-moment-f-T}. For $\mathbf H \in \mathcal T$, for notational convenience we define
\begin{equation}{\label{eq-def-a-H}}
    a_{\mathbf H}= \frac{ s^\aleph \operatorname{Aut}(\mathbf H) (n-\aleph-1)! }{ n! }\,.
\end{equation}

\begin{lemma}{\label{prop_first_moment_phi}}
We have uniformly for all $S_1,S_2 \cong \mathbf H \in \mathcal T$
\begin{align}
    \mathbb{E}_{\Pb}[\phi_{S_1,S_2}] & \circeq s^\aleph \cdot \mathbb{P}(\pi_*(S_1)=S_2) = a_{\mathbf H} \,. \label{eq-first-moment-phi-S}
\end{align}
\end{lemma}
The proof of Lemma~\ref{prop_first_moment_phi} is incorporated in Section~\ref{subsec:Proof-Lem-3.4} of the appendix. Now we estimate $\operatorname{Var}_{\Qb} [f_{\mathcal{T}}]$ and $\mathbb{E}_{\Pb} [f_{\mathcal{T}}]$ assuming Lemma~\ref{prop_first_moment_phi}. 

\begin{lemma}{\label{lem-first-moment-P}}
We have the following estimates:
\begin{enumerate}
    \item[(i)] $\mathbb E_{\Qb}[f_{\mathcal T}]=0$ and $\operatorname{Var}_{\Qb} [f_{\mathcal{T}}] = s^{2\aleph} |\mathcal{T}|$;
    \item[(ii)] $\mathbb{E}_{\Pb} [f_{\mathcal{T}}] \circeq s^{2\aleph} |\mathcal{T}|$.
\end{enumerate}
\end{lemma}

\begin{proof}
For Item (i), clearly we have $\mathbb E_{\Qb}[f_{\mathcal T}]=0$. Recalling \eqref{eq-standard-orthogonal} and \eqref{eq-def-f-T}, we have
\begin{align*}
    \operatorname{Var}_{\Qb} [f_{\mathcal{T}}] 
    & \overset{\eqref{eq-standard-orthogonal},\eqref{eq-def-f-T},\eqref{eq-def-a-H}}{=} \sum_{\mathbf H \in \mathcal{T}} \sum_{S_1,S_2 \cong \mathbf H} a_{\mathbf H}^2  \\
    &= \sum_{\mathbf H \in \mathcal{T}} a_{\mathbf H}^2 \cdot \#\{ S \subset \mathcal K_n : S \cong \mathbf H \}^2 = \sum_{\mathbf H \in \mathcal{T}} s^{2\aleph} = s^{2\aleph} |\mathcal{T}|  \,.
\end{align*}
As for Item (ii), by applying Lemma~\ref{prop_first_moment_phi} we have
\begin{equation*}
    \mathbb{E}_{\Pb} [f_{\mathcal{T}}] \circeq \sum_{\mathbf H \in \mathcal{T}} \sum_{S_1,S_2 \cong \mathbf H} a_{\mathbf H}^2 = s^{2\aleph} |\mathcal{T}| \,. \qedhere
\end{equation*}
\end{proof}
Recall our assumption that $s>\sqrt{\alpha}$ and \eqref{eq-num-trees}. By Lemma~\ref{lem-first-moment-P}, we have shown that
\begin{align*}
    \frac{\operatorname{Var}_{\Qb}[f_{\mathcal{T}}]}{(\mathbb{E}_{\Pb}[f_{\mathcal{T}}])^2} = o(1) \,.
\end{align*}

\subsection{Estimation of the second moment}{\label{subsec:est-2nd-moment}}

The purpose of this subsection is to show Item~(ii) of Proposition~\ref{prop-first-second-moment-f-T}. Recall \eqref{eq-def-f-T} and \eqref{eq-def-a-H}. A direct computation shows that 
\begin{align*}
    \operatorname{Var}_{\Pb} [f_{\mathcal{T}}] 
    &= \sum_{\mathbf H,\mathbf I \in \mathcal{T}} \sum_{  S_1,S_2 \cong \mathbf H } \sum_{T_1,T_2 \cong \mathbf I} a_{\mathbf H} a_{\mathbf I} \big( \mathbb{E}_{\Pb} [\phi_{S_1,S_2} \phi_{T_1,T_2}] - \mathbb{E}_{\Pb} [\phi_{S_1,S_2}] \mathbb{E}_{\Pb} [\phi_{T_1,T_2}] \big) \,.
\end{align*}
Now we estimate $\mathbb{E}_{\Pb} [\phi_{S_1,S_2} \phi_{T_1,T_2}]$, where $S_1,S_2 \cong \mathbf H$ and $T_1,T_2 \cong \mathbf I$. For $\mathbf H,\mathbf I \in \mathcal{T}$ (note that $\mathcal I(\mathbf H)=\mathcal I(\mathbf I)=\emptyset$), define 
\begin{equation}{\label{eq-def-R-H,I}}
    \mathsf R_{\mathbf H,\mathbf I} = \big\{ (S_1,S_2;T_1,T_2): S_1,S_2 \cong \mathbf H, T_1,T_2 \cong \mathbf I \big\} 
\end{equation}
and define the set of its ``principal elements''
\begin{equation}{\label{eq-def-R-H,I^*}}
    \mathsf R_{\mathbf H,\mathbf I}^* = \big\{ (S_1,S_2;T_1,T_2) \in \mathsf R_{\mathbf H,\mathbf I}: V(S_1) \cap V(T_1) = V(S_2) \cap V(T_2) = \emptyset \big\} \,.
\end{equation}

\begin{lemma} {\label{prop_principal}}
    (i) For all $(S_1,S_2;T_1,T_2) \in \mathsf R_{\mathbf H,\mathbf I} \setminus \mathsf R^*_{\mathbf H,\mathbf I}$ and for all $h>1$ we have
    \begin{align*}
        \big| \mathbb{E}_\Pb[ \phi_{S_1,S_2}\phi_{T_1,T_2} ] \big| & \le O_h(1) \cdot h^{2\aleph} \mathbf{1}_{ \{(S_1,S_2)=(T_1,T_2)\} }  \\
        & + [1+o(1)] \cdot n^{-0.5(|V(S_1) \triangle V(T_1)| + |V(S_2) \triangle V(T_2)|) - 0.8}  \,.
    \end{align*}   
    
    (ii) For $(S_1,S_2;T_1,T_2) \in \mathsf R^*_{\mathbf H,\mathbf I}$, we have
    \begin{align*}
        \mathbb{E}_{\Pb}[ \phi_{S_1,S_2} \phi_{T_1,T_2} ] \circeq \mathbb{E}_{\Pb}[ \phi_{S_1,S_2} ] \mathbb{E}_{\Pb}[ \phi_{T_1,T_2} ] \big( 1+\mathbf{1}_{\{\mathbf{H} \cong \mathbf{I}\}} \big) \,.
    \end{align*}
\end{lemma}
The proof of Lemma~\ref{prop_principal} is incorporated in Section~\ref{subsec:Proof-Lem-3.6} of the appendix. We use Lemma~\ref{prop_principal} to derive $\operatorname{Var}_{\Pb} [f_{\mathcal{T}}] = o(1) \cdot \mathbb{E}_{\Pb} [f_{\mathcal{T}}]^2$ in Section~\ref{subsec:proof-lem-3.2(ii)} of the appendix, which yields Item (ii) of Proposition~\ref{prop-first-second-moment-f-T}.

\section{Low-degree hardness for the detection problem}{\label{sec:detection-lower-bound}}

In this section we prove Theorem~\ref{main-thm-detection-lower-bound}. Throughout this section, we fix a small constant $\delta \in (0,0.1)$ and assume that 
\[
    s \leq \sqrt{\alpha} - \delta \mbox{ and } \epsilon^2 \lambda s \leq 1-\delta \,.
\]
We also choose a sufficiently large constant $N=N(k,\lambda,\delta,\epsilon,s) \geq 2 / \delta$ such that 
\begin{equation}{\label{eq-def-N}}
    \begin{aligned}
        & (\sqrt{\alpha}-\delta) (1+\epsilon^{N}k) \leq \sqrt{\alpha} - \delta/2 \,; \quad 10k(1-\delta)^N \leq (1-\delta/2)^{N} \,; \\
        & (\sqrt{\alpha}-\delta/4)(1+(1-\delta/2)^N)^2 \leq \sqrt{\alpha}-\delta/8 \,; \quad (1-\delta / 2)^N (N+1) \leq 1\,. 
    \end{aligned}
\end{equation}
Furthermore, we fix a sequence $D_n$ such that $\log D/\log n\to 0$ as $n\to \infty$. Without loss of generality, we assume $D_n \geq 2\log_2 n$ in the following proof. For the sake of brevity, we will only work with some fixed $n$ throughout the analysis, and we simply denote $D_n$ as $D$. While our main interest is to analyze the behavior for sufficiently large $n$, most of our arguments hold for all $n$, and we will explicitly point out in lemma-statements and proofs when we need the assumption that $n$ is sufficiently large.

\subsection{Truncation on admissible graphs} \label{subsec-admissible}

As previously suggested, it is crucial to work with a suitably truncated version of $\Pb$ rather than $\Pb$ itself. It turns out that an appropriate truncation is to control both the edge densities of subgraphs and the number of small cycles in the parent graph $G$, as explained in the following definition.

\begin{defn}\label{def-addmisible}
    Denote $\Tilde{\lambda}=\lambda\vee 1$. Given a graph $H=H(V,E)$, define 
    \begin{equation}\label{eq-def-Phi}
        \Phi(H) = \Big( \frac{2 \Tilde{\lambda}^2 k^2 n}{D^{50}} \Big)^{|V(H)|} \Big( \frac{ 1000 \Tilde{\lambda}^{20} k^{20} D^{50} }{ n } \Big)^{|E(H)|}  \,.
    \end{equation}
    Then we say the graph $H$ is \emph{bad} if $\Phi(H) < (\log n)^{-1}$, and we say a graph $H$ is \emph{self-bad} if $H$ is bad and $\Phi(H)<\Phi(K)$ for all $K \subset H$. Furthermore, we say that a graph $H$ is \emph{admissible} if it contains no bad subgraph and $\mathtt C_j(H) =\emptyset$ for $j \leq N$; we say $H$ is \emph{inadmissible} otherwise. Denote $\mathcal E = \mathcal E^{(1)} \cap \mathcal E^{(2)}$, where 
    $\mathcal E^{(1)}$ is the event that $G$ does not contain any bad subgraph with no more than $D^3$ vertices, and $\mathcal E^{(2)}$ is the event that $G$ does not contain any cycles with length at most $N$. 
\end{defn}

\begin{remark}
    We now provide a brief explanation for this rather involved definition of ``bad'' graphs. Roughly speaking, there are two possible reasons for a graph $H$ to be bad: one is that $H$ is atypically ``dense'' (i.e., $\Phi(H)$ is atypically small) and the other is that $H$ contains a small cycle (i.e., $\mathtt C_j(H) \neq \emptyset$). The motivation of ruling out all atypically dense graphs has already appeared in \cite{DDL23+}, where the authors chose a similar $\Phi(H)$. Roughly speaking, we expect that any subgraph of a sparse SBM graph with size no more than $n^{o(1)}$ has edge-to-vertex ratio $1+o(1)$. In the definition of $\Phi$, the term $\big( \tfrac{2 \Tilde{\lambda}^2 k^2 n}{D^{50}} \big)$ should be interpreted as $n^{1+o(1)}$, and $\big( \tfrac{ 1000 \Tilde{\lambda}^{20} k^{20} D^{50} }{ n } \big)$ as $n^{1-o(1)}$; the $o(1)$ terms are carefully tuned so that for a typical subgraph $H$ of a sparse SBM $\Phi(H)$ is much larger than $1$. In contrast, the need to rule out the influence of small cycles is a new challenge in the SBM setting, and is one of the main conceptual innovations in our work (as we have explained in Section~\ref{sec-contribution}). To see why ruling out the influence of small cycles is necessary, let us consider a simple polynomial: let $\mathbf H$ be the unlabeled graph that contains $\ell$ independent triangles and define 
    \begin{align*}
        f = \frac{\operatorname{Aut}(\mathbf H)}{n^{3\ell}} \sum_{H_1,H_2 \cong \mathbf H} \prod_{(i,j)\in E(H_1)}\frac{(A_{i,j}-\tfrac{\lambda s}{n})}{\sqrt{\tfrac{\lambda s}{n}(1-\tfrac{\lambda s}{n})}}\prod_{(i,j)\in E(H_2)}\frac{(B_{i,j}-\tfrac{\lambda s}{n})}{\sqrt{\tfrac{\lambda s}{n}(1-\tfrac{\lambda s}{n})}}   \,.
    \end{align*}
    A standard calculation yields that
    \begin{align*}
        \mathbb E_{\Qb}[f]=0, \quad \mathbb E_{\Qb}[f^2]=1+o(1) \mbox{ and } \mathbb E_{\Pb}[f] \geq \big( (1+(k-1)\epsilon^3) s^3 \big)^{\ell} \,,
    \end{align*}
    which will blow up for large $\ell$ if $(1+(k-1)\epsilon^3)s^3>1$. To tackle this issue, we choose to condition on the event that there is no small cycle in a sparse SBM (which happens with positive probability). Under this event, since $\mathbf H$ does not occur in either $A$ or $B$ we expect that the expectation of $f$ under $\Pb$ can be bounded.  
\end{remark}

\begin{lemma}\label{lem-Gc-is-typical}
    There exists a constant $c=c(\lambda,k,N)$ in $(0,1)$ such that for any permutation $\pi\in \mathfrak{S}_n$, it holds that $\Pb_*(\mathcal E \mid \pi_*=\pi) \geq c$. Therefore, we have $\Pb_*(\mathcal E)\geq c$.
\end{lemma}

The proof of Lemma~\ref{lem-Gc-is-typical} is incorporated in Section~\ref{subsec:proof-lem-4.2} of the appendix. We now introduce another key conceptual innovation of our work, i.e., the construction of a suitable measure $\Pb'$ that is statistically indistinguishable with $\Pb(\cdot \mid \mathcal E)$ (as mentioned in Section~\ref{sec-contribution}). Roughly speaking, the main technical challenge arises from the fact that under $\Pb(\cdot\mid \mathcal E)$ the marginal distribution of the community label $\sigma_*$ is rather complicated, as such conditioning will ``prefer'' the labeling with balanced community size. Our strategy is to construct a measure $\Pb'$ that on the one hand has $o(1)$ total variational distance to $\Pb(\cdot \mid \mathcal E)$, and on the other hand the community labels $\sigma_*(i)$'s still have some independence (unless $i$ belongs to some bad subgraphs).
We now show how to construct $\mathbb{P}_n'$ that satisfies Item (ii) of Theorem~\ref{main-thm-detection-lower-bound}. Recall the definition of ``good event'' $\mathcal E$ in Definition~\ref{def-addmisible}. 
\begin{defn}\label{def-G'-P'}
    List all self-bad subgraphs of $\mathcal K_n$ with at most $D^3$ vertices and all cycles of $\mathcal K_n$ with lengths at most $N$ in an arbitrary but prefixed order $(B_1,\ldots,B_{\mathtt M})$. Define a stochastic block model with ``bad graphs" removed as follows: (i) sample $G \sim \mathcal S(n,\tfrac{\lambda}{n};k,\epsilon)$; (ii) for each $\mathtt 1 \leq \mathtt i \leq \mathtt M$ such that $B_{\mathtt i} \subset G$, we independently uniformly remove one edge in $B_{\mathtt i}$. The unremoved edges in $G$ constitute a graph $G'$, which is the output of our modified stochastic block model. Clearly, from this definition $G'$ does not contain any cycle of length at most $N$ nor any bad subgraph with at most $D^3$ vertices. Conditioned on $G'$ and $\pi_*$, we define 
    \[
    A'_{i,j} = G'_{i,j}J'_{i,j}, B'_{i,j} = G'_{\pi_*^{-1}(i),\pi_*^{-1}(j)} K'_{i,j} \,,
    \]
    where $J'$ and $K'$ are independent Bernoulli variables with parameter $s$. Let $\mathbb P_*' = \mathbb P'_{*,n}$ be the law of $(\sigma_*,\pi_*,G,G',A',B')$ and denote $\Pb'=\Pb_n'$ the marginal law of $(A',B')$.
\end{defn}
We remark that under $\Pb'$ there is no bad subgraph in the parent graph $G'$ (and thus there is no bad subgraph in $A'$ or $B'$), for the following reason: if $G'$ contains a bad (sub)graph, it must contain a self-bad graph $B_{\mathtt i}$ (for example, we can simply consider the bad graphs in $G'$ with minimal edges); however, in Step (ii) we have removed at least one edge in $B_{\mathtt i}$ and thus there is a contradiction. The next lemma shows that $\Pb'$ has $o(1)$ total variational distance to $\Pb(\cdot\mid\mathcal E)$.

\begin{lemma}{\label{lem-TV-Pb-Pb'}}
    We have $\operatorname{TV}( \Pb', \Pb(\cdot \mid \mathcal E))=o(1)$.
\end{lemma}
The proof of Lemma~\ref{lem-TV-Pb-Pb'} is incorporated in Section~\ref{subsec:proof-lem-4.4} of the appendix.

\subsection{Reduction to admissible polynomials} \label{subsec-reduction-to-P'}

The goal of this and the next subsection is to prove Item (iii) in Theorem~\ref{main-thm-detection-lower-bound}, i.e., there is no degree-$D$ polynomial that can strongly separate $\Pb'$ and $\Qb$. As implied by \cite{BAH+22}, it suffices to show 
\begin{equation}{\label{eq-goal-complexity-lower-bound}}
    \| L'_{\leq D} \|^2 := \sup_{ \operatorname{deg}(f) \leq D } \Bigg\{ \frac{ \mathbb E_{\Pb'}[f] }{ \sqrt{\mathbb E_{\Qb}[f^2]} } \Bigg\} \leq O_{\delta,N}(1) \,.
\end{equation}
Recall Definition~\ref{def-phi-S1,S2}. Denote \(\mathcal P_{D}\) as the set of real polynomials on \(\{0,1\}^{2|\!\operatorname{U}\!|}\) with degree no more than $D$, and recall from \eqref{eq-standard-orthogonal} that 
\begin{equation}\label{eq-def-Od}
    \mathcal O_D=\{\phi_{S_1,S_2}:S_1,S_2 \Subset \mathcal{K}_n, |E(S_1)|+|E(S_2)| \leq D\}
\end{equation}
is an orthonormal basis for \(\mathcal P_{D}\) (under the measure $\mathbb Q$). Now we say a polynomial \(\phi_{S_1,S_2}\in \mathcal O_D\) is \emph{admissible} if both \(S_1\) and \(S_2\) are admissible graphs. Furthermore, we define $\mathcal O_D'\subset \mathcal O_D$ as the set of admissible polynomials in $\mathcal O_D$, and define \(\mathcal P_{D}' \subset \mathcal P_{D}\) as the linear subspace spanned by polynomials in $\mathcal O_D'$. 

Intuitively, due to the absence of inadmissible graphs under the law \(\Pb'\), only admissible polynomials are relevant in polynomial-based algorithms. Therefore, it is plausible to establish our results by restricting to polynomials in $\mathcal P_{D}'$. The following proposition as well as its consequence as in \eqref{eq-reduction-1} formalizes this intuition.

\begin{proposition}\label{prop-same-L^1-bounded-L^2}
     The following holds for some absolute constant $c$. For any $f\in \mathcal P_{D}$, there exists some $f'\in \mathcal P_{D}'$ such that $\mathbb E_{\Qb}[(f')^2]\le c\cdot\mathbb{E}_{\Qb}[f^2]$ and $f'=f$ a.s. under both $\Pb_*'$ and $\Pb'$. 
\end{proposition}
Provided with Proposition \ref{prop-same-L^1-bounded-L^2}, we immediately get that
\begin{equation}\label{eq-reduction-1}
    \sup_{f\in \mathcal P_{D}} \Bigg\{ \frac{\mathbb E_{\Pb'}[f]}{\sqrt{\mathbb E_{\Qb}[f^2]}} \Bigg\} \le O(1)\cdot \sup_{f\in \mathcal P_{D}'} \Bigg\{ \frac{\mathbb E_{\Pb'}[f]}{\sqrt{\mathbb E_{\Qb}[f^2]}} \Bigg\} \,.
\end{equation}
Thus, we successfully reduce the optimization problem over $\mathcal P_{D}$ to that over $\mathcal P_{D}'$ (up to a multiplicative constant factor, which is not material).

Now we turn to the proof of Proposition~\ref{prop-same-L^1-bounded-L^2}. For variables $X\in \{A,B\}$ (meaning that $X_{i,j}=A_{i,j}$ or $X_{i,j}=B_{i,j}$ for all $(i,j)\in \operatorname{U}$), denote for each subgraph $S$ that 
\begin{equation}\label{eq-def-f_S}
    \psi_{S}(\{X_{i,j}\}_{(i,j)\in \operatorname{U}})
    =\prod_{(i,j)\in E(S)} \frac{(X_{i,j}-\tfrac{\lambda s}{n})}{\sqrt{\tfrac{\lambda s}{n}(1-\tfrac{\lambda s}{n})}}\,.
\end{equation}
Recalling the definition of $\phi_{S_1, S_2}$, we can write it as follows:
\begin{equation}\label{eq-decomposition-of-phi-into-psi's}
    \phi_{S_1,S_2}(A,B)=\prod_{(i,j)\in E(S_1)}\frac{(A_{i,j}-\tfrac{\lambda s}{n})}{\sqrt{\tfrac{\lambda s}{n}(1-\tfrac{\lambda s}{n})}}\prod_{(i,j)\in E(S_2)}\frac{(B_{i,j}-\tfrac{\lambda s}{n})}{\sqrt{\tfrac{\lambda s}{n}(1-\tfrac{\lambda s}{n})}}=\psi_{S_1}(A)\psi_{S_2}(B) \,.
\end{equation}
In light of this, we next analyze the polynomial $\psi_S(X)$ (with $S \Subset \mathcal K_n$) via the following expansion:
\begin{equation*}
    \psi_S(X)=\sum_{K\Subset S} \left(-\frac{\sqrt{\lambda s/n}}{\sqrt{1-\lambda s/n}}\right)^{|E(S)|-|E(K)|}\prod_{(i,j)\in E(K)}\frac{X_{i,j}}{\sqrt{\tfrac{\lambda s}{n}(1-\tfrac{\lambda s}{n})}}\,,
\end{equation*}
where the summation is taken over all subgraphs of $S$ without isolated vertices (there are $2^{|E(S)|}$ many of them). We define the ``inadmissible-part-removed'' version of $\psi_S(X)$ by 
\begin{equation}\label{eq-def-f_S'}
    \Hat{\psi}_{S}(X)=\sum_{\substack {K\Subset S\\K\text{ is admissible}}}\left(-\frac{\sqrt{\lambda s/n}}{\sqrt{1-\lambda s/n}}\right)^{|E(S)|-|E(K)|}\prod_{(i,j)\in E(K)}\frac{X_{i,j}}{\sqrt{\tfrac{\lambda s}{n}(1-\tfrac{\lambda s}{n})}}\,,
\end{equation}
and obviously we have that $\psi_S(A)-\Hat{\psi}_S(A)=\psi_S(B)-\Hat{\psi}_S(B)=0$ a.s. under both $\Pb_*'$ and $\Pb'$. Although it is temping and natural to use the preceding reduction, in the actual proof later we need to employ some further structure, for which we introduce the following definitions.

\begin{definition}
For $S\Subset \mathcal K_n$, denote $\mathsf{Cycle}(S)=\cup_{j=3}^N \cup_{C\in \mathtt C_j(S)} C$. Define 
\begin{equation}{\label{eq-def-D(S)}} 
    \mathsf D(S)= \begin{cases}
        \emptyset, & \text{if }S \mbox{ is admissible}\,, \\
        \arg \max_{\mathsf{Cycle}(S) \subset H \subset S} \{ \Phi(H) \}, &\text{if } S \mbox{ is inadmissible}\,,
    \end{cases}
\end{equation}
(if there are multiple choices of $\mathsf D(S)$ we choose $\mathsf D(S)$ that minimize $|V(\mathsf D(S))|$). We also define 
\begin{equation}{\label{eq-def-A(S)}}
    \mathcal A(S)=\{ H \Subset S: S \setminus \mathsf D(S) \subset H, H \cap \mathsf D(S) \text{ is admissible} \} \,.
\end{equation}
We also define the polynomial (recall \eqref{eq-def-f_S} and \eqref{eq-def-f_S'})
\begin{equation}\label{eq-def-hat-f_S}
    \psi'_S(\{X_{i,j}\}_{(i,j)\in \operatorname{U}})=\psi_{S\setminus \mathsf D(S)}(\{X_{i,j}\}_{(i,j)\in \operatorname{U}})\cdot \hat{\psi}_{\mathsf D(S)}(\{X_{i,j}\}_{(i,j)\in \operatorname{U}})\,.
\end{equation}
Moreover, we define 
\begin{equation}\label{eq-def-hat-f-S1S2}
    \phi'_{S_1,S_2}(A,B) = \psi'_{S_1}(A) \psi'_{S_2}(B)\,, \forall S_1,S_2\subset \mathcal K_n\,.
\end{equation}
Then it holds that $\phi'_{S_1,S_2}(A,B)=\phi_{S_1,S_2}(A,B)$ a.s. under both $\Pb_*'$ and $\Pb'$.
\end{definition}

\begin{lemma}{\label{lem-D(S)-property}}
    For all inadmissible graph $S \Subset \mathcal{K}_n$ and all $H\in \mathcal A(S)$, it holds that $H$ itself is admissible and $\Phi(H) \geq \Phi(S)$. Furthermore, every $\psi'_S$ is a linear combination of $\{ \psi_H: H \in \mathcal{A}(S)\}$. As a result, $\phi'_{S_1,S_2}\in \mathcal P_{D}'$ for all $S_1,S_2\Subset \mathcal{K}_n\text{ with } |E(S_1)|+|E(S_2)|\le D$.
\end{lemma}
\begin{proof}
The proof of Lemma~\ref{lem-D(S)-property} is essentially identical to \cite[Lemma~3.6]{DDL23+}, where the crucial input is that $\Phi(H)$ is a log-submodular function, i.e., we have $\Phi(H \cup S) \Phi(H \Cap S) \leq \Phi(H) \Phi(S)$ (see Item (ii) of Lemma~\ref{lemma-facts-graphs} of the appendix). We omit further details here due to the high similarity.
\end{proof}

\begin{lemma}{\label{lem-characterization-D(S)}}
    For all $H \in \mathcal A(S)$, we have $\mathcal L(S) \subset V(H)$ and $\mathcal C_j(S) \subset H$ for $j>N$.
\end{lemma}
\begin{proof}
    Note that it suffices to show that $\mathcal L(S) \cap V(\mathsf D(S))=\emptyset$ and $V(\mathcal{C}_j(S)) \cap V(\mathsf D(S))=\emptyset$ for $j>N$. Suppose on the contrary that $u \in \mathcal L(S) \cap V(\mathsf D(S))$. Then we can define $\mathsf D'(S)$ as the subgraph of $\mathsf D(S)$ induced by $V(\mathsf D(S)) \setminus \{ u \}$. Clearly we have $|V(\mathsf D'(S))| = |V(\mathsf D(S))|-1$ and $|E(\mathsf D'(S))| \geq |E(\mathsf D(S))|-1$, which yields that $\Phi(\mathsf D'(S)) < \Phi(\mathsf D(S))$, contradicting with \eqref{eq-def-D(S)}. This shows that $\mathcal L(S) \cap V(\mathsf D(S))=\emptyset$. We can prove $V(\mathcal{C}_j(S)) \cap V(\mathsf D(S))=\emptyset$ similarly (by considering the subgraph induced by $V(\mathsf D(S))\setminus V(\mathcal{C}_j(S))$). 
\end{proof}
We now elaborate on the polynomials $\psi'_S(X)$ more carefully. Write
\begin{equation}{\label{eq-def-Lambda-H,S}}  
    \psi'_S(X) = \sum_{ H \in \mathcal A(S)} \Lambda_S(H) \psi_{H}(X) \,,
\end{equation}
where (same as \cite[Equation~(3.13)]{DDL23+})
\begin{equation}{\label{eq-def-hatpsi-Lambda-H,S}}  
    \Lambda_S(H) = \Bigg( \frac{\sqrt{\lambda/n}}{\sqrt{1-\lambda/n}}\Bigg)^{|E(S)|-|E(H)|} \sum_{ J:J \in \mathcal A(S), H \Subset J } (-1)^{|E(S)|-|E(J)|} \,.
\end{equation}
Similar to \cite[Equation~(3.14)]{DDL23+}, we can show that
\begin{equation}{\label{eq-est-of-Lambda}}
    |\Lambda_S(H)| \leq (4\sqrt{\lambda/n})^{ |E(S)|-|E(H)| } \,.
\end{equation}
With these estimates in hand, we are now ready to prove Proposition~\ref{prop-same-L^1-bounded-L^2}. 

\begin{proof}[Proof of Proposition~\ref{prop-same-L^1-bounded-L^2}]
For any $f\in \mathcal P_{D}$, we can write 
\begin{align*}
f=\sum_{\phi_{S_1,S_2} \in \mathcal O_D} C_{S_1,S_2} \phi_{S_1,S_2}
\end{align*}
since $\mathcal O_D$ is an orthonormal basis for $\mathcal P_{D}$ (as we mentioned at the beginning of this subsection) and we define $ f'=\sum_{\phi_{S_1,S_2}\in \mathcal O_D} C_{S_1,S_2} \phi'_{S_1,S_2}$. Then it is clear that $f'(A,B)=f(A,B)$ a.s. under $\Pb_*'$, and that $f'\in \mathcal P_{D}'$ from Lemma~\ref{lem-D(S)-property}. Now we show that $\mathbb E_\Qb [(f')^2]\le O(1)\cdot\mathbb E_\Qb [f^2]$. For simplicity, we define 
\begin{align*}
    \mathcal{R}(H_1,H_2) =\{(S_1,S_2): S_1,S_2 \Subset \mathcal K_n, |E(S_1)|+|E(S_2)|\leq D, H_1\in \mathcal{A}(S_1),H_2\in\mathcal{A}(S_2) \} \,.
\end{align*}
Recalling \eqref{eq-def-hat-f-S1S2} and \eqref{eq-def-Lambda-H,S}, we have that
\begin{align}
    \phi'_{S_1,S_2}(A,B) &= \Big( \sum_{H_1 \in \mathcal{A}(S_1)} \Lambda_{S_1}(H_1) \psi_{H_1}(A) \Big) \cdot \Big( \sum_{H_2 \in \mathcal{A}(S_2)} \Lambda_{S_2}(H_2) \psi_{H_2}(B) \Big) \nonumber \\
    &\overset{\eqref{eq-decomposition-of-phi-into-psi's}}{=} \sum_{  H_1 \in \mathcal{A}(S_1), H_2 \in \mathcal{A}(S_2) } \Lambda_{S_1}(H_1) \Lambda_{S_2}(H_2) \phi_{H_1,H_2}(A,B) \,. \label{equ-simple-form-phi'}
\end{align}
Thus, $f'$ can also be written as (recall \eqref{eq-def-Od})
\begin{align*}
    f'&=\sum_{\substack{S_1,S_2 \Subset \mathcal K_n\\|E(S_1)|+|E(S_2)|\le D}} C_{S_1,S_2} \Big(\sum_{H_1\in \mathcal A(S_1),H_2\in\mathcal A(S_2)} \Lambda_{S_1}(H_1) \Lambda_{S_2}(H_2) \phi_{H_1,H_2}(A,B) \Big)\\
    &= \sum_{ \substack{ H_1,H_2 \operatorname{admissible} }}\Bigg( \sum_{(S_1,S_2)\in \mathcal{R}(H_1,H_2)} C_{S_1,S_2} \Lambda_{S_1}(H_1) \Lambda_{S_2}(H_2) \Bigg) \phi_{H_1,H_2}(A,B) \,.
\end{align*}
Therefore, by \eqref{eq-standard-orthogonal}, we have that $\mathbb E_\Qb[(f')^2]$ is upper-bounded by
\begin{align}
    \nonumber& \sum_{ \substack{ H_1,H_2 \operatorname{admissible} }} \Bigg( \sum_{(S_1,S_2)\in \mathcal{R}(H_1,H_2)} C_{S_1,S_2} \Lambda_{S_1}(H_1) \Lambda_{S_2}(H_2) \Bigg)^2 \\
    \nonumber\stackrel{\eqref{eq-est-of-Lambda}}{\leq} & \sum_{ \substack{ H_1,H_2 \operatorname{admissible} \\  }} \Bigg( \sum_{(S_1,S_2)\in \mathcal{R}(H_1,H_2)} \big(\tfrac{16\lambda}{n}\big)^{ \frac{1}{2}(|E(S_1)|+|E(S_2)|-|E(H_1)|-|E(H_2)|) } |C_{S_1,S_2}| \Bigg)^2 \,. \\
    \leq &\sum_{ \substack{ H_1,H_2 \operatorname{admissible} }} \Big( \sum_{(S_1,S_2)\in \mathcal{R}(H_1,H_2)} D^{ -40 (\tau(S_1)+\tau(S_2)-\tau(H_1)-\tau(H_2)) } C^2_{S_1,S_2} \Big) \nonumber \\
    & * \Big( \sum_{ (S_1,S_2)\in \mathcal{R}(H_1,H_2)} \big(\tfrac{16\lambda}{n}\big)^{ |E(S_1)|+|E(S_2)|-|E(H_1)|-|E(H_2)| } D^{ 40 (\tau(S_1)+\tau(S_2)-\tau(H_1)-\tau(H_2)) } \Big)\,, \label{eq-last-bracket}
\end{align}
where the last inequality follows from Cauchy-Schwartz inequality.

Next we upper-bound the right-hand side of \eqref{eq-last-bracket}. To this end, we first show that
\begin{equation}\label{eq-est-last-bracket}
    \mbox{the last bracket in \eqref{eq-last-bracket} is uniformly bounded by } O(1) 
\end{equation}
for any two admissible graphs $H_1,H_2$. Note that using Lemma~\ref{lem-D(S)-property} and Lemma~\ref{lem-property-H-Subset-S} in the appendix (note that $H \ltimes S$ since $\mathcal I(S)=\emptyset$), we have 
\begin{equation}\label{eq-tauS-larger-tauH}
  \tau(S) \geq \tau(H)\mbox{ and }\Phi(H) \geq \Phi(S) \mbox{ for } H \in \mathcal A(S)\,.  
\end{equation}
Thus,
\begin{align}
    \sum_{ \substack{ S \Subset \mathcal K_n: |E(S)|\leq D \\ H \in \mathcal{A}(S) } } \big( \tfrac{16\lambda}{n} \big)^{ |E(S)|-|E(H)| } D^{40(\tau(S)-\tau(H))} \leq  \sum_{l,m=0}^{2D} \big(\tfrac{16\lambda}{n}\big)^{l+m} D^{40l} \cdot \mathrm{ENUM}_{l,m}\,, \label{eq-bound-second-bracket} 
\end{align}
where 
\begin{align*}
    \mathrm{ENUM}_{l,m}=\#\big\{ S\Subset \mathcal K_n: H \in \mathcal{A}(S), |E(S)|-|E(H)|=l+m, |V(S)|-|V(H)|=m \big\} \,.
\end{align*}
In light of \eqref{eq-tauS-larger-tauH}, in order for $\mathrm{ENUM}_{l,m}\neq 0$, we must have $\big( \tfrac{2\Tilde{\lambda}^2 k^2 n}{D^{50}} \big)^{m} \big( \tfrac{1000 \Tilde{\lambda}^{20} k^{20} D^{50}}{n} \big)^{l+m} \leq 1$. In addition, by Lemma~\ref{lem-characterization-D(S)} we have $\mathcal L(S) \subset V(H)$ and $\mathcal{C}_j(S) \subset H$ for any $H \in \mathcal{A}(S)$ and $j>N$. Thus, using Lemma~\ref{lem-enu-cycle-path} in the appendix (note that $S \Subset \mathcal K_n$ and $H\subset S$ imply that $H \ltimes S$), we have
\begin{align*}
    \mathrm{ENUM}_{l,m} \leq \sum_{p_3,\ldots,p_N \geq 0} (2D)^{4l} n^{m} \prod_{j=3}^{N} \frac{1}{p_j!} = O(1) \cdot (2D)^{4l} n^m \,.    
\end{align*}
Plugging this estimation into \eqref{eq-bound-second-bracket}, we obtain that \eqref{eq-bound-second-bracket} is bounded by $O(1)$ times 
\begin{align*}
    & \sum_{l,m=0}^{2D} \big(\tfrac{16\lambda}{n}\big)^{l+m} D^{40l} \cdot (2D)^{4l} n^{m} \cdot \mathbf{1}_{ \{ (2\Tilde{\lambda}^2 k^2 n/D^{50} )^m (1000 \Tilde{\lambda}^{20} k^{20} D^{50}/n)^{l+m} \leq 1 \} } \\
    \leq\ & \sum_{l,m=0}^{2D} (16\lambda)^m \big( \tfrac{256 D^{44} \lambda}{n} \big)^{l} \mathbf{1}_{  \{ (2000 \Tilde{\lambda}^2 k^2)^m (1000 \Tilde{\lambda}^{20} k^{20} D^{50}/n)^l \leq 1 \} } \overset{\Tilde{\lambda} \geq \lambda}{\leq} \sum_{l,m=0}^{2D} 2^{-l} k^{-m} = O(1) \,.
\end{align*}
Putting together the inequality $\eqref{eq-bound-second-bracket}=O(1)$ with respect to $H_1$ and $H_2$ verifies \eqref{eq-est-last-bracket}, since the last bracket in \eqref{eq-last-bracket} is upper-bounded by the product of these two sums. Therefore, we get that $\mathbb{E}_{\mathbb{Q}}[(f')^2]$ is upper-bounded by $O(1)$ times
\begin{align*}
     & \sum_{ H_1,H_2 \operatorname{admissible} } \Big( \sum_{(S_1,S_2)\in \mathcal{R}(H_1,H_2)} {D}^{ -40 (\tau(S_1)+\tau(S_2)-\tau(H_1)-\tau(H_2)) } C^2_{S_1,S_2} \Big) \\
     =& \sum_{ \substack{ S_1,S_2 \Subset \mathcal K_n \\ |E(S_1)|+|E(S_2)|\le D } } C^2_{S_1,S_2} \sum_{ H_1 \in \mathcal{A}(S_1), H_2 \in \mathcal{A}(S_2) } {D}^{ -40((\tau(S_1)+\tau(S_2)-\tau(H_1)-\tau(H_2)) } \,.
\end{align*}
In addition, for any fixed $S_1,S_2$ such that $|E(S_1)|\le D$ and $|E(S_2)| \leq D$, by \eqref{eq-tauS-larger-tauH}
\begin{align*}
    &\ \sum_{ H_1 \in \mathcal{A}(S_1), H_2 \in \mathcal{A}(S_2) } {D}^{-40((\tau(S_1)+\tau(S_2)-\tau(H_1)-\tau(H_2))} \\ 
    \le& \ \sum_{k_1=0}^{|E(S_1)|} \sum_{k_2=0}^{|E(S_2)|} D^{-40(k_1+k_2)} \cdot \#\big\{ (H_1,H_2):H_i\in \mathcal A(S_i), \tau(H_i)=\tau(S_i)-k_i \mbox{ for } i=1,2 \big\} \\
    \le& \ \sum_{k_1=0}^D \sum_{k_2=0}^D D^{-40(k_1+k_2)} \cdot D^{15(k_1+k_2)}\leq 2 \,,
\end{align*}
where the second inequality follows from Lemma~\ref{lem-enu-Subset-large-graph} in the appendix and the last one comes from the fact that $D \ge 100$. Hence, we have $\mathbb{E}_{\mathbb{Q}}[(f')^2] \leq O(1) \cdot \sum_{S_1,S_2} C_{S_1,S_2}^2 = O(1) \cdot \mathbb{E}_{\mathbb{Q}}[f^2]$, which completes the proof of Proposition~\ref{prop-same-L^1-bounded-L^2}.
\end{proof}

\subsection{Bounds on $\mathbb P'$-moments}

To bound the right-hand side of \eqref{eq-reduction-1}, we need the following estimation of $\mathbb{E}_{\Pb'}[ \phi_{S_1,S_2}]$. For $H \subset S$, we define $\mathtt N(S,H)$ to be 
\begin{align}
    \mathtt N(S,H) &= \big( \tfrac{D^{28}}{n^{0.1}} \big)^{ \frac{1}{2} (|\mathcal L(S) \setminus V(H)| + \tau(S)-\tau(H)) } (1-\tfrac{\delta}{2})^{|E(S)|-|E(H)|} \,. \label{eq-def-mathtt-N}
\end{align}
\begin{proposition}{\label{prop-untruncate-expectation-Pb}}
   For all admissible $S_1,S_2\Subset \mathcal K_n$ with $|E(S_1)|,|E(S_2)| \leq D$, we have that $\big| \mathbb{E}_{\Pb'}[\phi_{S_1,S_2}] \big|$ is bounded by $O(1)$ times (note that in the summation below $H_1,H_2$ may have isolated vertices)
    \begin{align}{\label{eq-bound-expectation-Pb}}
         \sum_{ \substack{ H_1 \subset S_1, H_2 \subset S_2 \\ H_1 \cong H_2 } } \frac{ (\sqrt{\alpha}-\tfrac{\delta}{4})^{|E(H_1)|}  \operatorname{Aut}(H_1) }{ n^{|V(H_1)|} }  * \frac{ 2^{|\mathfrak{C}(S_1,H_1)|+|\mathfrak{C}(S_2,H_2)|} \mathtt N(S_1,H_1) \mathtt N(S_2,H_2) }{ n^{\frac{1}{2}(|V(S_1)|+|V(S_2)|-|V(H_1)|-|V(H_2)|)} }  \,.
    \end{align}
\end{proposition}
We remark that $\mathtt N(S,H)$ should be thought as $n^{-\Omega(1)\cdot (|\mathcal L(S) \setminus V(H)|+\tau(S)-\tau(H))}$, and the bound in \eqref{eq-bound-expectation-Pb} should be thought as follows. In the proof later $H_1$ will be the graph $\pi(S_1) \cap S_2$ and we will sum over all possible realizations of $H_1$. Since the total measure of all $\pi$'s satisfying that $\pi(S_1) \cap S_2 = H_1$ is given by $\tfrac{\operatorname{Aut}(H_1)}{n^{|V(H_1)|}}$, the main technical step of our argument is to bound the conditional expectation of $\phi_{S_1,S_2}$ given $\pi$ by some terms related to $\mathtt N(S_1,H_1) \mathtt N(S_2,H_2)$. Roughly speaking, this suggests that the conditional expectation of $\phi_{S_1,S_2}$ will be smaller assuming the following two items: (1) there are many leaves in $S_1$ that do not belong to $H_1$, as we will use some combinatorial arguments to show that the labels in such leaves create certain cancellations; (2) $S_1$ is ``denser'' than $H_1$, as there are more edges $S_1 \setminus H_1$ and each edge will contribute a factor of $n^{-\Omega(1)}$ to the expectation. For each deterministic permutation $\pi \in \mathfrak S_n$ and each labeling $\sigma \in [k]^{n}$, we denote $\Pb_{\sigma,\pi}'=\Pb'( \cdot \mid \sigma_{*}=\sigma, \pi_*=\pi )$, $\Pb_{\pi}'=\Pb'( \cdot \mid \pi_*=\pi )$ and $\Pb_{\sigma}'=\Pb'( \cdot \mid \sigma_*=\sigma )$ respectively. It is clear that 
\begin{align}
    &\ \big|\mathbb{E}_{\Pb'}[\phi_{S_1,S_2}]\big|= \Big|\frac{1}{n!}\sum_{\pi\in \mathfrak{S}_n}\mathbb{E}_{\Pb'_\pi}[\phi_{S_1,S_2}]\Big| \le \frac{1}{n!}\sum_{\pi\in\mathfrak{S}_n}\big|\mathbb{E}_{\Pb_\pi'}[\phi_{S_1,S_2}]\big| \,. \label{eq-relaxation-1} 
\end{align}
For $H\subset S$, we define
    \begin{equation}{\label{eq-def-mathtt-M}}
        \mathtt M(S,H) = \big( \tfrac{D^{8}}{n^{0.1}} \big)^{ \frac{1}{2}( |\mathcal L(S)\setminus V(H)| +\tau(S)-\tau(H) ) } (1-\tfrac{\delta}{2})^{|E(S)|-|E(H)|} \,.
    \end{equation}
Then we proceed to provide a delicate estimate on \(|\mathbb{E}_{\Pb_\pi'}[\phi_{S_1,S_2}]|\), as in the next lemma.
\begin{lemma}{\label{lem-trun-expectation-given-pi}}
    For all admissible $S_1,S_2\Subset \mathcal K_n$ with at most $D$ edges and for all permutation $\pi$ on $[n]$, denote $H_1 = S_1 \cap \pi^{-1}(S_2)$ and $H_2=\pi(S_1)\cap S_2$. We have that $\big|\mathbb E_{\Pb_\pi'}[\phi_{S_1,S_2}] \big|$ is bounded by $O(1)$ times
    \begin{equation}
    \begin{aligned}
        & (\sqrt{\alpha}-\tfrac{\delta}{4})^{|E(H_1)|} \sum_{ H_1 \ltimes K_1 \subset S_1 } \sum_{ H_2 \ltimes K_2 \subset S_2 } \frac{ \mathtt M(S_1,K_1) \mathtt M(S_2,K_2) \mathtt M(K_1,H_1) \mathtt M(K_2,H_2) }{ n^{ \frac{1}{2}(|V(S_1)|+|V(S_2)|-|V(H_1)|-|V(H_2)|) } }  \,.
    \end{aligned}
    \end{equation}
\end{lemma}

We remark that the definition of $\mathtt M(S,H)$ is similar to the definition of $\mathtt N(S,H)$ in \eqref{eq-def-mathtt-N}, except that the exponent $D^{28}$ is changed to $D^8$. The reason is that our strategy is to first bound $\big|\mathbb E_{\Pb_\pi'}[\phi_{S_1,S_2}] \big|$ by $\sum_{ H_1 \ltimes K_1 \subset S_1 } \mathtt M(S_1,K_1)\mathtt M(K_1,H_1) \sum_{ H_2 \ltimes K_2 \subset S_2 } \mathtt M(S_2,K_2)\mathtt M(K_2,H_2)$, and then bound $\sum_{ H_i \ltimes K_i \subset S_i } \mathtt M(S_i,K_i)\mathtt M(K_i,H_i)$ by $\mathtt N(S_i,H_i)$; the difference in the exponent of $D$ is tuned carefully such that the later bound holds. The proof of Lemma~\ref{lem-trun-expectation-given-pi} is the most technical part of this paper and is included in Section~\ref{sec:Proof-Lem-4.10} in the appendix. Now we can finish the proof of Proposition~\ref{prop-untruncate-expectation-Pb}.

\begin{proof}[Proof of Proposition~\ref{prop-untruncate-expectation-Pb}.]
    Note that we have
    \begin{equation}{\label{eq-measure-intersection-pattern}}
        \mu \big( \{ \pi: S_1 \cap \pi^{-1}(S_2) = H_1 \} \big) \leq [1+o(1)] \cdot \operatorname{Aut}(H_1) n^{-|V(H_1)|} \,. 
    \end{equation}
    Combined with Lemma~\ref{lem-trun-expectation-given-pi}, it yields that the right-hand side of \eqref{eq-relaxation-1} is bounded by (up to a $O(1)$ factor)
    \begin{align}
        & \sum_{\substack{ H_1,H_2: H_1 \cong H_2 \\ H_1 \subset S_1, H_2 \subset S_2 }} \frac{ \operatorname{Aut}(H_1) (\sqrt{\alpha} -\tfrac{\delta}{4})^{|E(H_1)|} }{ n^{|V(H_1)|} } * \frac{ \mathtt P(S_1,H_1) \mathtt P(S_2,H_2) }{ n^{ \frac{1}{2}(|V(S_1)|+|V(S_2)|-|V(H_1)|-|V(H_2)|) } } \,,  \label{eq-relaxation-2}
    \end{align}
    where (recall \eqref{eq-def-mathtt-M})
    \begin{equation}{\label{eq-def-mathtt-P}}
        \mathtt P(S,H) = \sum_{H \ltimes K \subset S} \mathtt M(S,K) \mathtt M(K,H) \,.
    \end{equation}
    We claim that we have the following estimation, with its proof incorporated in Section~\ref{subsec:Proof-Claim-4.11} in the appendix.
    \begin{claim}{\label{claim-relation-mathtt-M-N}}
        We have $\mathtt P(S,H) \leq [1+o(1)] \cdot 2^{|\mathfrak{C}(S,H)|} \mathtt N(S,H)$.
    \end{claim}
    \noindent Note that the estimate as in \eqref{eq-bound-expectation-Pb} will be obvious if we plug Claim~\ref{claim-relation-mathtt-M-N} into \eqref{eq-relaxation-2}.
\end{proof}

Now we can finally complete our proof of Item (iii) of Theorem~\ref{main-thm-detection-lower-bound}.

\begin{proof}[Proof of Item (iii), Theorem~\ref{main-thm-detection-lower-bound}.]
Recall Definition~\ref{def-phi-S1,S2} and \eqref{eq-goal-complexity-lower-bound}. Note that $\mathcal O_{D}$ is an orthonormal basis under $\Qb$. As incorporated in \cite[Equation (3.18)]{DDL23+}, we get from the standard results that 
\begin{equation}{\label{eq-expression-SNR}}
    \| L'_{\leq D} \|^2 = \sum_{\phi_{S_1,S_2} \in \mathcal{O}_D'} \big(\mathbb{E}_{\Pb'}[\phi_{S_1,S_2}] \big)^2 \,.
\end{equation}
Recall \eqref{eq-def-mathtt-P}. By Proposition~\ref{prop-untruncate-expectation-Pb} and Cauchy-Schwartz inequality, $\| L'_{\leq D} \|^2$ is upper-bounded by $O(1)$ times
\begin{align}
    & \sum_{\phi_{S_1,S_2} \in \mathcal O'_D} \Big( \sum_{ \substack{ H_1 \subset S_1, H_2 \subset S_2 \\ H_1 \cong H_2 } } n^{0.02|\mathcal I(H_1)|} \mathtt N(S_1,H_1) \mathtt N(S_2,H_2) \Big)  \label{eq-L-leq-D-relaxation-1.2} \times \\
    & \Big( \sum_{ \substack{ H_1 \subset S_1, H_2 \subset S_2 \\ H_1 \cong H_2 } } \frac{ (\sqrt{\alpha}-\tfrac{\delta}{4})^{2|E(H_1)|} \operatorname{Aut}(H_1)^2 }{ n^{2|V(H_1)| + 0.02|\mathcal I(H_1)| } } \cdot \frac{ 4^{|\mathfrak{C}(S_1,H_1)|+|\mathfrak{C}(S_2,H_2)|} \mathtt N(S_1,H_1) \mathtt N(S_2,H_2) }{ n^{( |V(S_1)|+|V(S_2)|-|V(H_1)|-|V(H_2)| )} } \Big) \,. \nonumber
\end{align}
We firstly bound the bracket in $\eqref{eq-L-leq-D-relaxation-1.2}$. Note that for all $H \subset S$, we have $|\mathcal L(S) \setminus V(H)| \leq |V(S)|-|V(H)|$, and thus $|E(S)|-|E(H)| \geq |\mathcal L(S) \setminus V(H)| + \tau(S)-\tau(H)$. In addition, $H \ltimes S$ provided with $S \Subset \mathcal K_n$. Thus, from Lemma~\ref{lem-decomposition-H-Subset-S} in the appendix we see that $|\mathcal L(S) \setminus V(H)| + \tau(S)-\tau(H) \geq |\mathcal I(H)|/2$, since all isolated vertices of $H$ must be endpoints of paths in the decomposition of $E(S)\setminus E(H)$ in Lemma~\ref{lem-decomposition-H-Subset-S} in the appendix. Thus, for any fixed admissible $S\Subset \mathcal{K}_n$, we have that $\sum_{H:H \subset S} n^{0.01|\mathcal I(H)|} \mathtt N(S,H)$ is bounded by (recall \eqref{eq-def-mathtt-N})
\begin{align*}
    & \sum_{m \geq 0} \sum_{l \geq m \geq r/2 \geq 0} n^{0.01r} (1-\tfrac{\delta}{2})^{l} \big(\tfrac{D^{14}}{n^{0.05}}\big)^{m} * \#\Big\{ H \ltimes S: H \mbox{ is admissible}, |\mathcal I(H)|=r, \\
    &|E(S)|-|E(H)|=l, |\mathcal L(S)\setminus V(H)| + \tau(S)-\tau(H) = m  \Big\} \\
    \leq\ & \prod_{j=N+1}^{D} \sum_{q_j\geq 0 }\binom{ |\mathcal C_j(S)| }{ q_j } (1-\tfrac{\delta}{2})^{j q_j} \sum_{m \geq r/2 \geq 0} 
    n^{0.01r} \big(\tfrac{D^{14}}{n^{0.05}}\big)^{m} D^{8m}  \\
    \leq\ & [1+o(1)] \cdot \prod_{j=N+1}^{D} ( 1+(1-\tfrac{\delta}{2})^j )^{ |\mathcal C_j(S)| } \leq [1+o(1)] \cdot (1+(1-\tfrac{\delta}{2})^N)^{|\mathfrak C(S,H)|+|\mathcal I(H)|+2|E(H)|} \,,
\end{align*}
where the first inequality follows from Lemma~\ref{lem-enu-general-subset-large-graph} in the appendix and the last inequality follows from
\begin{align*}
    \sum_{j=N+1}^{D}|\mathcal C_j(S)| \leq |\mathcal C(S)| \leq |\mathfrak{C}(S,H)| + |V(H)| \leq |\mathfrak C(S,H)|+|\mathcal I(H)|+2|E(H)| \,.
\end{align*}
Thus, we have that
\begin{align}
    \eqref{eq-L-leq-D-relaxation-1.2} \leq [1+o(1)] \cdot \big( 1+(1-\delta/2)^N \big)^{ \sum_{i=1,2} (|\mathfrak C(S_i,H_i)|+|\mathcal I(H_i)|+2|E(H_i)|) } \,. \label{eq-relaxation-1.2.1}
\end{align}
Recall \eqref{eq-def-N}, \eqref{eq-expression-SNR} and Proposition~\ref{prop-untruncate-expectation-Pb}. By \eqref{eq-relaxation-1.2.1} and \eqref{eq-def-N} (which helps us bounding $(1-\delta/2)^N$), we get that $\| L'_{\leq D} \|^2$ is bounded by $O(1)$ times (denoted by $\widetilde{\mathtt N}(S,H)=\frac{ 8^{|\mathfrak{C}(S,H)|} \mathtt N(S,H) }{ n^{(|V(S)|-|V(H)|)} }$)
\begin{align*}
    & \sum_{ (S_1,S_2): \phi_{S_1,S_2} \in \mathcal O_{D}' } \sum_{ \substack{ H_1 \subset S_1, H_2 \subset S_2 \\ H_1 \cong H_2 } } \frac{ (\sqrt{\alpha}-\tfrac{\delta}{8})^{2|E(H_1)|} \operatorname{Aut}(H_1)^2 }{ n^{2|V(H_1)|+0.01 |\mathcal I(H_1)|} } \cdot \widetilde{\mathtt N}(S_1,H_1) \widetilde{\mathtt N}(S_2,H_2) \\
    = \ & \sum_{\substack{H_1 \cong H_2, H_1,H_2 \text{ admissible} \\ |E(H_1)|+|E(H_2)| \leq D}} \frac{  (\sqrt{\alpha}-\tfrac{\delta}{8})^{2 |E(H_1)| } \mathrm{Aut}(H_1)^2 }{ n^{2|V(H_1)|+0.01|\mathcal I(H_1)|}  } \cdot \sum_{ \substack{(S_1,S_2): S_1,S_2 \Subset \mathcal K_n\\ H_1 \subset S_1, H_2 \subset S_2  } } \widetilde{\mathtt N}(S_1,H_1) \widetilde{\mathtt N}(S_2,H_2) \,. 
\end{align*}
Recall that we use $\widetilde{H}_1$ to denote the subgraph of $H_1$ obtained by removing all the vertices in $\mathcal I(H_1)$. For $|V(H_1)| \leq |V(S_1)| \le 2D$, we have $\operatorname{Aut}(H_1) = \operatorname{Aut}( \widetilde{H}_1 ) \cdot |\mathcal I(H_1)|! \leq (2D)^{|\mathcal I(H_1)|} \operatorname{Aut}( \widetilde{H}_1 )$. Thus, we have that 
\begin{align*}
    & \sum_{\substack{H_1 \cong H_2, H_1,H_2 \text{ admissible} \\ |E(H_1)|+|E(H_2)| \leq D}} \frac{ (\sqrt{\alpha}-\tfrac{\delta}{8})^{2|E(H_1)|} \mathrm{Aut}(H_1)^2 }{ n^{2|V(H_1)|+0.01|\mathcal I(H_1)|}  } \\
    \leq\ & \sum_{\substack{|E(\mathbf H)| \leq D, \mathcal I(\mathbf H)=\emptyset \\ \mathbf H \textup{ is admissible} }} \sum_{j \geq 0}\ \sum_{ \substack{ (H_1,H_2): \widetilde{H}_1 \cong \widetilde{H}_2 \cong \mathbf H \\ |\mathcal I(H_1)|=|\mathcal I(H_2)|=j } } n^{-0.01j} \cdot \frac{ \mathrm{Aut}(\mathbf H)^2 (2D)^{2j} (\sqrt{\alpha}-\tfrac{\delta}{8})^{2 |E(\mathbf H)| } }{ n^{2(|V(\mathbf H)|+j)} } \\
    \circeq \ & \sum_{\substack{|E(\mathbf H)| \leq D, \mathcal I(\mathbf H)=\emptyset \\ \mathbf H \textup{ is admissible} }} (\sqrt{\alpha}-\tfrac{\delta}{8})^{2 |E(\mathbf H)| } \leq O_\delta(1) \,,
\end{align*}
where the $\circeq$ follows from $\#\{ H \subset \mathcal K_n: \widetilde{H} \cong \mathbf H, |\mathcal I(H)|=j \} \circeq \frac{ n^{ |V(\mathbf H)|+j } }{ \operatorname{Aut}(\mathbf H) }$ and the last inequality follows from \cite[Lemma~A.3]{DDL23+}. In order to complete the proof of Item (3), (in light of the preceding two displays) it remains to show that 
\begin{equation}{\label{eq-relation-mathtt-N}}
    \sum_{ \substack{(S_1,S_2): S_1,S_2 \Subset \mathcal K_n\\ H_1 \subset S_1, H_2 \subset S_2  } } \widetilde{\mathtt N}(S_1;H_1) \widetilde{\mathtt N}(S_2;H_2) = O_{\delta,N}(1) \,.
\end{equation}
From \eqref{eq-def-mathtt-N}, we have
\begin{align}
    & 
    \widetilde{\mathtt N}(S;H) = 
    \frac{ \big(\tfrac{D^{14}}{n^{0.05}}\big)^{|\mathcal L(S) \setminus V(H)|+(\tau(S)-\tau(H))} 8^{|\mathfrak{C}(S,H)|} (1-\tfrac{\delta}{2})^{|E(S)|-|E(H)|} }{ n^{(|V(S)|-|V(H)|)} } \,. \nonumber
\end{align}
Since $S$ contains no isolated vertex, we have $|V(S)|-|V(H)| \leq 2(|E(S)|-|E(H)|)$. Then 
$\sum_{H \subset S \Subset \mathcal K_n}
    \widetilde{\mathtt N}(S;H)$ is further bounded by
\begin{align}
    \sum_{ \substack{r,m,p,q \geq 0 , m \geq p-q\\ q\leq 2p }} \frac{   \big(\tfrac{D^{14}}{n^{0.05}}\big)^{m} 8^r(1-\delta/2)^{p}}{ n^q } * \#\Big\{ S \mbox{ admissible} : H \subset S \Subset \mathcal K_n, \mathfrak{C}(S,H)=r, \nonumber \\
    |\mathcal L(S) \setminus V(H)| + \tau(S)-\tau(H)=m, |E(S)|-|E(H)|=p, |V(S)|-|V(H)|=q \Big\}\nonumber\\
    =   \sum_{ \substack{r,m,p,q \geq 0 , m \geq p-q\\ q\leq 2p }} \frac{   \big(\tfrac{D^{14}}{n^{0.05}}\big)^{m} 8^r(1-\delta/2)^{p}}{ n^q } \sum_{\substack{c_{N+1},\ldots,c_D:\\c_{N+1}+\cdots+c_D=r}} \operatorname{Count}(c_{N+1},\ldots,c_D) \,, \label{eq-relaxation-1.3.1}
\end{align}
where $\operatorname{Count}(c_{N+1},\ldots,c_D)$ equals to
\begin{align*}
    \#\Big\{& S \mbox{ admissible} \!: H \subset S ; |\mathfrak{C}_l(S,H)|=c_l, l \geq N; |\mathcal L(S) \setminus V(H)| + \tau(S)-\tau(H)=m,\\
    &|E(S)|-|E(H)|=p, |V(S)|-|V(H)|=q \Big\}\,.
\end{align*}
In addition, by Lemma~\ref{lem-enu-Subset-small-graph} in the appendix (i.e., Equation~\eqref{eq-enu-Subset-small-graph}), we have that (noting that $S \Subset \mathcal K_n$, and together with $H \subset S$ imply that $H \ltimes S$) 
\begin{align}
    \operatorname{Count}(c_{N+1},\ldots,c_D)\leq (2D)^{3m} n^{q} \sum_{ \substack{p_l \geq c_l \\ p \geq \sum_{j=N+1}^D jp_j}} \prod_{l=N}^{D} \frac{1}{p_l!} \,. \label{eq-relaxation-1.1.1}
\end{align} 
Plugging \eqref{eq-relaxation-1.1.1} into \eqref{eq-relaxation-1.3.1}, we get that \eqref{eq-relaxation-1.3.1} is bounded by 
\begin{align*}
    & \sum_{c_{N+1},\ldots,c_D \geq 0} \sum_{p_l \geq c_l} \sum_{ \substack{m,p,q\geq 0, m \geq p-q\\p\geq\sum_{l=N+1}^D lp_l,q\leq 2p} }  (1-\delta/2)^{p} 8^{c_{N+1}+\ldots+c_D} \Big( \frac{8 D^{17}}{n^{0.05}} \Big)^{m}  \prod_{l=N+1}^{D} \frac{ 1 }{ p_l! } \\
    =\ & \sum_{ \substack{p_l \geq c_{l} \geq 0 \text{ for } N+1 \leq l \leq D \\ p\geq\sum_{l=N+1}^D lp_l } }  (1-\delta/2)^{p} 8^{c_{N+1}+\ldots+c_D} \prod_{l=N+1}^{D} \frac{1}{p_l!} \sum_{ \substack{ 0 \leq q\leq 2p \\ m \geq (p-q) \vee 0 } } \Big( \frac{8 D^{17}}{n^{0.05}} \Big)^{m}   \\
    \leq\ & [1+o(1)] \sum_{ \substack{p_l \geq c_{l} \geq 0 \text{ for } N+1 \leq l \leq D } } 8^{c_{N+1}+\ldots+c_D} \prod_{l=N+1}^{D} \frac{1}{p_l!} \sum_{p\geq\sum_{l=N+1}^D lp_l} (2p+1) (1-\delta/2)^{p} \\
    \leq\ & O_{\delta,N}(1)\cdot \sum_{ \substack{p_l \geq 0 \text{ for } N+1 \leq l \leq D } } \Big( \sum_{l=N+1}^D lp_l+1 \Big)  \prod_{l=N+1}^{D} \frac{(1-\delta/2)^{lp_l}}{p_l!} \sum_{c_l \leq p_l} 8^{c_{N+1}+\ldots+c_D} \\
    \leq\ & O_{\delta,N}(1)\cdot \sum_{p_{N+1},\ldots,p_D \geq 0} \Big( \sum_{l=N+1}^D lp_l+1 \Big) \prod_{l=N+1}^{D} \frac{ (10(1-\frac{\delta}{2})^{l})^{p_l} }{p_l!} \\ 
    \leq\ & O_{\delta,N}(1)\cdot \sum_{p_{N+1},\ldots,p_D \geq 0} \prod_{l=N+1}^{D} \frac{ lp_l(10(1-\frac{\delta}{2})^{l})^{p_l} }{p_l!} =O_{\delta,N}(1) \cdot e^{O_{\delta,N}(1)} = O_{\delta,N}(1)\,.
\end{align*}
Thus, we have verified \eqref{eq-relation-mathtt-N} as desired.
\end{proof}

\section*{Acknowledgment}

J. Ding is partially supported by by National Key R$\&$D program of China (Project No. 2023YFA1010103), NSFC Key Program (Project No. 12231002), and by New Cornerstone Science Foundation through the XPLORER PRIZE.

\appendix

\section{Supplementary proofs in Section~\ref{sec:detection-upper-bound}}{\label{sec:supp-proof-sec-3}}

\subsection{Proof of Lemma~\ref{prop_first_moment_phi}}{\label{subsec:Proof-Lem-3.4}}

We start our proof with some straightforward computations. Clearly we have
\begin{align}
    &\mathbb{E}_{\Pb_{\sigma,\pi}}\big[ \Bar{A}_{i,j} \big] = \mathbb{E}_{\Pb_{\sigma,\pi}}\big[ \Bar{B}_{\pi(i),\pi(j)} \big] = \tfrac{\omega(\sigma_i,\sigma_j)\epsilon\lambda s}{n} \,, \label{eq-Pb-sigma-pi-A,B} \\
    &\mathbb{E}_{\Pb_{\sigma,\pi}}\big[ \Bar{A}_{i,j} \Bar{B}_{\pi(i),\pi(j)} \big] = \tfrac{(a \omega (\sigma_i ,\sigma_j ) + b)\lambda s^2 }{n} = [1+O(n^{-1})] \cdot \tfrac{(1+\epsilon \omega(\sigma_i,\sigma_j))\lambda s^2}{n} \,, \label{eq-Pb-sigma-pi-AB}
\end{align}
where $a = \epsilon (1-\frac{2\lambda}{n}) = \epsilon + O(\frac{1}{n})$ and $b = 1-\frac{\lambda}{n} = 1 + O(\frac{1}{n})$ are introduced for convenience. Then decomposing $\mathbb E_{\mathbb P_{\sigma, \pi}} [\phi_{S_1, S_2}]$ into products over edges in the symmetric difference between $E(S_1)$ and $E(\pi^{-1}(S_2))$ as well as over edges in their intersection, we can apply \eqref{eq-Pb-sigma-pi-A,B} and \eqref{eq-Pb-sigma-pi-AB} accordingly and obtain that (below the $\circeq$ is used to account for factors of $1-\lambda s/n$ in the definition of $\phi_{S_1,S_2}$) 
\begin{align*}
    \mathbb{E}_{\Pb_{\sigma,\pi}}\big[ \phi_{S_1,S_2} \big] \circeq \prod_{(i,j) \in E(S_1) \triangle E(\pi^{-1}(S_2))} \tfrac{\omega(\sigma_i,\sigma_j)\sqrt{\epsilon^2 \lambda s}}{ \sqrt{n} } \prod_{(i,j) \in E(S_1) \cap E(\pi^{-1}(S_2))} s (b+a \omega(\sigma_i,\sigma_j)) \,. 
\end{align*}
Thus, we have
\begin{align}
    & \mathbb E_{\Pb_{\pi}} \big[ \phi_{S_1,S_2} \big] \circeq s^{ |E(S_1) \cap E(\pi^{-1}(S_2))| } \big( \tfrac{\epsilon^2\lambda s}{n} \big)^{\frac{1}{2}( |E(S_1)|+|E(S_2)|-2|E(S_1) \cap E(\pi^{-1}(S_2))| ) } \nonumber \\
    & * \mathbb E_{ \sigma\sim\nu } \Bigg[ \prod_{(i,j) \in E(S_1) \triangle E(\pi^{-1}(S_2))}  \omega(\sigma_i,\sigma_j) \prod_{(i,j) \in E(S_1) \cap E(\pi^{-1}(S_2))} \big( b+a\omega(\sigma_i,\sigma_j) \big) \Bigg] \,, \label{eq-untruncate-exp-Pb-phi-1}
\end{align}
where we recall that $\nu$ is the uniform distribution on $[k]^{n}$. For $i,j\in\{0,1\}$, denote $\mathsf{K}_{i,j}=\mathsf{K}_{i,j}(S_1,S_2,\pi)$ the set of edges which appear $i$ times in $S_1$ and appear $j$ times in $\pi^{-1}(S_2)$. Also, define $\mathsf{K}_s = \cup_{0\le i,j \le 1, i+j=s} \mathsf{K}_{i,j}$. Define $\mathsf{L}_{i,j}$ and $\mathsf{L}_s$ with respect to the vertices in the similar manner. With a slight abuse of notations, we will also use $\mathsf K_{s}$ and $\mathsf K_{i,j}$ to denote their induced graphs.  

\begin{lemma}{\label{lemma_observation_appendix_C}}
We have the following.
\begin{enumerate}
    \item[(i)] Suppose $J \subset \mathcal K_n$ and suppose $u \in \mathcal L(J)$ with $(u,v) \in E(J)$. Then for any function $\psi$ measurable with respect to $\{ \sigma_{i} : i \in V(J) \setminus \{ u \} \}$ we have $\mathbb{E}_{\sigma\sim\nu} \big[ \omega (\sigma_u,\sigma_v) \cdot \psi \big]= 0$. In particular, for any tree $T$ we have
    \begin{equation}\label{eq-prod-edges-expectation-tree}
        \mathbb E_{\sigma\sim\nu}\Big[\prod_{(i,j)\in E(T)}\omega(\sigma_i,\sigma_j)\Big]=0 \,.
    \end{equation}
    \item[(ii)] Define $\mathfrak A = \{ \pi \in \mathfrak{S}_n : |\mathsf{L}_{2}| \ge |E(\mathsf{K}_{2} )| + 2 \}$, then
    \begin{align*}
        \mathbb{E}_{\Pb}\big[ \phi_{S_1,S_2} \big] = \mathbb{E}_{\Pb}\big[ \phi_{S_1,S_2} \mathbf{1}_{\{ \pi_*(S_1)=S_2 \}} \big] + \mathbb{E}_{\Pb}\big[ \phi_{S_1,S_2} \mathbf{1}_{\{ \pi_*(S_1)\neq S_2 \} \cap \{ \pi_* \in \mathfrak A \}} \big] \,.
    \end{align*}
\end{enumerate}
\end{lemma}
\begin{proof}
As for Item (i), define $\sigma_{u}$ and $\sigma_{\setminus u}$ to be the restriction of $\sigma$ on $\{u\}$ and on $[n] \setminus \{u\}$, respectively. Also define $\nu_{u}$ and $\nu_{\setminus u}$ to be the restriction of $\nu$ on $\{u\}$ and on $[n] \setminus \{u\}$, respectively. Then we have
\begin{align*}
    &\mathbb{E}_{\sigma\sim\nu} \Big[ \omega (\sigma_u,\sigma_v)  \psi \Big] = \mathbb{E}_{\sigma_{\setminus u}\sim\nu_{\setminus u}}  \mathbb{E}_{\sigma_{u}\sim\nu_u} \Big[ \omega (\sigma_u,\sigma_v)  \psi \Big]  = \mathbb{E}_{ \sigma_{\setminus u}\sim\nu_{\setminus u} } \Big[ \psi  \mathbb{E}_{\sigma_{u} \sim \nu_u} \big[ \omega(\sigma_u,\sigma_v) \big] \Big] =0 \,,
\end{align*}
which also immediately implies \eqref{eq-prod-edges-expectation-tree}.
As for Item (ii), it suffices to show that (recall \eqref{eq-untruncate-exp-Pb-phi-1})
\begin{equation}{\label{eq-C.2}}
    \mathbb E_{ \sigma\sim\nu } \Bigg[ \prod_{(i,j) \in E(S_1) \triangle E(\pi^{-1}(S_2))} \omega(\sigma_i,\sigma_j) \prod_{(i,j) \in E(S_1) \cap E(\pi^{-1}(S_2))} \big( b +a\omega(\sigma_i,\sigma_j) \big) \Bigg] =0
\end{equation}
for those $\pi \not \in \mathfrak A$ such that $\pi(S_1) \neq S_2$. Expanding the second product in \eqref{eq-C.2}, we get that proving \eqref{eq-C.2} is equivalent to showing that (recall the definition of $\mathsf{K}_1$ and $\mathsf{K}_2$) 
\begin{align*}
    & \mathbb E_{ \sigma\sim\nu } \Bigg[ \sum_{\mathsf{K}' \Subset \mathsf{K}_2} b^{|E(\mathsf K_2 )|-|E(\mathsf K' )|} \prod_{(i,j) \in E(\mathsf{K}_1)} \omega(\sigma_i,\sigma_j) \prod_{(i,j) \in E(\mathsf{K}')} \Big( a \omega(\sigma_i,\sigma_j) \Big) \Bigg]  \\
    =\ & \mathbb E_{ \sigma\sim\nu } \Bigg[ \sum_{\mathsf{K}' \Subset \mathsf{K}_2} b^{|E(\mathsf K_2 )|-|E(\mathsf K' )|} {a}^{|E(\mathsf{K}')|}\prod_{(i,j) \in E(\mathsf{K}_1 \cup \mathsf{K}')} \omega(\sigma_i,\sigma_j) \Bigg] =0 \,.
\end{align*}
By Item (i), it suffices to prove when $\pi \not \in \mathfrak A$ and $\pi(S_1) \neq S_2$ we have $\mathcal{L}(\mathsf{K}' \cup \mathsf{K}_1)\neq\emptyset$ for all $\mathsf{K}' \Subset \mathsf{K}_2$. If $\mathsf{K}_2=\emptyset$, we have $|\mathsf{L}_2| \leq 1$ since $\pi \not \in \mathfrak A$. Since a tree has at least $2$ leaves, there exists $u \in \mathcal L(S_1) \setminus V( \pi^{-1}(S_2) )$, and thus $u \in \mathcal L( \mathsf{K}' \cup \mathsf{K}_1 )$ for all $\mathsf{K}' \Subset \mathsf{K}_2$. Now suppose $\mathsf{K}_2 \neq \emptyset$. Since $\mathsf{K}_2 = S_1 \cap \pi^{-1}(S_2)$ is a subgraph of the tree $S_1$, we have
\begin{equation}{\label{eq-C..4}}
    |\mathsf{L}_2| \geq |V(\mathsf{K}_2)| \geq |E(\mathsf{K}_2)| + 1 \,.
\end{equation}
Also, from $\pi_* \not \in \mathfrak A$ we obtain $|\mathsf{L}_2| \le |E(\mathsf{K}_{2})| + 1$. Therefore, the inequalities in \eqref{eq-C..4} must be equalities, showing that $\mathsf{K}_2$ is a tree and $\mathsf{L}_2= V(\mathsf{K}_2)$. Since $\pi(S_1) \neq S_2$, $S_1 \doublesetminus \mathsf{K}_{2}$ is not empty and contains at least one connected component, which we write as $S_{1}^{*}$ (note that $S_1^*$ must be connected to $\mathsf K_2$ in $S_1$). We next prove that $|V(S_1^*) \cap V(\mathsf{K}_2) | \le 1$. Since $S_1^* \cup \mathsf{K}_2 \subset S_1$ is connected, it must be a subtree of $S_1$. Therefore, it cannot contain any cycle and thus we have $|V(S_1^*) \cap V(\mathsf{K}_2) | \le 1$: this is because otherwise we have two vertices in $V(S_1^*) \cap V(\mathsf{K}_2)$ which are connected by a path in $S_1^*$ and also a path in $\mathsf{K}_2$ (and clearly these two paths are edge disjoint), forming a cycle and leading to a contradiction. 
Now, since $S_1^*$ is a tree and $|V(S_1^*) \cap V(\mathsf{K}_2) | \le 1$, there exists at least one leaf in $S_{1}^{*}$ which does not belong to $\mathsf K_2$. Therefore, this leaf remains a leaf in $\mathsf{K}_1 \cup \mathsf{K}'$ for all $\mathsf{K}' \Subset \mathsf{K}_2$, which proves the desired result.
\end{proof}

\begin{lemma} {\label{lemma_order_growth}}
Recall that $\mu$ is the uniform measure over all permutations in $\mathfrak S_n$. For $m \geq 0$ denote $\operatorname{Overlap}_m = \{ \pi \in \mathfrak{S}_n : |E(\mathsf{K}_2 )| = \aleph - m, |\mathsf{L}_2| \ge \aleph - m + 2 - \mathbf{1}_{\{ m = 0 \} } \}$. We have for $m \ge 1$ and sufficiently large $n$
\begin{align*}
    \mu(\operatorname{Overlap}_m) \leq n^{m-0.5}\mu(\operatorname{Overlap}_0)\,.
\end{align*}
\end{lemma}

\begin{proof}
Firstly note that $\operatorname{Overlap}_0=\{ \pi \in \mathfrak S_n : \pi(S_1)=S_2 \}$, and thus we have 
\begin{equation*}
    \# \operatorname{Overlap}_0 = \operatorname{Aut}(S_1) \cdot (n-\aleph-1)! \geq (n-\aleph-1)! \,.
\end{equation*}
It remains to bound $\#\operatorname{Overlap}_m$ for $m \geq 1$. For each $\pi \in \operatorname{Overlap}_m$, denote $ V_{\operatorname{ov}}=\{ v \in V(S_1) : \pi(v) \in V(S_2) \}$, we have that $|V_{\operatorname{ov}}| \ge \aleph-m+2$. Also, there are at most $\binom{\aleph+1}{|V_{\operatorname{ov}}|} \leq 2^{\aleph+1}$ choices for $V_{\operatorname{ov}}$, and at most $(\aleph+1)^{|V_{\operatorname{ov}}|} \leq (\aleph+1)^{\aleph+1}$ choices for $(\pi(v))_{v\in V_{\operatorname{ov}}}$. Thus
\begin{align*}
    \#\operatorname{Overlap}_m \le 2^{\aleph+1} (\aleph+1)^{\aleph+1} \cdot  (n-\aleph+m-2)! \le n^{m-0.5} \cdot (n-\aleph-1)!  \,,
\end{align*}
where the last inequality follows from the fact that $\aleph^{2\aleph} = n^{o(1)}$ for $\aleph=o( \frac{\log n}{\log \log n} )$. This completes the proof.
\end{proof}

Now we can finish our proof of Lemma~\ref{prop_first_moment_phi}. 

\begin{proof}[Proof of Lemma~\ref{prop_first_moment_phi}]
Using Item~(ii) in Lemma~\ref{lemma_observation_appendix_C}, we have
\begin{equation}{\label{eq-C.3}}
    \mathbb{E}_{\Pb}[\phi_{S_1,S_2}] = \mathbb{E}_{\pi\sim\mu}\Big[ \mathbb{E}_{\Pb_{\pi}} \big[ \phi_{S_1,S_2} \textbf{1}_{\{ \pi(S_1)=S_2 \} } \big] \Big] + \mathbb{E}_{\pi\sim\mu}\Big[ \mathbb{E}_{\Pb_{\pi}} \big[ \phi_{S_1,S_2} \textbf{1}_{ \{ \pi(S_1)\ne S_2 \} \cap \{\pi\in \mathfrak A\} } \big] \Big] \,,
\end{equation}
where $\mu$ is the uniform distribution over $\mathfrak S_n$.
Using \eqref{eq-untruncate-exp-Pb-phi-1}, we have that 
\begin{align}
    & \mathbb{E}_{\pi\sim\mu} \Big[  \mathbb{E}_{\Pb_{\pi}} \big[ \phi_{S_1,S_2} \textbf{1}_{\{ \pi(S_1)=S_2 \}} \big] \Big] \nonumber \\
    \circeq\ & s^{|E(S_1)|} \mathbb{E}_{\pi\sim\mu} \Bigg\{ \textbf{1}_{\{ \pi(S_1)=S_2 \}} \mathbb E_{\sigma\sim\nu} \Big[ \prod_{(i,j) \in E(S_1)} (b+a \omega(\sigma_i,\sigma_j)) \Big]  \Bigg\} \nonumber \\
    \circeq\ & s^\aleph \mu( \{ \pi\in \mathfrak{S}_n: \pi(S_1)=S_2 \} ) \,, \label{eq-C.4}
\end{align}
where the second equality follows from the fact that $S_1$ is a tree and \eqref{eq-prod-edges-expectation-tree}. In addition, we have that (recall our assumption that $\epsilon^2 \lambda s\leq 1$, which appears in \eqref{eq-untruncate-exp-Pb-phi-1})
\begin{align}
    & \mathbb{E}_{\pi\sim\mu}\Big[ \mathbb{E}_{\Pb_{\pi}} \big[ \phi_{S_1,S_2} \textbf{1}_{\{\pi(S_1)\ne S_2\} \cap \{\pi\in \mathfrak A\} } \big] \Big] \nonumber \\
    \circeq\ & \mathbb E_{\pi\sim\mu} \Bigg\{ \textbf{1}_{\{\pi(S_1)\ne S_2\} \cap \{\pi\in \mathfrak A\} } \cdot s^{|E(S_1) \cap E(\pi^{-1}(S_2))|} \big( \tfrac{\epsilon^2\lambda s}{\sqrt{n}} \big)^{|E(S_1)|+|E(S_2)|-2|E(S_1) \cap E(\pi^{-1}(S_2))|} \nonumber \\
    & * \mathbb E_{ \sigma\sim\nu } \Big[ \prod_{(i,j) \in E(S_1) \triangle E(\pi^{-1}(S_2))} \omega(\sigma_i,\sigma_j) \prod_{(i,j) \in E(S_1) \cap E(\pi^{-1}(S_2))} \big(b +a\omega(\sigma_i,\sigma_j) \big) \Big] \Bigg\} \nonumber \\ 
    \le \ & [1+o(1)] \cdot \sum_{m=1}^{\aleph} \mathbb E_{\pi\sim\mu} \Bigg\{ \textbf{1}_{ \{ \pi \in \operatorname{Overlap}_m \} } \cdot s^{\aleph-m} n^{ -m } \mathbb E_{ \sigma\sim\nu } \Big[ \prod_{(i,j) \in E(S_1) \triangle E(\pi^{-1}(S_2))} |\omega(\sigma_i,\sigma_j)| \nonumber \\
    & *  \prod_{(i,j) \in E(S_1) \cap E(\pi^{-1}(S_2))} \big| 1 +\epsilon \omega(\sigma_i,\sigma_j) \big| \Big] \Bigg\} 
    \leq \sum_{m=1}^{\aleph} k^{2\aleph} s^{\aleph-m} n^{-m} \mu( \operatorname{Overlap}_m ) \,, \label{eq-C.6}
\end{align}
where in the last inequality we used $|\omega(\sigma_i,\sigma_j)| \leq k-1$ and $\epsilon^2\lambda s\leq 1$. By Lemma~\ref{lemma_order_growth}, we see that
\begin{align}
    \eqref{eq-C.6} \leq n^{-0.5}k^{2\aleph}s^{\aleph}\sum_{m=1}^{\aleph}s^{-m}\mu(\operatorname{Overlap}_0)\overset{(5)}{\leq} o(1) \cdot s^{\aleph}\mu(\{\pi\in\mathfrak{S}_n: \pi(S_1)=S_2\}) \,.\label{eq-C.7}
\end{align}
Plugging \eqref{eq-C.4}, \eqref{eq-C.6} and \eqref{eq-C.7} into \eqref{eq-C.3}, we obtain
$$
\mathbb{E}_{\Pb}[ \phi_{S_1,S_2} ] \circeq s^\aleph \cdot \mathbb{P}(\pi_*(S_1)=S_2) \,,
$$
which completes the proof of the first equality in Lemma~\ref{prop_first_moment_phi}. Note that the second equality is obvious.
\end{proof}

 \subsection{Proof of Lemma~\ref{prop_principal}}{\label{subsec:Proof-Lem-3.6}}

This subsection is devoted to the proof of Lemma~\ref{prop_principal}. We first need a general lemma for estimating the joint moments of $\Bar{A}$ and $\Bar{B}$. 
\begin{lemma} {\label{up_to_second_moment_AB}}
For $0 \le r,t \le 2$ and $r+t \ge 1$, there exist $u_{r,t} = u_{r,t}(\epsilon,\lambda,s,n)$ and $v_{r,t}=v_{r,t}(\epsilon,\lambda,s,n)$ which tend to constants as $n \to +\infty$, such that
\[
\mathbb{E}_{\Pb_{\sigma,\pi}} \big[ \Bar{A}_{i,j}^r \Bar{B}_{\pi(i),\pi(j)}^t \big] = \tfrac{ \omega(\sigma_i,\sigma_j)u_{r,t} + v_{r,t} }{ n } \,.
\]
In particular, we have $u_{1,1} = (1+O(n^{-1}))\epsilon \lambda s^2, v_{1,1} = (1+O(n^{-1}))\lambda s^2$ and  $v_{1,0}=v_{0,1}=0$. 
\end{lemma}
\begin{proof}
    For the case $r+t=1$ or $r=t=1$, it suffices to recall \eqref{eq-Pb-sigma-pi-A,B} and \eqref{eq-Pb-sigma-pi-AB}. For general cases, we have
    \begin{align*}
    & \mathbb{E}_{\Pb_{\sigma,\pi}} \big[ \Bar{A}_{i,j}^r \Bar{B}_{\pi(i),\pi(j)}^t \big] = \mathbb{E}_{\Pb_{\sigma,\pi}} \big[ \Bar{A}_{i,j}^r \Bar{B}_{\pi(i),\pi(j)}^t \mathbf{1}_{\{ G_{i,j} = 0 \} } \big] + \mathbb{E}_{\Pb_{\sigma,\pi}} \big[ \Bar{A}_{i,j}^r \Bar{B}_{\pi(i),\pi(j)}^t \mathbf{1}_{\{ G_{i,j} = 1 \} } \big] \\
    = \ & \ \big( 1-\tfrac{(1+\epsilon\omega(\sigma_i,\sigma_j))\lambda}{n} \big) \big(-\tfrac{\lambda s}{n}\big)^{r+t} \\
    + \ & \ \tfrac{ (1+\epsilon \omega(\sigma_i,\sigma_j))\lambda }{ n } \Big( s \big(1-\tfrac{\lambda s}{n} \big)^r + (1-s) \big( -\tfrac{\lambda s}{n} \big)^r \Big) \Big( s \big(1-\tfrac{\lambda s}{n}\big)^t + (1-s) \big(-\tfrac{\lambda s}{n}\big)^t \Big) \,.
    \end{align*}
Therefore, there exist $u'_{r,t} = u'_{r,t}(\epsilon,\lambda,s,n)$ and $ v'_{r,t} = v'_{r,t}(\epsilon,\lambda,s,n)$ which tend to constants as $n \to +\infty$, such that
\begin{equation}
\mathbb{E}_{\Pb_{\sigma,\pi}} \big[ \Bar{A}_{i,j}^r \Bar{B}_{\pi(i),\pi(j)}^t \big] = 
\begin{cases}
    \frac{u'_{r,t}}{n} , & \omega (\sigma_i ,\sigma_j ) = k-1 \,, \\
    \frac{v'_{r,t}}{n} , & \omega (\sigma_i ,\sigma_j ) = -1 \,. 
\end{cases}
\end{equation}
Taking $u_{r,t} = \frac{u'_{r,t}+(k-1)v'_{r,t}}{k}$ and $v_{r,t} = \frac{u'_{r,t}- v'_{r,t}}{k}$ yields our claim.
\end{proof}

Now we give estimations on all principal terms first. For simplicity, for $0\leq i,j\leq 2$ denote by $\operatorname{K}_{i,j}$ the set of edges which appear $i$ times in $S_1$ and $T_1$ (i.e., the total number of times appearing in $S_1$ and $T_1$ is $i$), and appear $j$ times in $\pi^{-1}(S_2)$ and $\pi^{-1}(T_2)$. In addition, we define $\operatorname{K}_s = \cup_{0\le i,j \le 2, i+j=s} \operatorname{K}_{i,j}$. Define $\operatorname{L}_{i,j}$ and $\operatorname{L}_s$ with respect to vertices in the similar manner. With a slight abuse of notations, we will also use $\operatorname K_{s}$ and $\operatorname K_{i,j}$ to denote their induced graphs. Recall (11). It is clear that in the case $(S_1,S_2;T_1,T_2) \in \mathsf R_{\mathbf{H},\mathbf{I}}^*$, we have that $\operatorname{K}_s = \emptyset$ and $\operatorname{L}_s = \emptyset$ for $s\ge 3$. Similar to $\mathfrak{A}$ defined in the first moment computation, we define
\begin{equation*}
    \mathfrak A ' = \{ \pi \in \mathfrak S_n: |\operatorname{L}_2| \ge |E(\operatorname{K}_2 )| + 3 \} \,.
\end{equation*}
\begin{lemma}{\label{principal-decomposition}}
Denote the set of permutations $\mathcal{M} = \{ \pi\in\mathfrak S_n : \pi (S_1 \cup T_1 ) = S_2 \cup T_2 \}$. For $(S_1,S_2;T_1,T_2)\in \mathsf R_{\mathbf H,\mathbf I}^*$, we have 
\begin{align}
    \mathbb{E}_{\Pb} \big[ \phi_{S_1,S_2}\phi_{T_1,T_2} \big] &= \mathbb{E}_{\Pb} \big[ \phi_{S_1,S_2}\phi_{T_1,T_2} \mathbf{1}_{ \{\pi_* \in \mathfrak A ' \cup \mathcal{M} \}} \big] \,. \label{eq-principal-decomposition} 
\end{align}
\end{lemma}
\begin{proof}
Similar to Lemma~\ref{lemma_observation_appendix_C} (ii), it suffices to prove when $\pi \in (\mathfrak A ' \cup \mathcal M )^c$ we have $\mathcal L (\operatorname{K}_1 \cup \operatorname{K}') \neq \emptyset$ for all $\operatorname{K}' \Subset \operatorname{K}_2$. If $\operatorname{K}_2 = \emptyset$, then $\operatorname{K}'=\emptyset$ and $|\operatorname{L}_2| \le 2$. Recalling Definition (11) and recalling our assumption that $(S_1,S_2;T_1,T_2)\in \mathsf R_{\mathbf H,\mathbf I}^*$ we have $V(S_1 ) \cap V(T_1 ) = \emptyset$. Therefore $|\mathcal L (S_1 \cup T_1) \setminus \operatorname{L}_2| = \big|\big(\mathcal L (S_1 ) \cup \mathcal L (T_1 )\big) \setminus \operatorname{L}_2 \big| \geq |\mathcal L (S_1 ) | + |\mathcal L (T_1 ) | - |\operatorname{L}_2 | \geq 2$. Thus, for all $\operatorname{K}' \Subset \operatorname{K}_2$ we have $\mathcal L (\operatorname{K}_1 \cup \operatorname{K}') \neq \emptyset$ by $\big(\mathcal L (S_1 ) \cup \mathcal L (T_1 )\big) \setminus \operatorname{L}_2 \subset \mathcal L (\operatorname{K}_1 \cup \operatorname{K}')$. If $\operatorname{K}_2 \neq \emptyset$, then $\operatorname{K}_2$ is a forest; in addition since $\pi \not\in \mathfrak A'$, $\operatorname{K}_2$ has at most two connected components. By $\pi \not\in \mathcal{M}$ and $V(S_1 )\cap V(T_1 )= \emptyset$, we know that either $S_1 \not\subset \operatorname{K}_2$ or $T_1 \not\subset \operatorname{K}_2$ holds. 
We may assume $S_1 \not\subset \mathrm K_2$. Since $S_1$ and $T_1$ are vertex disjoint, we see that the connected components of $S_1\cap \mathrm K_2$ are also connected components of $\mathrm K_2$; otherwise, suppose that there exists $(u,v) \in E(\mathrm K_2)$ such that $v$ is in the component and $u \in V(\mathrm K_2) \setminus V(S_1) \subset V(S_2)$, then we have $(u,v) \not \in E(S_1) \cup E(S_2)$, contradicting to $(u,v) \in E(\mathrm K_2)$ since $V(T_1) \cup V(T_2)=\emptyset$. Thus, if $S_1 \cap \mathrm K_2$ is disconnected, then both connected components of $\mathrm K_2$ are in $S_1$ and therefore $T_1 \cap \mathrm K_2  = \emptyset$. In this case, we have $\emptyset \neq \mathcal L(T_1 ) \subset \mathcal L(\operatorname{K}' \cup \operatorname{K}_1 )$ for all $\operatorname{K}'\Subset \operatorname{K}_2$. Else if $S_1 \cap \mathrm K_2$ is connected, then by the same arguments in Lemma~\ref{lemma_observation_appendix_C} (ii), we have $\emptyset \neq \mathcal L \big( S_1 \doublesetminus (S_1 \cap \mathrm K_2 ) \big) \subset \mathcal L (\operatorname{K}' \cup \operatorname{K}_1 )$ for all $\operatorname{K}'\Subset \operatorname{K}_2$. Combining the two cases above we complete the proof.
\end{proof}

\begin{lemma}\label{principal-enumeration-pi}
Suppose $(S_1,S_2;T_1,T_2) \in \mathsf R_{\mathbf H,\mathbf I}^*$. For $m \geq 0$ denote $\operatorname{Overlap}'_m = \{ \pi \in \mathfrak{S}_n : |E(\operatorname{K}_2 )|=2\aleph-m, |\operatorname{L}_2| \ge 2\aleph-m+3-\mathbf{1}_{\{ m=0\} } \}$. For $m\ge 1$ and sufficiently large $n$, we have
\begin{align*}
    \mu(\pi \in \operatorname{Overlap}'_m) \leq n^{m-0.5}\mu(\pi \in \operatorname{Overlap}'_0)\,.
\end{align*}
\end{lemma}

\begin{proof}
Firstly note that $\operatorname{Overlap}'_0=\{ \pi \in \mathfrak S_n : \pi(S_1 \cup T_1 )=S_2 \cup T_2 \}$, and thus we have 
\begin{equation*}
    \# \operatorname{Overlap}'_0 \geq \operatorname{Aut}(S_1)\operatorname{Aut}(T_1) \cdot (n-2\aleph-2)! \geq (n-2\aleph-2)! \,.
\end{equation*}
It remains to bound $\#\operatorname{Overlap}'_m$ for $m \geq 1$. For each $\pi \in \operatorname{Overlap}'_m$, denoting $V'_{\operatorname{ov}}=\{ v \in V(S_1 \cup T_1 ) : \pi(v) \in V(S_2 \cup T_2) \}$, we must have $|V'_{\operatorname{ov}}| \ge 2\aleph-m+3$. Also, there are at most $\binom{2\aleph+2}{|V'_{\operatorname{ov}}|} \leq 2^{2\aleph+2}$ choices for $V'_{\operatorname{ov}}$, and at most $(2\aleph+2)^{|V'_{\operatorname{ov}}|} \leq (2\aleph+2)^{2\aleph+1}$ choices for $(\pi(v))_{v\in V'_{\operatorname{ov}}}$. Thus
\begin{align*}
    \#\operatorname{Overlap}'_m \le 2^{2\aleph+2} (2\aleph+2)^{2\aleph+1} \cdot  (n-2\aleph+m-3)! \le n^{m-0.5} \cdot (n-2\aleph-2)!  \,,
\end{align*}
where the last inequality follows from the fact that $\aleph^{2\aleph} = n^{o(1)}$ for $\aleph=o( \frac{\log n}{\log \log n} )$. This completes the proof.
\end{proof}

Next we deal with non-principal terms. Define the set of good permutations:
\begin{align}\label{eq-def-frak-G}
    \mathfrak{G} = \{ \pi \in \mathfrak{S}_n : 2|\operatorname{L}_4| + |\operatorname{L}_3| - |\operatorname{L}_1| \ge 2|E(\operatorname{K}_4 )| + |E(\operatorname{K}_3 )| - |E(\operatorname{K}_1 )| + 2 \} \,.
\end{align}
Also, if $(S_1,S_2) \neq (T_1,T_2)$ define $\mathfrak{D}=\emptyset$; if $(S_1,S_2) = (T_1,T_2)$, define 
\begin{align}
    \mathfrak{D} = \{ \pi \in \mathfrak S_n: V(\pi(S_1)) \cap V(S_2) = \emptyset \} \,. \label{eq-def-mathfrak-D}
\end{align}

\begin{lemma}{\label{on_good_events}}
For $(S_1,S_2;T_1,T_2)\in \mathsf R_{\mathbf H,\mathbf I} \setminus \mathsf R_{\mathbf H,\mathbf I}^*$, we have 
\[
\mathbb{E}_{\Pb} \big[ \phi_{S_1,S_2}\phi_{T_1,T_2} \big] = \mathbb{E}_{\Pb} \big[ \phi_{S_1,S_2}\phi_{T_1,T_2} \mathbf{1}_{ \{ \pi_* \in \mathfrak{G} \cup \mathfrak D \} } \big]
\]
\end{lemma}

\begin{proof} 
Note that when $(S_1,S_2)=(T_1,T_2)$, we have $\mathfrak{G}= \{ \pi\in\mathfrak S_n: |\operatorname{L}_4| \geq |E(\operatorname{K}_4)|+1 \} = \{ \pi\in\mathfrak S_n: V(\pi(S_1)) \cap V(S_2) \neq \emptyset \} = \mathfrak S_n \setminus \mathfrak D$. Thus, (again similar to Lemma~\ref{lemma_observation_appendix_C} (2)) it suffices to show that when $(S_1,S_2) \neq (T_1,T_2)$ we have $\mathcal L (\operatorname{K}_{1} \cup \operatorname{K}' ) \neq\emptyset$ for all $\pi \in (\mathfrak{G} \cup \mathfrak{D})^c$ and for all $\operatorname{K}' \Subset \operatorname{K}_2 \cup \operatorname{K}_3 \cup \operatorname{K}_4$. First, we have
\[
4|\operatorname{L}_4| + 3|\operatorname{L}_3| + 2|\operatorname{L}_2| + |\operatorname{L}_1| = 4\aleph + 4 = 4|E(\operatorname{K}_4 )| + 3|E(\operatorname{K}_3 )| + 2|E(\operatorname{K}_2 )| + |E(\operatorname{K}_1 )| + 4 \,,
\]
and thus $\pi \in \mathfrak{G}$ is equivalent to
\[
\sum_{s=1}^{4} |\operatorname{L}_s| - \sum_{s=1}^{4} |E(\operatorname{K}_s )| \le 1 \,.
\]
Define the union graph $G_{\cup} \defby S_1 \cup T_1 \cup \pi^{-1}(S_2 \cup T_2)$. Then $\pi \in \mathfrak G$ is further equivalent to
\begin{equation}\label{eq-leq-tree}
    |V(G_{\cup})| \le |E(G_{\cup})| + 1 \,.
\end{equation}
Now suppose $(S_1,S_2) \neq (T_1 ,T_2)$ and $\pi \in (\mathfrak{G} \cup \mathfrak{D})^c$. Since \eqref{eq-leq-tree} does not hold in this case, we immediately have that $G_{\cup}$ contains at least two connected components. Now we proceed to show that $\mathcal{L}(\operatorname{K}_1\cup \operatorname{K}') \neq \emptyset$. We first deal with the case that exactly one of $V(S_1) \cap V(T_1)$ and $V(S_2) \cap V(T_2)$ is not empty. Assuming $V(S_2) \cap V(T_2) \ne \emptyset$, we have that $\pi^{-1}(S_2 \cup T_2)$ is contained in one of the connected components (in $G_{\cup}$). Since $G_{\cup}$ contains at least two connected components, we have that either $S_1$ or $T_1$ is not connected to $\pi^{-1}(S_2 \cup T_2)$. We may assume that $S_1$ is not connected to $\pi^{-1}(S_2 \cup T_2)$. Recalling that we have also assumed that $V(S_1) \cap V(T_1) = \emptyset$, we have $S_1 \subset \operatorname{K}_1$ and $S_1$ is one of the connected components in $G_{\cup}$, and therefore $\emptyset \ne \mathcal L (S_1) \subset \mathcal L(\operatorname{K}_1 \cup \operatorname{K}')$ for all $\operatorname{K}' \Subset \operatorname{K}_2 \cup \operatorname{K}_3 \cup \operatorname{K}_4$. 

Recall Definition (10) and Definition (11). By our assumption that $(S_1,S_2;T_1,T_2) \in \mathsf R_{\mathbf H,\mathbf I} \setminus \mathsf R^{*}_{\mathbf H,\mathbf I}$, the only remaining case is $V(S_1) \cap V(T_1) \ne \emptyset$ and $V(S_2) \cap V(T_2) \ne \emptyset$. In this case $S_1 \cup T_1$ and $S_2 \cup T_2$ are both connected. Since for $\pi \not\in \mathfrak G $ we have shown that $G_{\cup}$ has at least two connected components, we thus see that $S_1 \cup T_1$ and $\pi^{-1}(S_2 \cup T_2)$ are two distinct connected components. Thus,
\begin{align}
    |V(G_{\cup })| &= |V(S_1 \cup T_1 )| + |V(S_2 \cup T_2)| \nonumber \\ 
    &\le |E(S_1 \cup T_1 )| + |E(S_2 \cup T_2)| + 2 = |E(G_{\cup } )| + 2 \,. \label{eq-nonprincipal-smallcase}
\end{align}
Since \eqref{eq-leq-tree} does not hold in this case either, we have that in fact $|V(G_\cup)| = |E(G_\cup)|+2$, showing that $S_1 \cup T_1$ and $\pi^{-1} (S_2 \cup T_2 )$ must be vertex-disjoint trees. By $(S_1,S_2) \neq (T_1 ,T_2)$, one of the forests $F_1 = S_1 \doublesetminus T_1 \subset \operatorname{K}_1$ and $F_2=\pi^{-1}(S_2\doublesetminus T_2)\subset \operatorname{K}_1$ is not empty. We may assume that $F_1 = S_1 \doublesetminus T_1 = (S_1 \cup T_1) \doublesetminus T_1$ is not empty. Combined with the fact that $S_1 \cup T_1$ is a tree, we know that $\emptyset \neq \mathcal{L}( F_1 ) \subset \mathcal{L}(\operatorname{K}_1\cup \operatorname{K}')$ for all $\operatorname{K}'\Subset \operatorname{K}_2 \cup \operatorname{K}_3 \cup \operatorname{K}_4$ by the same arguments in Lemma~\ref{lemma_observation_appendix_C} (ii), which completes the proof of this lemma. 
\end{proof}

\begin{lemma}{\label{lemma_order_of_growth_second}}
Define $\operatorname{Overlap}^*_m = \{ \pi \in \mathfrak S_n : |V(S_1 \cup T_1 ) \cap (V( \pi^{-1}(S_2 \cup T_2) ))| = m \} $. Then we have $\mu (\operatorname{Overlap}^*_m) \le n^{-m+o(1)}$.
\end{lemma}
\begin{proof} 
Let $W = V(S_1 \cup T_1)$. Observe that if $|V(S_1 \cup T_1 ) \cap (V( \pi^{-1}(S_2 \cup T_2) ))| = m$, then the enumeration of $(\pi(v))_{v \in W}$ is bounded by $(2\aleph)^m n^{|W|-m} \le (2\aleph)^{2\aleph} n^{|W|-m}$. It directly follows that
\begin{equation*}
    \mu (\operatorname{Overlap}^*_m) \le \frac{(2\aleph)^{2\aleph} n^{|W|-m}}{(n-2\aleph)^{|W|}} \le n^{-m+o(1)}\,. \qedhere
\end{equation*}
\end{proof}
We now finish the proof of Lemma~\ref{prop_principal}.

\begin{proof}[Proof of Lemma~\ref{prop_principal}]
We first prove Item (i). Suppose $(S_1,S_2;T_1,T_2)\in \mathsf R^*_{\mathbf H,\mathbf I}$. Using Lemma \ref{principal-decomposition}, we have
\begin{align}
    \mathbb{E}_{\Pb}[\phi_{S_1,S_2}\phi_{T_1,T_2}] &= \mathbb{E}_{\pi\sim\mu}\Big[ \mathbb{E}_{\Pb_{\pi}} \big[ \phi_{S_1,S_2}\phi_{T_1,T_2} \textbf{1}_{\{ \pi \in \mathcal M \} } \big] \Big] \label{eq-C.113-part-1} \\
    &+ \mathbb{E}_{\pi\sim\mu}\Big[ \mathbb{E}_{\Pb_{\pi}} \big[ \phi_{S_1,S_2}\phi_{T_1,T_2} \textbf{1}_{ \{ \pi\in \mathfrak A ' \backslash \mathcal{M} \} } \big] \Big] \,. \label{eq-C.113-part-2}
\end{align}
Using Lemma~\ref{up_to_second_moment_AB} and recalling Definition (3), we have that \eqref{eq-C.113-part-1} is bounded by $[1+o(1)]$ times 
\begin{align}
    &\mathbb E_{\pi\sim\mu}\Bigg\{\mathbf{1}_{\pi(S_1\cup T_1 )=S_2\cup T_2}\mathbb E_{\sigma\sim\nu}\Big[\prod_{(i,j)\in E(S_1)\cup E(T_1)}\frac{v_{1,1}+u_{1,1}\omega(\sigma_i,\sigma_j)}{\lambda s}\Big]\Bigg\} \nonumber \\
    =&s^{2\aleph}\mu( \pi(S_1 \cup T_1 )=S_2 \cup T_2 ) \circeq \mathbb{E}_{\Pb} [\phi_{S_1 ,S_2 } ] \mathbb{E}_{\Pb} [\phi_{T_1 ,T_2 } ] (1 + \mathbf{1}_{ \{S_1 \cong T_1\} })\,, \label{eq-C.114}
\end{align}
where the first equality follows from the fact that $S_1,T_1$ are disjoint trees, Lemma~\ref{up_to_second_moment_AB} and \eqref{eq-prod-edges-expectation-tree}, and the second equality is from Lemma 3.3 and the fact that (since $(S_1,S_2;T_1,T_2)\in\mathsf R_{\mathbf H,\mathbf I}^*$)
\begin{equation}
    \frac{\mu( \pi(S_1 \cup T_1 )=S_2 \cup T_2 )}{\mu(\pi(S_1)=S_2)\mu(\pi(T_1)=T_2)} \circeq 1+\mathbf{1}_{S_1\cong T_1} \,.\notag
\end{equation} 
In addition, we have that \eqref{eq-C.113-part-2} is bounded by $[1+o(1)]$ times (writing $\doublesymdiff_E = E((S_1 \cup T_1)\doublesymdiff \pi^{-1}(S_2 \cup T_2))$ and $\cap_E = E((S_1 \cup T_1) \cap \pi^{-1}(S_2 \cup T_2))$)
\begin{align}
    & \mathbb E_{\pi\sim\mu}\Bigg\{\mathbf{1}_{\{\pi\in \mathfrak{A}' \backslash \mathcal M \}} \big(\tfrac{\epsilon^2 \lambda s}{n}\big)^{\frac{1}{2} |\doublesymdiff_E| } * \mathbb E_{\sigma\sim \nu}\Big[  \prod_{(i,j)\in \doublesymdiff_E} \omega(\sigma_i,\sigma_j) \prod_{(i,j)\in\cap_E} \tfrac{v_{1,1}+u_{1,1}\omega(\sigma_i,\sigma_j)}{\lambda s}\Big] \Bigg\}\nonumber\\
    \leq\ & [1+o(1)] \sum_{m=1}^{2\aleph} \mathbb E_{\pi\sim\mu} \Bigg\{ \textbf{1}_{ \{ \pi \in \operatorname{Overlap}'_m \} } \cdot s^{2\aleph-m} \big(\tfrac{\epsilon^2\lambda s}{n}\big)^{m} \mathbb E_{ \sigma\sim\nu } \Big[ \prod_{(i,j) \in \doublesymdiff_E} |\omega(\sigma_i,\sigma_j) |\nonumber \\
    & *  \prod_{(i,j) \in \cap_E} \big| 1+\epsilon \omega(\sigma_i,\sigma_j) \big| \Big] \Bigg\} 
    \leq \sum_{m=1}^{2\aleph} \mu(\operatorname{Overlap}'_m) s^{2\aleph-m} k^{4\aleph} (\epsilon^2\lambda s/n)^m \,, \label{eq-C.116}
\end{align}
where in the last inequality we used $|\omega(\sigma_i,\sigma_j)| \leq k-1$ for all $i,j \in [n]$. By Lemma~\ref{principal-enumeration-pi}, we see that
\begin{align}
    \eqref{eq-C.116} &\ \leq n^{-0.5}\sum_{m=1}^{2\aleph}k^{4\aleph}s^{2\aleph}(\epsilon^2\lambda)^{m} \mu(\operatorname{Overlap}'_0) \nonumber \\
    &\overset{(5)}{\leq} s^{2\aleph} n^{-0.4} \mu(\operatorname{Overlap}'_0) = o(1) \cdot s^{2\aleph} \mu(\pi(S_1 \cup T_1 )=S_2 \cup T_2 ) \,. \label{eq-C.117}
\end{align}
Combining \eqref{eq-C.114}, \eqref{eq-C.116} and \eqref{eq-C.117}, we obtain that
$$
\mathbb{E}_{\Pb}[ \phi_{S_1,S_2}\phi_{T_1,T_2} ] \circeq s^{2\aleph} \mu(\pi(S_1 \cup T_1 )=S_2 \cup T_2 ) \,,
$$
which completes our proof of Item (i).

For Item (ii), by Lemma~\ref{on_good_events} we know
\begin{align*}
    \mathbb{E}_\Pb[\phi_{S_1,S_2}\phi_{T_1,T_2} ] = \mathbb{E}_{\Pb}\big[ \phi_{S_1,S_2} \phi_{T_1,T_2} \mathbf{1}_{\{ \pi_* \in \mathfrak G \setminus \mathfrak D \} } \big] +  \mathbb{E}_{\Pb}\big[\phi_{S_1,S_2} \phi_{T_1,T_2} \mathbf{1}_{ \{ \pi_* \in \mathfrak D \} } \big] \,.
\end{align*}
Define $\kappa=(\kappa_{i,j})_{0\leq i,j\leq 2}$ and $\ell=(\ell_{i,j})_{0\leq i,j\leq 2}$, and define $\Pi_{\kappa,\ell}$ to be the subset of $\mathfrak{S}_n$ such that the $(|E(\operatorname{K}_{i,j} )|)_{0 \leq i,j \leq 2}=\kappa$ and $(|\operatorname{L}_{i,j}|)_{0 \leq i,j \leq 2}=\ell$. Then, using Lemma~\ref{up_to_second_moment_AB} and $|E(\operatorname{K}_1 )| + 2|E(\operatorname{K}_2 )| + 3|E(\operatorname{K}_3 )| + 4|E(\operatorname{K}_4 )| = 4\aleph$, we know that for $\pi \in (\mathfrak G \setminus \mathfrak D) \cap \Pi_{\kappa ,\ell}$
\begin{align*}
    \Big| \mathbb{E}_{\Pb_{\pi}} \big[ \phi_{S_1,S_2} \phi_{T_1,T_2} \big] \Big| &\leq \big( \tfrac{\lambda s}{n} \big)^{-2\aleph} \mathbb{E}_{\sigma \sim \nu} \Big[ \prod_{0\le y,z \le 2} \prod_{(i,j) \in E(\operatorname{K}_{yz})} \frac{|u_{y,z} \omega(\sigma_i , \sigma_j ) + v_{y,z}|}{n} \Big] \\
    &\leq n^{\kappa_{22} +0.5\kappa_{21} + 0.5\kappa_{12} -0.5\kappa_{01} - 0.5\kappa_{10}}L^{4\aleph} \,,
\end{align*}
where $L=\frac{1}{s}(1+\lambda)^2 \max_{0\leq r,t\leq 2}(1+(k-1) |u_{r,t}|+|v_{r,t}|)$. Thus we have
\begin{align}
    & \Big| \mathbb{E}_\Pb \big[ \phi_{S_1,S_2}\phi_{T_1,T_2} \mathbf{1}_{\{ \pi_* \in \mathfrak G \setminus \mathfrak D \} } \big] \Big| = \Big| \mathbb E_{\pi\sim\mu} \Big[ \sum_{\kappa,\ell} \mathbb{E}_{\Pb_{\pi}}\big[ \phi_{S_1,S_2} \phi_{T_1,T_2} \big] \mathbf{1}_{\{ \pi \in (\mathfrak G 
    \setminus \mathfrak D) \cap \Pi_{\kappa ,\ell} \} }  \Big] \Big| \nonumber \\
    \le \ & L^{4\aleph} \sum_{\kappa,\ell} n^{\kappa_{22} +0.5\kappa_{21} + 0.5\kappa_{12} -0.5\kappa_{01} - 0.5\kappa_{10}} \mu\big( \Pi_{\kappa,\ell} \cap (\mathfrak G \setminus \mathfrak D) \big) \,. \label{second-moment-B-setminus-C-1}
\end{align}
Recall \eqref{eq-def-frak-G}. Defining $\mathcal U=\{ (\kappa,\ell): 2\ell_4+\ell_3-\ell_1 \geq 2\kappa_4+\kappa_3-\kappa_1+2 \}$, we have
\begin{align}
    \eqref{second-moment-B-setminus-C-1} &\leq L^{4\aleph} \sum_{(\kappa,\ell) \in \mathcal U} n^{\kappa_{22} +0.5\kappa_{21} + 0.5\kappa_{12} -0.5\kappa_{01} - 0.5\kappa_{10}} \mu \big( \operatorname{Overlap}^*_{\ell_{11} + \ell_{12} + \ell_{21} + \ell_{22}} \big)  \nonumber \\
    &\leq L^{4\aleph} \sum_{(\kappa,\ell) \in \mathcal U} n^{\kappa_{22} +0.5\kappa_{21} + 0.5\kappa_{12} -0.5\kappa_{01} - 0.5\kappa_{10} - \ell_{11} - \ell_{12} - \ell_{21} - \ell_{22} + 0.1} \nonumber \\ 
    &\leq L^{4\aleph} \sum_{\kappa,\ell} n^{\ell_{22} +0.5\ell_{21} + 0.5\ell_{12} -0.5\ell_{01} - 0.5\ell_{10} - \ell_{11} - \ell_{12} - \ell_{21} - \ell_{22} - 0.9} \nonumber \\
    &= L^{4\aleph} \sum_{\kappa , \ell} n^{-0.5\ell_{01} - 0.5\ell_{10} - \ell_{11} - 0.5\ell_{12} - 0.5\ell_{21} - 0.9 } \,, \label{second-moment-B-setminus-C-2}
\end{align}
where the second inequality follows from Lemma~\ref{lemma_order_of_growth_second} and the third inequality is from the definition of $\mathcal{U}$.
Since $|V(S_1) \triangle V(T_1)| = \ell_{10} + \ell_{11} + \ell_{12}$ and $|V(S_2) \triangle V(T_2)| = \ell_{01} + \ell_{11} + \ell_{21}$, we have
\begin{align}
    \eqref{second-moment-B-setminus-C-2} \le L^{4\aleph} (16\aleph(\aleph+1))^{6} n^{-0.5(|V(S_1 )\triangle V(T_1 )| + |V(S_2 )\triangle V(T_2 )|) - 0.9 } \,. \label{second-moment-B-setminus-C-3}
\end{align}
For sufficiently large $n$ we have $((1+\lambda)^2 L)^{4\aleph} (16\aleph(\aleph+1))^{6} \le n^{0.1}$. Thus, combining \eqref{second-moment-B-setminus-C-1}, \eqref{second-moment-B-setminus-C-2} and \eqref{second-moment-B-setminus-C-3}, we have
\begin{equation}\label{phiS-1B-backslash-C}
    \big| \mathbb{E}_\Pb[\phi_{S_1 , S_2}\phi_{T_1 , T_2} \mathbf{1}_{\{ \pi_* \in \mathfrak G  \setminus \mathfrak D \} } ] \big| \le [1+o(1)] \cdot n^{-0.5(|V(S_1)\triangle V(T_1)| + |V(S_2)\triangle V(T_2)|) - 0.8 } \,.
\end{equation}
We now treat the term $\mathbb{E}_\Pb[\phi_{S_1,S_2} \phi_{T_1,T_2} \mathbf{1}_{\{ \pi_* \in \mathfrak D \}}]$ in the case of $(S_1,S_2)=(T_1,T_2)$. Note that for sufficiently large $n$ we have $\mathbb{E}_{\Pb_{\sigma,\pi}} \big[ \big(A_{i,j} - \tfrac{\lambda s}{n} \big)^2 \big]\le \frac{\sqrt{h}\lambda s(1+\epsilon\omega(\sigma_i,\sigma_j))}{n}$ by the fact that (recall our assumption that $h>1$)
\begin{align*}
    \mathbb{E}_{\Pb_{\sigma,\pi}} \Big[ \big( A_{i,j} - \tfrac{\lambda s}{n} \big)^2 \Big] = \big( \tfrac{\lambda s}{n} \big)^2 \big( 1 - \tfrac{\lambda s(1+\epsilon\omega(\sigma_i,\sigma_j))}{n} \big) + \big( 1 - \tfrac{\lambda s}{n}\big)^2 \tfrac{\lambda s(1+\epsilon \omega (\sigma_i ,\sigma_j))}{n} \,.
\end{align*}
By independence of edges under $\mathbb P_{\sigma,\pi}$ we have for $\pi\in\mathfrak{D}$
\begin{align*}
    \mathbb E_{\mathbb P_{\sigma,\pi}}[\phi_{S_1,S_2}^2] &= \prod_{(i,j) \in E(S_1)} \mathbb{E}_{\Pb_{\sigma, \pi}} \big[ (A_{i,j}-\tfrac{\lambda s}{n})^2 \big] \prod_{(i',j') \in E(S_2)} \mathbb{E}_{\Pb_{\sigma,\pi}} \big[ (B_{i',j'}-\tfrac{\lambda s}{n})^2 \big]  \\
    &\leq h^{\aleph}\prod_{(i,j) \in E(S_1  \cup \pi^{-1}(S_2))} \tfrac{\lambda s(1+\epsilon\omega (\sigma_i,\sigma_j))}{n} \,.
\end{align*}
Thus, we get that when $(S_1,S_2)=(T_1,T_2)$
\begin{align*}
    &\mathbb{E}_\Pb\big[ \phi_{S_1,S_2}\phi_{T_1,T_2} \mathbf{1}_{ \{ \pi_* \in \mathfrak D \}} \big] 
    = \mathbb{E}_{\pi\sim\mu,\sigma\sim\nu}\Big[ \mathbf{1}_{ \{ \pi \in \mathfrak D \}} \cdot \mathbb{E}_{\Pb_{\pi,\sigma}}[\phi_{S_1 ,S_2}^2 ] \Big]  \\
    \le\ & \Big( \tfrac{\lambda s}{n\sqrt{h}}(1-\tfrac{\lambda s}{n}) \Big)^{-2\aleph} \mathbb{E}_{\pi\sim\mu,\sigma\sim\nu} \Bigg[ \mathbf{1}_{ \{ \pi \in \mathfrak D \}} \prod_{(i,j) \in E(S_1  \cup \pi^{-1}(S_2))} \tfrac{\lambda s(1+\epsilon\omega (\sigma_i,\sigma_j))}{n}  \Bigg] \\
    \leq\ & \big( \tfrac{1}{\sqrt{h}}(1-\tfrac{\lambda s}{n}) \big)^{-2\aleph} \,,
\end{align*}
where the last equality follows from the fact that $S_1 \cup \pi^{-1}(S_2)$ is a forest for $\pi\in\mathfrak{D}$ and \eqref{eq-prod-edges-expectation-tree}. Since $1-\frac{\lambda s}{n} \ge \frac{1}{\sqrt{h}}$, we know
\begin{equation}\label{phiS-1C}
    \mathbb{E}_\Pb \big[ \phi_{S_1,S_2} \phi_{T_1,T_2} \mathbf{1}_{ \{ \pi_{*} \in \mathfrak D \}} \big] \le O_h(1) \cdot h^{2\aleph}.
\end{equation}
When $(S_1,S_2) \ne (T_1,T_2)$, we recall from \eqref{eq-def-mathfrak-D} that $\mathfrak D=\emptyset$, which gives 
$$
    \mathbb{E}_\Pb[\phi_{S_1 ,S_2}\phi_{T_1 ,T_2} \mathbf{1}_{ \{ \pi_* \in \mathfrak D} \}] = 0 \,.
$$
Therefore, combining \eqref{phiS-1B-backslash-C} and \eqref{phiS-1C}, we finish the proof.
\end{proof}

\subsection{Proof of Item (ii) of Proposition~\ref{prop-first-second-moment-f-T}}{\label{subsec:proof-lem-3.2(ii)}}

we use Lemma~\ref{prop_principal} to estimate $\operatorname{Var}_{\Pb} [f_{\mathcal{T}}]$. Recall that
\begin{align}
    \operatorname{Var}_{\Pb} [f_{\mathcal{T}}] &= \sum_{\mathbf H,\mathbf I \in \mathcal{T}} \sum_{(S_1,S_2;T_1,T_2) \in \mathsf R_{\mathbf H,\mathbf I}} a_{\mathbf H} a_{\mathbf I} (\mathbb{E}_{\Pb} [\phi_{S_1,S_2} \phi_{T_1,T_2}] - \mathbb{E}_{\Pb}[\phi_{S_1,S_2}] \mathbb{E}_\Pb[\phi_{T_1,T_2}]) \nonumber \\
    &= \sum_{\mathbf H,\mathbf I \in \mathcal{T}} \sum_{(S_1,S_2;T_1,T_2) \in \mathsf R_{\mathbf H,\mathbf I}^*} a_{\mathbf H} a_{\mathbf I} \mathbb{E}_\Pb [\phi_{S_1,S_2} \phi_{T_1,T_2}] \label{eq-part-1-var-Pb} \\
    &+ \sum_{\mathbf H,\mathbf I \in \mathcal{T}} \sum_{(S_1,S_2;T_1,T_2) \in \mathsf R_{\mathbf H,\mathbf I} \setminus \mathsf R_{\mathbf H,\mathbf I}^*} a_{\mathbf H} a_{\mathbf I} \mathbb{E}_\Pb [\phi_{S_1,S_2} \phi_{T_1,T_2}] \label{eq-part-2-var-Pb} \\
    &- \sum_{\mathbf H,\mathbf I \in \mathcal{T}} \sum_{(S_1,S_2;T_1,T_2) \in \mathsf R_{\mathbf H,\mathbf I}} a_{\mathbf H} a_{\mathbf I} \mathbb{E}_\Pb [\phi_{S_1,S_2}] \mathbb{E}_\Pb [\phi_{T_1,T_2}] \,. \label{eq-part-3-var-Pb}
\end{align}
By Item (ii) in Lemma~\ref{prop_principal}, we have
\begin{align}
    \eqref{eq-part-1-var-Pb} \circeq\ & \sum_{\mathbf H,\mathbf I \in \mathcal{T}} \sum_{(S_1,S_2;T_1,T_2) \in \mathsf R_{\mathbf H,\mathbf I}^*} a_{\mathbf H} a_{\mathbf I} \mathbb{E}_\Pb [\phi_{S_1,S_2}] \mathbb{E}_\Pb [\phi_{T_1,T_2}] \label{eq-part-1.1-var-Pb} \\
    &+ \sum_{\mathbf H \in \mathcal{T}} \sum_{(S_1,S_2;T_1,T_2) \in \mathsf R^*_{\mathbf H,\mathbf H}} a_{\mathbf H}^2 \mathbb{E}_\Pb [\phi_{S_1,S_2}] \mathbb{E}_\Pb [\phi_{T_1,T_2}] \,. \label{eq-part-1.2-var-Pb}
\end{align}
We first deal with \eqref{eq-part-1.1-var-Pb} using the fact that
\begin{equation}{\label{eq-card-R-H,I^*}}
    \#\mathsf R^*_{\mathbf H,\mathbf I}= \frac{ n! }{ (n-2\aleph-2)! \operatorname{Aut}(\mathbf H)^2 } \cdot \frac{ n! }{ (n-2\aleph-2)! \operatorname{Aut}(\mathbf I)^2 } \,.
\end{equation}
Applying Lemma~\ref{prop_first_moment_phi} we have
\begin{align}
    \eqref{eq-part-1.1-var-Pb} & \overset{(9)}{\circeq} \sum_{\mathbf H,\mathbf I \in \mathcal{T}} a_{\mathbf H} a_{\mathbf I} \cdot \#\mathsf R^*_{\mathbf H,\mathbf I} \cdot \frac{s^{2\aleph} \operatorname{Aut}(\mathbf H) \operatorname{Aut}(\mathbf I) ((n-\aleph-1)!)^2}{(n!)^2} \nonumber \\
    & \overset{\eqref{eq-card-R-H,I^*}}{=}  \sum_{\mathbf H,\mathbf I \in \mathcal{T}} a_{\mathbf H} a_{\mathbf I}  \cdot \frac{s^{2\aleph} ((n-\aleph-1)!)^2}{((n-2\aleph-2)!)^2 \operatorname{Aut}(\mathbf H) \operatorname{Aut}(\mathbf I)} \nonumber \\
    & \overset{(8)}{=} \sum_{\mathbf H,\mathbf I \in \mathcal{T}} s^{4\aleph} \Big( \frac{((n-\aleph-1)!)^2}{n!(n-2\aleph-2)!} \Big)^2 \overset{(5)}{\circeq} s^{4\aleph} |\mathcal T|^2 \overset{\text{Lemma}~\ref{lem-first-moment-P}, \text{(ii)}}{\circeq} \mathbb{E}_{\Pb} [f_{\mathcal{T}}]^2 \,. \label{eq-bound-part-1.1-var-Pb}
\end{align}
Similarly, by applying Lemma~\ref{prop_first_moment_phi} we see that 
\begin{align}
    \eqref{eq-part-1.2-var-Pb} \circeq s^{4\aleph} |\mathcal{T}| = o(1) \cdot \mathbb{E}_{\Pb} [f_{\mathcal{T}}]^2 \,. \label{eq-bound-part-1.2-var-Pb}
\end{align}
Thus, we get that 
\begin{equation}\label{eq-principal-term-sum-in-second-moment}
    \eqref{eq-part-1-var-Pb} \circeq \mathbb{E}_{\Pb} [f_{\mathcal{T}}]^2\,.
\end{equation}
 Next we estimate \eqref{eq-part-3-var-Pb} (for convenience below we use \eqref{eq-part-3-var-Pb} to denote the term therein without the minus sign). Note that
\begin{equation}{\label{eq-card-R-H,I}}
    \#\mathsf R_{\mathbf H,\mathbf I}= \frac{ (n!)^2 }{ ((n-\aleph-1)!)^2 \operatorname{Aut}(\mathbf H)^2 } \cdot \frac{ (n!)^2 }{ ((n-\aleph-1)!)^2 \operatorname{Aut}(\mathbf I)^2 } \,.
\end{equation}
Combined with Lemma~\ref{prop_first_moment_phi}, it yields that
\begin{align}
    \eqref{eq-part-3-var-Pb} & \overset{\eqref{eq-first-moment-phi-S}}{\circeq} \sum_{\mathbf H,\mathbf I \in \mathcal{T}} a_{\mathbf H} a_{\mathbf I} \cdot \#\mathsf R_{\mathbf H,\mathbf I} \cdot \frac{s^{2\aleph} \operatorname{Aut}(\mathbf H) \operatorname{Aut}(\mathbf I) ((n-\aleph-1)!)^2}{(n!)^2} \nonumber \\
    & \overset{\eqref{eq-card-R-H,I}}{=} \sum_{\mathbf H,\mathbf I \in \mathcal{T}} a_{\mathbf H} a_{\mathbf I} \cdot \frac{s^{2\aleph} (n!)^2 }{ ((n-\aleph-1)!)^2 \operatorname{Aut}(\mathbf H) \operatorname{Aut}(\mathbf I)} \nonumber \\
    & \overset{\eqref{eq-def-a-H}}{\circeq} \sum_{\mathbf H,\mathbf I \in \mathcal{T}} s^{4\aleph} = s^{4\aleph} |\mathcal T|^2 \overset{\text{Lemma}~\ref{lem-first-moment-P}, \text{(ii)}}{\circeq} \mathbb{E}_{\Pb} [f_{\mathcal{T}}]^2 \,. \label{eq-bound-part-3-var-Pb}
\end{align}
Finally we deal with \eqref{eq-part-2-var-Pb}. Denote $\gamma_s=\big( \frac{1}{2}+\frac{s^2}{2\alpha} \big)^{1/2}>1$. Applying Item (i) of Lemma~\ref{prop_principal} with $h=\gamma_s$ we have $|\eqref{eq-part-2-var-Pb}|$ is bounded by
\begin{align}
     & [1+o(1)] \cdot \Bigg( O_{\gamma_s}(1) \cdot \sum_{\mathbf H \in \mathcal{T}} \sum_{(S_1, S_2 ; T_1, T_2) \in \mathsf R_{\mathbf H,\mathbf H} \setminus \mathsf R_{\mathbf H,\mathbf H}^* } a_{\mathbf H}^2 \mathbf{1}_{\{ (S_1,S_2)=(T_1,T_2) \} }\gamma_s^{2\aleph} \nonumber \\
    & + \sum_{\mathbf H,\mathbf I \in \mathcal{T}} \sum_{(S_1, S_2 ; T_1, T_2) \in \mathsf R_{\mathbf H,\mathbf I} \setminus \mathsf R_{\mathbf H,\mathbf I}^* } a_{\mathbf H} a_{\mathbf I}  n^{-0.5(|V(S_1) \triangle V(T_1)| + |V(S_2)\triangle V(T_2)|) - 0.8}  \Bigg) \nonumber \\
    \le\ & [1+o(1)] \cdot \Bigg( O_{\gamma_s}(1) \cdot \sum_{\mathbf H \in \mathcal{T}} \frac{a_{\mathbf H}^2 \gamma_{s}^{2\aleph}(n!)^2}{((n-\aleph-1)!\operatorname{Aut(\mathbf{H}))^2}}
    + \sum_{\mathbf H,\mathbf I \in \mathcal{T}} a_{\mathbf H} a_{\mathbf I} \sum_{\substack{i \neq \aleph + 1 \mbox{ or }\\j\neq \aleph + 1}} \frac{ |\mathsf R_{\mathbf H,\mathbf I}^{(i,j)}| }{ n^{i+j+0.8} } \Bigg) \,, \label{eq-part-2-var-Pb-relax}
\end{align}
where $\mathsf R_{\mathbf H,\mathbf I}^{(i,j)}$ is the collection of $(S_1,S_2;T_1,T_2) \in \mathsf R_{\mathbf H,\mathbf I}$ such that $|V(S_1) \cap V(T_1)| = \aleph+1-i$ and $|V(S_2) \cap V(T_2)| = \aleph+1-j$.
We know from direct enumeration that
\begin{align*}
    \#\mathsf R_{\mathbf H,\mathbf I}^{(i,j)} &= \frac{1}{(\operatorname{Aut}(\mathbf H) \operatorname{Aut}(\mathbf I))^2} \cdot \frac{n!(\aleph+1)!\binom{\aleph+1}{i}}{(n-\aleph-1-i)!i!} \cdot \frac{n!(\aleph+1)!\binom{\aleph+1}{j}}{(n-\aleph-1-j)!j!} \\
    & \le (2(\aleph+1))^{\aleph+1}n^{2\aleph+2+i+j} \overset{ (5)}{ \leq } n^{2\aleph+2.1+i+j} \,,
\end{align*}
and we have the bound (recall \eqref{eq-def-a-H})
\begin{align*}
    a_{\mathbf H} \le [1+o(1)] \cdot \frac{(\aleph+1)! s^\aleph}{n^{\aleph+1}} \overset{ \eqref{eq-choice-K} }{\le} [1+o(1)] \cdot n^{-\aleph-0.9} s^\aleph \,,
\end{align*}
where in the first inequality we used the crude bound that $\operatorname{Aut}(H)\leq (\aleph+1)!$ and in the last inequality we used $(\aleph+1)!=n^{o(1)}$. Hence, we have
\begin{align}
    & \sum_{\mathbf H, \mathbf I \in \mathcal{T}} a_{\mathbf H} a_{\mathbf I} \sum_{(i,j) \ne (\aleph+1, \aleph+1)} n^{-i-j-0.8} \big| \mathsf R_{\mathbf H,\mathbf I}^{(i,j)} \big| \le  \sum_{\mathbf H, \mathbf I \in \mathcal{T}} n^{-0.5} s^{2\aleph} \le |\mathcal{T}|^2 n^{-0.5} s^{4\aleph}2^{2\aleph} \nonumber \\
    \leq \ & [1+o(1)]\cdot n^{-0.5}2^{2\aleph} \mathbb{E}_{\Pb} [f_{\mathcal{T}}]^2 \overset{\eqref{eq-choice-K}}{=} o(1) \cdot \mathbb{E}_{\Pb} [f_{\mathcal{T}}]^2  \,, \label{eq-bound-part-2.1-var-Pb}
\end{align}
where in the second inequality we used the fact that $s^2 > \alpha > \frac{1}{4}$ and the third equality follows from Item (ii) of Lemma~\ref{lem-first-moment-P}. In addition, we have
\begin{align}
    &\sum_{\mathbf H \in \mathcal{T}} \frac{ a_{\mathbf H}^2 \gamma_{s}^{2\aleph} (n!)^2 }{((n-\aleph-1)!\operatorname{Aut}(\mathbf H))^2} \nonumber 
    \overset{\eqref{eq-choice-K}}{\le} \  [1+o(1)]\cdot |\mathcal{T}|s^{2\aleph} \gamma_{s}^{2\aleph} \\
    \overset{s^2>\alpha, \eqref{eq-num-trees}}{=} & 
    o(1) \cdot (|\mathcal{T}|s^{2\aleph})^2 = o(1) \cdot (\mathbb{E}_{\Pb} [f_{\mathcal{T}}])^2 \,. \label{eq-bound-part-2.2-var-Pb}
\end{align}
Combined with \eqref{eq-part-2-var-Pb-relax} and \eqref{eq-bound-part-2.1-var-Pb}, it gives that
\begin{equation}{\label{eq-bound-part-2-var-Pb}}
    \eqref{eq-part-2-var-Pb} = o(1) \cdot \mathbb{E}_{\Pb} [f_{\mathcal{T}}]^2 \,.
\end{equation}
Plugging \eqref{eq-principal-term-sum-in-second-moment}, \eqref{eq-bound-part-3-var-Pb} and \eqref{eq-bound-part-2-var-Pb} into the decomposition formula for $\operatorname{Var}_{\mathbb P}[f_\mathcal T]$, we get that
\begin{align*}
    \operatorname{Var}_{\Pb} [f_{\mathcal{T}}]
    &= [1+o(1)] \cdot \mathbb{E}_{\Pb} [f_{\mathcal{T}}]^2 + o(1) \cdot \mathbb{E}_{\Pb} [f_{\mathcal{T}}]^2 - [1+o(1)] \cdot \mathbb{E}_{\Pb} [f_{\mathcal{T}}]^2 + o(1) \cdot \mathbb{E}_{\Pb} [f_{\mathcal{T}}]^2 \\
    &= o(1) \cdot \mathbb{E}_{\Pb} [f_{\mathcal{T}}]^2 \,,
\end{align*}
which yields Item (ii) of Proposition~\ref{prop-first-second-moment-f-T}.

\section{Preliminaries on graphs}{\label{sec:prelim-graphs}}

\begin{lemma}{\label{lemma-facts-graphs}}
Let $S,T\subset \mathcal K_n$. Recall that $S \Cap T \Subset \mathcal{K}_n$ is defined as edge-induced subgraphs of $\mathcal{K}_n$. We have the following properties:
\begin{enumerate}
    \item[(i)] $|V(S\cup T)|+|V(S\Cap T)|\le|V(S)|+|V(T)|, |E(S\cup T)|+|E(S\Cap T)|=|E(S)|+|E(T)|$.
    \item[(ii)] $\tau(S \cup T) + \tau(S \Cap T) \geq \tau(S) + \tau(T)$ and $\Phi(S \cup T) \Phi(S \Cap T) \leq \Phi(S) \Phi(T)$.
    \item[(iii)] $|\mathtt C_j(S \cup T)| + |\mathtt C_j(S \cap T)| \geq |\mathtt C_j(S)| + |\mathtt C_j(T)|$. 
    \item[(iv)] Recall the notion of self-bad in Definition~\ref{def-addmisible}. If $S \subset T$, $S$ is self-bad and $V(S)=V(T)$, then $T$ is self-bad.
    \item[(v)] If $S$ and $T$ are both self-bad, then $S \cup T$ is self-bad.
\end{enumerate}
\end{lemma}

{\begin{proof}
    By definition, we have $V(S\cup T)=V(S)\cup V(T)$, $E(S\cup T)=E(S)\cup E(T)$ and $E(S\Cap T)=E(S) \cap E(T)$. In addition, we have $V(S\Cap T) \subset V(S) \cap V(T)$; this is because for any $i \in V(S \Cap T)$, there exists some $j$ such that $(i,j) \in E(S\Cap T)$ and thus $i \in V(S) \cap V(T)$. Therefore, (i) follows from the inclusion-exclusion formula.
    Provided with (i), (ii) follows directly from Equation~\eqref{eq-def-Phi}. 
    
    For (iii), since $|\mathtt C_j(S)| = \sum_{C \in \mathtt C_j(\mathcal K_n)} \mathbf{1}_{\{C \subset S\}}$, it suffices to show that
    \[
        \mathbf{1}_{\{C \subset S\}} + \mathbf{1}_{\{C \subset T\}} \leq \mathbf{1}_{\{C \subset S \cap T \}} + \mathbf{1}_{\{C \subset S \cup T\}} \,,
    \]
    which can be verified directly.
    
    For (iv), since clearly $T$ is bad, it remains to show that $\Phi(K) \geq \Phi(T)$ for all $K \subset T$. Denoting $V=V(K) \subset V(T)=V(S)$ and recalling the definition of $T_V,S_V$ in the notation section, we have
    \begin{align*}
        \Phi(K) &\geq \Phi(T_{V}) = \Phi(S_{V}) \cdot \big( \tfrac{ 1000 \Tilde{\lambda}^{20} k^{20} D^{50} }{ n } \big)^{ |E(T_{V})|-|E(S_{V})| } \\
        &\geq \Phi(S_{V}) \cdot \big( \tfrac{ 1000 \Tilde{\lambda}^{20} k^{20} D^{50} }{ n } \big)^{ |E(T)|-|E(S)| } \geq \Phi(S) \cdot \big( \tfrac{ 1000 \Tilde{\lambda}^{20} k^{20} D^{50} }{ n } \big)^{ |E(T)|-|E(S)| } = \Phi(T) \,,
    \end{align*}
    where the first inequality follows from $K \subset T_V$, the second inequality follows from the fact that $|E(T)|-|E(T_{V})| = |E(T_{\setminus V})|\geq |E(S_{\setminus V})| = |E(S)|-|E(S_{V})|$, the third inequality follows from the assumption that $S$ is self-bad, and the last equality follows from $V(S)=V(T)$. Thus, $T$ is also self-bad.

    Finally, for (v), first note that 
    \[
        \Phi(S \cup T) \leq \Phi(S) \Phi(T) / \Phi(S \Cap T) \leq \Phi(S) \,,
    \]
    where the first inequality follows from (ii) and the second inequality follows from the assumption that $T$ is self-bad. This implies that $S \cup T$ is bad. It remains to show that $\Phi(S \cup T) \leq \Phi(K)$ for all $K \subset S \cup T$. Applying (ii) and the assumption that $T$ is self-bad, we have
    \[
        \Phi(K \cup T) \leq \Phi(T) \Phi(K) / \Phi(K \Cap T) \leq \Phi(K) \,.
    \]
    Again, applying (ii) and the assumption that $S$ is self-bad, we have
    \[
        \Phi(S \cup T) \leq \Phi(S) \Phi(K \cup T) / \Phi(S \Cap (K \cup T)) \leq \Phi(K \cup T) \,.
    \]
    Combining the preceding two inequalities yields $\Phi(K) \geq \Phi(K \cup T) \geq \Phi(S \cup T)$.
\end{proof}

The next few lemmas prove some properties for subgraphs of $S$.

\begin{lemma}{\label{lem-property-H-Subset-S}}
    For $H \ltimes S$, we have $|\mathcal L(S) \setminus V(H)| \geq 2(\tau(H)-\tau(S))$. In particular, for $H \ltimes S$ such that $\mathcal L(S) \subset V(H)$, we have $\tau(H) \leq \tau(S)$. 
\end{lemma}
\begin{proof}
    Without loss of generality, we may assume that $S$ contains no isolated vertex. Clearly we have $V(S) \setminus V(H) \subset V(S \doublesetminus H)$ from $H \ltimes S$. We now construct a bipartite graph $(\mathbf{V}_1,\mathbf{V}_2,\mathbf{E})$ as follows: denote $\mathbf{V}_1 = V(S \doublesetminus H)$ and $\mathbf{V}_2 = E(S) \setminus E(H)$ (note that each vertex in $\mathbf V_2$ is an edge in the graph $S$) and connect $(v,u) \in \mathbf{V}_1 \times \mathbf{V}_2$ (that is, let $(v, u)\in \mathbf E$) if and only if $v$ is incident to the edge $u$. We derive the desired inequality by calculating $|\mathbf{E}|$ in two different ways. On the one hand, clearly each $u \in \mathbf{V}_2$ is connected to exactly two endpoints of $u$ and thus $|\mathbf{E}|=2|\mathbf{V}_2|=2(|E(S)|-|E(H)|)$. On the other hand, each $v \in \mathcal L(S) \setminus V(H)$ is connected to at least one element in $\mathbf{V}_2$, and each $v \in (V(S) \setminus V(H)) \setminus (\mathcal L(S) \setminus V(H))$ is connected to at least two elements in $\mathbf{V}_2$. Thus, we have
    \begin{align*}
        |\mathbf E| \geq 2 |V(S) \setminus V(H)| - |\mathcal L(S) \setminus V(H)| = 2(|V(S)|-|V(H)|) - |\mathcal L(S) \setminus V(H)| \,,
    \end{align*}
    which yields that $|\mathcal L(S) \setminus V(H)| \geq 2(\tau(H)-\tau(S))$ (recall that $|\mathbf E|=2(|E(S)|-|E(H)|)$). 
\end{proof}

\begin{lemma}{\label{lem-decomposition-H-Subset-S}}
    For $H \subset S$, we can decompose $E(S)\setminus E(H)$ into $\mathtt m$ cycles ${C}_{\mathtt 1}, \ldots, {C}_{\mathtt m}$ and $\mathtt t$ paths ${P}_{\mathtt 1}, \ldots, {P}_{\mathtt t}$ (with a slight abuse of notations, we will also let a path $P$ to denote a subgraph of $\mathcal K_n$) for some $\mathtt m, \mathtt t\geq 0$ such that the following hold.
    \begin{enumerate}
        \item[(i)] ${C}_{\mathtt 1}, \ldots, {C}_{\mathtt m}$ are vertex-disjoint (i.e., $V(C_{\mathtt i}) \cap V(C_{\mathtt j})= \emptyset$ for all $\mathtt i \neq \mathtt j$) and $V(C_{\mathtt i}) \cap V(H)=\emptyset$ for all $1\leq\mathtt i\leq \mathtt m$.
        \item[(ii)] $\operatorname{EndP}({P}_{\mathtt j}) \subset V(H) \cup (\cup_{\mathtt i=1}^{\mathtt m} V(C_{\mathtt i})) \cup (\cup_{\mathtt k=1}^{\mathtt j-1} V(P_{\mathtt k})) \cup \mathcal L(S)$ for $1 \leq \mathtt j \leq \mathtt t$.
        \item[(iii)] $\big( V(P_{\mathtt j}) \setminus \operatorname{EndP}(P_{\mathtt j}) \big) \cap \big( V(H) \cup (\cup_{\mathtt i=1}^{\mathtt m} V(C_{\mathtt i})) \cup (\cup_{\mathtt k=1}^{\mathtt j-1} V(P_{\mathtt k}) ) \cup \mathcal L (S) \big) = \emptyset$ for $\mathtt 1 \leq \mathtt j \leq \mathtt t$.
        \item[(iv)] $\mathtt t = |\mathcal L(S) \setminus V(H)|+\tau(S)-\tau(H)$.
    \end{enumerate}
\end{lemma}
\begin{proof}
    We prove our lemma when $\mathcal L(S) \subset V(H)$ first. If $S$ can be decomposed into connected components $S=S_1 \sqcup S_2 \ldots \sqcup S_{r}$ and $H \cap S_i = H_i$, then it suffices to show the results for each $(H_i,S_i)$ since $\mathcal L(S_i) \subset V(H_i)$. Thus, we may assume without loss of generality that $S$ is connected. We initialize $\mathcal P= \mathcal C=\emptyset$ and perform the following procedure until $\overline {\mathcal P\cup \mathcal C} :=(\cup_{P\in \mathcal P} P) \cup (\cup_{C\in \mathcal C} C)$ contains all edges in $E(S) \setminus E(H)$.

    \

    \noindent (a) As long as there exists a cycle $C$ such that $V(C) \subset V(S) \setminus (V(H) \cup V(\mathcal C))$, we update $\mathcal C$ by adding $C$ to it (here we slightly abuse the notation by $V(\mathcal C) = V(\cup_{C\in \mathcal C} C)$, and we will do the same for $\mathcal P$ and for $E$).

    \

    \noindent (b) After Step (a) is finished, as long as $E(S) \setminus E(H) \not\subset E(\overline{\mathcal{P} \cup \mathcal{C}})$, we may construct a path $P \subset S \doublesetminus (H \cup\overline{ \mathcal{P} \cup \mathcal{C}})$ such that $V(P)\cap( V(\mathcal P) \cup V(H) \cup V(\mathcal C)) = \operatorname{EndP}(P)$ as follows. First, 
    \begin{figure}[!ht]
        \centering
        \vspace{0cm}
        \includegraphics[height=5.5cm,width=12cm]{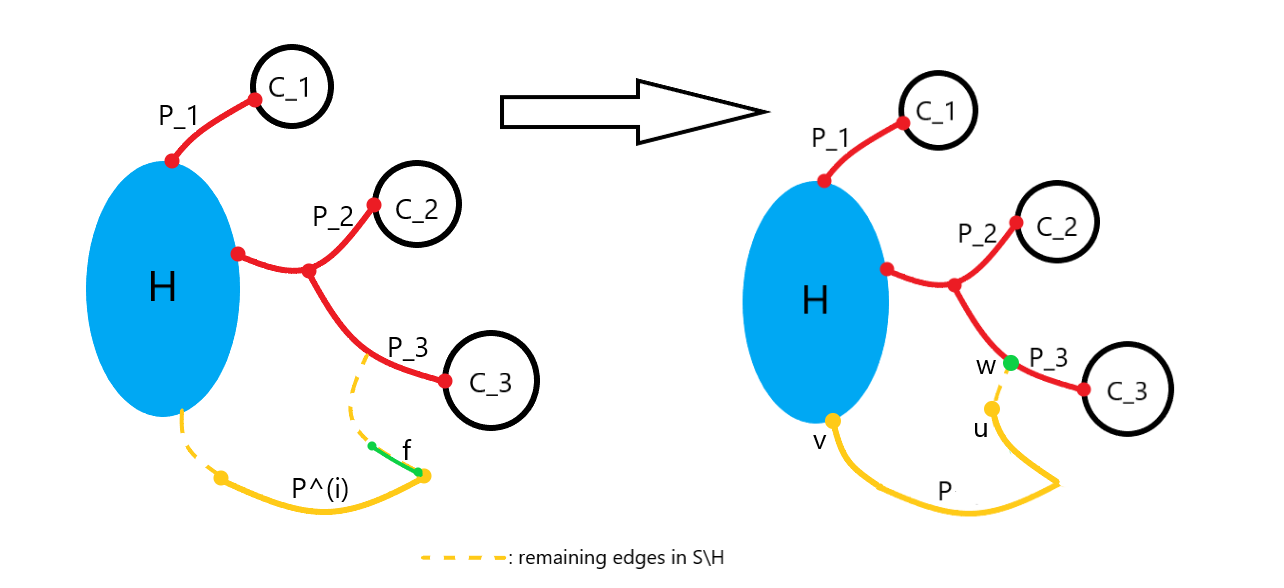}
        \caption{Construction of Paths}
        \label{fig:2}
    \end{figure}
    choose an arbitrary edge $\mathtt e = (\mathtt u_0 , \mathtt v_0 ) \in E(S \doublesetminus (H \cup \overline{\mathcal{P} \cup \mathcal{C}}))$. Then, starting from $P^{(1)} = \{ \mathtt e \}$, for $i\geq 1$ we replace $P^{(i)}$ by $P^{(i+1)} \defby P^{(i)} \cup \{ \mathtt f \}$ whenever there exists an edge $\mathtt f \in  E(S \doublesetminus (H \cup \overline{\mathcal{P} \cup \mathcal{C}}))$ incident to $\operatorname{EndP}(P^{(i)})$ such that $V(P^{(i)} \cup \{\mathtt f \})\cap( V(\mathcal P) \cup V(H) \cup V(\mathcal C)) \subset \operatorname{EndP}(P^{(i)} \cup \{ \mathtt f \} )$ (see the left-hand side of Figure~\ref{fig:2} for an illustration). Clearly this sub-procedure will stop at some point and we suppose that it yields a path $P$ with endpoints $u,v$. We claim $u,v\in V(\mathcal{P})\cup V(\mathcal{C}) \cup V(H)$. Since Step (a) was completed, $P$ is not a cycle disjoint with $V(H)\cup V(\mathcal C ) $. Thus, when $|\operatorname{EndP}(P)| = 1$ we have $\emptyset \ne V(P) \cap ( V(\mathcal P) \cup V(H) \cup V(\mathcal C)) \subset \operatorname{EndP}(P)$, and as a result $u = v \in V(\mathcal P) \cup V(H) \cup V(\mathcal C)$. When $|\operatorname{EndP}(P)| = 2$, we prove our claim by contradiction, for which we suppose $u\notin  V(\mathcal P) \cup V(H) \cup V(\mathcal C)$. Then we have $u\notin \mathcal{L}(S)$ since $\mathcal{L}(S)\subset V(H)$. Thus, there exists $w \in V(S)$ such that $(u,w) \in E(S) \setminus E(P)$. If $w \in V(P) \setminus \{ v \}$, by $V(P) \cap ( V(\mathcal P) \cup V(H) \cup V(\mathcal C)) \subset \operatorname{EndP}(P)$ we have $w \not\in V(\mathcal P) \cup V(H) \cup V(\mathcal C)$. In this case, $P$ contains a cycle disjoint with $V(H) \cup V(\mathcal C)$, which contradicts to (a). Thus, $w \not\in V(P) \setminus \{ v \}$. But in this case, the sub-procedure for producing $P$ would not have stopped at $u$ (as it should extend $w$ at least; see the right-hand side of Figure~\ref{fig:2} for an illustration); this implies that $u$ is not an endpoint of $P$, arriving at a contradiction. Therefore, we have that $u \in  V(\mathcal P) \cup V(H) \cup V(\mathcal C)$. By symmetry we know $u,v \in V(\mathcal P) \cup V(H) \cup V(\mathcal C)$. Hence
    $$ 
    V(P)\cap( V(\mathcal P) \cup V(H) \cup V(\mathcal C)) \subset \operatorname{EndP}(P) \subset V(P)\cap( V(\mathcal P) \cup V(H) \cup V(\mathcal C)) \,,
    $$
    which yields that $P$ satisfies our conditions. Therefore, we can update $\mathcal P$ by putting $P$ in it.
    
    When the procedure stops, we obtain the following:
    \begin{equation}
        \mathcal{C} = \{ {C}_1, \ldots, {C}_{\mathtt m} \} \mbox{ and } \mathcal{P}=\{ {P}_1, \ldots, {P}_{\mathtt t} \} \,.
    \end{equation}
    Now we verify this choice of $\mathcal C, \mathcal P$ satisfies (i)--(iv). (i), (ii) and (iii) are straightforward by our procedure. For (iv), note that we may track the update of $\tau(H \cup \overline{\mathcal P \cup \mathcal C})$ through our whole procedure when performing an update resulted from adding $C_{\mathtt i}$ or $P_{\mathtt j}$: $\tau(H \cup \overline{\mathcal P \cup \mathcal C})$ remains unchanged in each update from $C_{\mathtt i}$, and $\tau(H \cup \overline{\mathcal P \cup \mathcal C})$ increases by 1 in each update from $P_{\mathtt j}$. Therefore, the total increase of $\tau(H \cup \overline{\mathcal P \cup \mathcal C})$ through our whole procedure is $\mathtt t$, which proves (iv). 

    For general cases, we finish our proof by applying the preceding proof to $H_{\operatorname{leaf}} \subset S$ such that $E(H_{\operatorname{leaf}}) = E(H)$ and $V(H_{\operatorname{leaf}}) = V(H) \cup \mathcal L(S)$.
\end{proof}

\begin{cor}{\label{cor-revised-decomposition-H-Subset-S}}
    For $H \subset S$, we can decompose $E(S)\setminus E(H)$ into $\mathtt m$ cycles ${C}_{\mathtt 1}, \ldots, {C}_{\mathtt m}$ and $\mathtt t$ paths ${P}_{\mathtt 1}, \ldots, {P}_{\mathtt t}$ for some $\mathtt m, \mathtt t\geq 0$ such that the following hold.
    \begin{enumerate}
        \item[(i)] ${C}_{\mathtt 1}, \ldots, {C}_{\mathtt m}$ are independent cycles in $S$.
        \item[(ii)] $V(P_{\mathtt j}) \cap \big( V(H) \cup (\cup_{\mathtt i=1}^{\mathtt m} V(C_{\mathtt i})) \cup (\cup_{\mathtt k \neq \mathtt j} V(P_{\mathtt k}) ) \cup \mathcal L (S) \big) = \operatorname{EndP}(P_{\mathtt j})$ for $1 \leq \mathtt j \leq \mathtt t$.
        \item[(iii)] $\mathtt t \leq 5(|\mathcal L (S) \setminus V(H) | + \tau(S)-\tau(H))$.
    \end{enumerate}
\end{cor}
\begin{proof}
    We prove our corollary when $\mathcal L(S) \subset V(H)$ first. Using Lemma~\ref{lem-decomposition-H-Subset-S}, we can decompose $E(S) \setminus E(H)$ into $\mathtt m'$ cycles and $\mathtt t'=\tau(S)-\tau(H)$ paths satisfying (i)--(iv) in Lemma~\ref{lem-decomposition-H-Subset-S}. Denote $\mathtt I=\{ \mathtt 1 \leq \mathtt i \leq \mathtt m' : C_{\mathtt i} \not \in \mathcal C(S) \}$. For each $\mathtt i \in \mathtt I$, writing 
    \[
    X_{\mathtt i} = \# \big( V(C_{\mathtt i}) \cap ( \cup_{\mathtt j=1}^{\mathtt t'} \operatorname{EndP}(P_{\mathtt j}) ) \big) \,,
    \]
    we then have $\sum_{\mathtt i \in \mathtt I} X_i \leq 2 \mathtt t'$. Thus, we can decompose $\{ C_{\mathtt i}: \mathtt i \in \mathtt I \}$ into at most $2\mathtt t'$ paths $\widetilde{P}_{1}, \ldots, \widetilde{P}_{2 \mathtt t'}$, with their endpoints in $\cup_{\mathtt j=1}^{\mathtt t'}\operatorname{EndP}(P_{\mathtt j})$. Now set $\mathcal C = \{ C_{\mathtt i} : \mathtt i \not \in \mathtt I \}$ and initialize $\mathcal P= \{ \widetilde{P}_{1}, \ldots, \widetilde{P}_{2 \mathtt t'} \}$. 
    Next for $1 \leq \mathtt j \leq \mathtt t'$, we perform the following procedure: (a) for each $u \in \operatorname{EndP}(P_{\mathtt j}) \cap \operatorname{IntP}(\mathcal P)$ where $\operatorname{IntP}(\mathcal P) = \big( \cup_{P \in \mathcal P} V(P) \big) \setminus \big( \cup_{p\in \mathcal P} \operatorname{EndP}(P) \big)$, we find $P^{u} \in \mathcal P$ such that $u \in V(P^u)$ and break $P^u$ at $u$ into two sub-paths $P^u (1)$, $P^u (2)$; (b) we update $\mathcal P$ by removing $P^u$ and adding $P^u(1), P^u(2)$ for each $u\in \operatorname{EndP}(P_{\mathtt j}) \setminus \operatorname{EndP}(\mathcal P )$; (c) we update $\mathcal P$ by adding $P_\mathsf j$. Since $|\operatorname{EndP}(P_{\mathtt j})| \le 2$, the whole procedure for each $\mathtt j$ increases $|\mathcal P|$ by $3$ at most. 
    
    After completing the aforementioned procedures for $1 \leq \mathtt j \leq \mathtt t'$, we finally obtain $\mathcal P=\{ P_{\mathtt 1},\ldots, P_{\mathtt t} \}$. From our construction, we see that $\mathcal P$ satisfies (i) and (ii). As for (iii), it holds since $|\mathcal P | \le 2\mathtt t '+ 3 \mathtt t '= 5 \mathtt t '$. This completes our proof when $\mathcal L(S) \subset V(H)$.

    For general cases, we complete our proof by applying the preceding proof to $H_{\operatorname{leaf}} \subset S$ such that $E(H_{\operatorname{leaf}}) = E(H)$ and $V(H_{\operatorname{leaf}}) = V(H) \cup \mathcal L(S)$. 
\end{proof}
\begin{remark}{\label{rmk-endpoints-have-degree-geq-3}}
    Note that in the special case where $\mathcal L(S) \subset V(H)$, we may further require that for each $u \in \operatorname{EndP}(P_{\mathtt j}) \setminus V(H)$, there are at least $3$ different $P_{\mathtt i}$'s having $u$ as endpoints. Otherwise, $u$ is exactly the endpoint of two paths $P_{\mathtt i}, P_{\mathtt j}$, and we can merge $P_{\mathtt i}$ and $P_{\mathtt j}$ into a longer path.
\end{remark}

The next few lemmas deal with enumerations of specific graphs, which will take advantage of previous results in this section.

\begin{lemma}{\label{lem-enu-cycle-path}}
    Given a vertex set $\mathsf A$ with $|\mathsf A| \leq D$, we have 
    \begin{align*}
        \#\Big\{ & (C_{\mathtt 1},\ldots,C_{\mathtt m};P_{\mathtt 1},\ldots,P_{\mathtt t}) : C_{\mathtt i} \mbox{ and } (\cup_{\mathtt i} V(C_{\mathtt i})) \cap \mathsf A = \emptyset, P_{\mathtt j} \text{'s are paths} ; \\
        & \#\big( (\cup_{\mathtt i} V(C_{\mathtt i})) \cup (\cup_{\mathtt j} V(P_{\mathtt j})) \big) \leq 2D; \#\{ \mathtt i: |V(C_{\mathtt i})|=x \}=p_x, |E(P_{\mathtt j})|=q_{\mathtt j} \Big\} \\
        \leq\ & (2D)^{2\mathtt t} \prod_{x} \frac{n^{xp_x}}{p_x!} \prod_{\mathtt j=1}^{\mathtt t} n^{q_{\mathtt j}-1} \,.
    \end{align*}
\end{lemma}
\begin{proof}
    Clearly, the enumeration of $\{ C_{\mathtt 1}, \ldots, C_{\mathtt m} \}$ is bounded by $\prod_{x} \frac{n^{x p_x}}{p_x!}$. In addition, given $\{ C_{\mathtt 1}, \ldots, C_{\mathtt m} , P_{\mathtt 1}, \ldots, P_{\mathtt j-1} \}$, we have at most $(2D)^2$ choices for the possible endpoints of $P_{\mathtt j}$ (here we use the bound on $\#\big( (\cup_{\mathtt i} V(C_{\mathtt i})) \cup (\cup_{\mathtt j} V(P_{\mathtt j})) \big)$), and at most $n^{q_{\mathtt j}-1}$ choices for $V(P_{\mathtt j}) \setminus \operatorname{EndP}(P_{\mathtt j})$. Thus, given $\{ C_{\mathtt 1}, \ldots, C_{\mathtt m} \}$, the enumeration of $\{ P_{\mathtt 1}, \ldots, P_{\mathtt t} \}$ is bounded by
    \begin{align*}
        \prod_{\mathtt j=1}^{\mathtt t} (2D)^2 n^{q_{\mathtt j}-1} = (2D)^{2\mathtt t} n^{ q_{1} + \ldots + q_{\mathtt t} -\mathtt t } \,,
    \end{align*}
    and the desired result follows from the multiplication principle.
\end{proof}

\begin{lemma}{\label{lem-enu-Subset-large-graph}}
    For $H \subset \mathcal{K}_n$ with $|E(H)| \leq D$, we have     
    \begin{equation}{\label{eq-enu-Subset-large-graph}}
        \begin{aligned}
            \#\Big\{ & S: H \ltimes S, |E(S)| \leq D, |E(S)|-|E(H)|=\ell+\kappa, \tau(S)-\tau(H)=\ell; \\
            & \mathcal{L}(S)\subset V(H), \cup_{j>N} \mathcal{C}_j(S) \subset H \Big\}
            \leq (2D)^{4\ell} n^{\kappa} \sum_{3p_3+\ldots+Np_N \leq \kappa+l} \prod_{j=3}^{N} \frac{1}{p_j!} \,.
        \end{aligned}
    \end{equation}    
\end{lemma}
\begin{proof}
    Take $S$ as an element in the set of \eqref{eq-enu-Subset-large-graph}. Using Lemma~\ref{lem-decomposition-H-Subset-S}, we can write (for some $\mathtt m, \ell \geq 0$)
    \[
    S \doublesetminus H = \big( \sqcup_{\mathtt i=1}^{\mathtt m} C_{\mathtt i} \big) \bigsqcup \big( \sqcup_{\mathtt j=1}^{\ell} P_{\mathtt j} \big) \,,
    \]
    where $\{ C_{\mathtt i} : 1 \leq \mathtt i \leq \mathtt m \}$ is a collection of disjoint cycles and $\{ P_{\mathtt j} : 1 \leq \mathtt j \leq \ell \}$ is a collection of paths satisfying (i)--(iv) in Lemma~\ref{lem-decomposition-H-Subset-S}. In addition, since $V(\mathcal C_j(S)) \subset V(H)$ for all $j>N$, for each $C_{\mathtt i}$ with $|V(C_{\mathtt i})|>N$, from Item (iii) in Lemma~\ref{lem-decomposition-H-Subset-S} there must exist $P_{\mathtt j}$ such that $\operatorname{End}(P_{\mathtt j}) \cap V(C_{\mathtt i}) \neq \emptyset$ (since independent cycles with length at least $N+1$ are contained in $H$). This yields that
    \begin{equation}{\label{eq-constraint-p-q-1}}
        \# \big\{ \mathtt i: |V(C_{\mathtt i})| > N \big\} \leq 2\ell \,.
    \end{equation}
    We are now ready to prove \eqref{eq-enu-Subset-large-graph} by bounding the enumeration of $\{ C_{\mathtt i} : 1 \leq \mathtt i \leq \mathtt m \}$ and $\{ P_{\mathtt j} : 1 \leq \mathtt j \leq \ell \}$. To this end, we assume $p_{x}=\#\{ \mathtt i: |V(C_{\mathtt i})|=x \}$ and $q_{\mathtt j} = |E(P_{\mathtt j})|$. Then we have
    \begin{equation}{\label{eq-constraint-p-q-2}}
        \sum_{i=3}^{\kappa} ip_i + \sum_{\mathtt j=1}^{\ell} q_{\mathtt j} = |E(S \doublesetminus H)| = \ell+\kappa \,.
    \end{equation}
    We first fix $p_3,\ldots,p_{\kappa}$ and $q_{\mathtt 1}, \ldots, q_{\ell}$. Applying Lemma~\ref{lem-enu-cycle-path} with $\mathsf A=V(H)$ we get that the enumeration of $\{ C_{\mathtt i} : 1 \leq \mathtt i \leq \mathtt m \}$ and $\{ P_{\mathtt j} : 1 \leq \mathtt j \leq \ell \}$ is bounded by 
    \begin{equation}\label{eq-pqfix-enum}
        (2D)^{2\ell} n^{ q_1+\ldots+q_{\ell}+ 3p_3+\ldots+\kappa p_\kappa-\ell } \prod_{i=3}^{\kappa} \frac{1}{p_i!} = (2D)^{2\ell} n^{\kappa} \prod_{i=3}^{N} \frac{1}{p_i!} \,.
    \end{equation}
    We next bound the enumeration on $p_3, \ldots, p_\kappa$ and $q_1, \ldots, q_\ell$ satisfying \eqref{eq-constraint-p-q-1} and \eqref{eq-constraint-p-q-2}. Note that
    \begin{align*}
        & \#\big\{ (p_{N+1},\ldots,p_{\kappa}, q_1,\ldots,q_{\ell}): p_{N+1}+\ldots+p_{\kappa} \leq 2\ell, q_1+\ldots+q_{\ell} \leq \ell+\kappa \big\} \\
        & \leq \kappa^{2\ell} \cdot (\ell+\kappa )^{\ell} \leq (2D)^{3\ell}\,, 
    \end{align*}
    where the last inequality follows from $\kappa,\ell \leq |E(S)|\leq D$. Combined with \eqref{eq-pqfix-enum}, this completes the proof of the lemma.
\end{proof}

\begin{lemma}{\label{lem-enu-Subset-small-graph}}
    For $S \subset \mathcal K_n$ with $|E(S)| \leq D$, we have
    \begin{equation}{\label{eq-enu-Subset-small-graph}}
    \begin{aligned}
        \# \Big\{ & H: H \ltimes S, \mathtt C_j(H) = \emptyset \mbox{ for } j \leq N ; \tau(S)-\tau(H) = \ell,  \\ 
        & \mathcal L(S) \cup (\cup_{j>N}V(\mathcal C_j(S))) \subset V(H) \Big\} \leq 2D^{15\ell} \,.
    \end{aligned}
    \end{equation}
\end{lemma}
\begin{proof}
    Take $H$ as an element in the set of \eqref{eq-enu-Subset-small-graph}. Using Corollary~\ref{cor-revised-decomposition-H-Subset-S}, we have $S \doublesetminus H$ can be written as (for some $\mathtt m, \mathtt t\geq 0$)
    \[
        S \doublesetminus H = \big( \sqcup_{\mathtt i=1}^{\mathtt m} C_{\mathtt i} \big) \bigsqcup \big( \sqcup_{\mathtt j=1}^{\mathtt t} P_{\mathtt j} \big) \,,
    \]
    where $\{ C_{\mathtt i} : 1 \leq \mathtt i \leq \mathtt m \}$ is a collection of independent cycles of $S$ and $\{ P_{\mathtt j} : 1 \leq \mathtt j \leq \mathtt t \}$ is a collection of paths satisfying (i)--(iii) in Corollary~\ref{cor-revised-decomposition-H-Subset-S}. Thus, in order to bound the enumeration of $H$ it suffices to bound the enumeration of $\{C_{\mathtt i}:1\leq \mathtt i\leq \mathtt m\}$ and $\{P_{\mathtt j}:1\leq \mathtt j\leq \mathtt t\}$. In addition, since for any valid $H$ we have $\mathtt C_j(H) = \emptyset$ for $j \leq N$ and $\cup_{j>N} \cup_{C \in \mathcal C_j(S)} C \subset V(H)$, the choice of $\{ C_{\mathtt i} : 1 \leq \mathtt i \leq \mathtt m \}$ is fixed given $\{ P_{\mathtt j} : 1 \leq \mathtt j \leq \mathtt t \}$. Thus, it suffices to upper-bound the total enumeration of $\{ P_{\mathtt j} : 1 \leq \mathtt j \leq \mathtt t \}$. 
    Given $\mathtt t \leq 5(\tau(S)-\tau(H)) =5\ell$, for each $1 \leq \mathtt j \leq \mathtt t$, the enumeration of $\operatorname{EndP}(P_{\mathtt j})$ is bounded by $D^2$. In addition, given $\operatorname{EndP}(P_{\mathtt j})$, since (by (ii) in Corollary~\ref{cor-revised-decomposition-H-Subset-S}) the vertices in $V(P_{\mathtt j}) \setminus \operatorname{EndP}(P_{\mathtt j})$ have exactly degree $2$ in $S$, the enumeration of $P_{\mathtt j}$ is bounded by $D$ (since once you choose the vertex right after the starting point of $P_{\mathtt j}$, the whole path is determined). Thus, the total enumeration of $\{ P_{\mathtt j} : 1 \leq \mathtt j \leq \mathtt t \}$ is bounded by
    \[
        \sum_{\mathtt t \leq 5\ell} D^{3\mathtt t} \leq 2 D^{15\ell} \,,
    \]
    finishing the proof of the lemma.
\end{proof}

\begin{lemma}{\label{lem-enu-general-subset-large-graph}}
    For $H \subset \mathcal{K}_n$, we have (below we write $\mathfrak P=\{(p_{N+1},\ldots,p_D): \sum_{i=N+1}^{D} p_i$ $\leq p, p_l \geq c_l \text{ for all } N+1 \leq l\leq D \}$ and $\sum_\mathfrak P$ for the summation over $(p_{N+1}, \ldots, p_D)\in \mathfrak P$)
    \begin{align}
        & \#\Big\{ S \mbox{ admissible}\!: H \ltimes S; |\mathfrak C_l(S,H)|=c_l \mbox{ for } l > N; |\mathcal L(S) \setminus V(H)| + \tau(S)-\tau(H) = m;  \nonumber \\
        & |E(S)|-|E(H)|=p, |V(S)|-|V(H)|=q, |E(S)| \leq D \Big\} \leq (2D)^{3m} n^{q} \sum_{\mathfrak P} \prod_{j=N+1}^{D} \frac{1}{p_j!} \,. \label{eq-enu-general-subset-large-graph}
    \end{align}
\end{lemma}
\begin{proof}
    Take $S$ as an element in the set of \eqref{eq-enu-general-subset-large-graph}. By Lemma~\ref{lem-decomposition-H-Subset-S}, we can decompose $S\doublesetminus H$ as (for some $\mathtt t\geq 0$)
    \[
        S \doublesetminus H = \big( \sqcup_{\mathtt i=1}^{\mathtt t} C_{\mathtt i} \big) \bigsqcup \big( \sqcup_{\mathtt j=1}^{m} P_{\mathtt j} \big) \,,
    \]
    where $\{ C_{\mathtt i} : 1 \leq \mathtt i \leq \mathtt t \}$ is a collection of disjoint cycles and $\{ P_{\mathtt j} : 1 \leq \mathtt j \leq m \}$ is a collection of paths satisfying (i)--(iv) in Lemma~\ref{lem-decomposition-H-Subset-S}. In addition, we have $|V(C_{\mathtt i})|>N$ for $1\leq \mathtt i\leq \mathtt t$ since $S$ is admissible. As before, it suffices to bound the enumeration of $\{ C_{\mathtt i} : 1 \leq \mathtt i \leq \mathtt t \}$ and $\{ P_{\mathtt j} : 1 \leq \mathtt j \leq m \}$. To this end, we assume $p_{x}=\#\{ \mathtt i: |V(C_{\mathtt i})|=x \}$ and $q_{\mathtt j} = |E(P_{\mathtt j})|$. Then we have
    \begin{equation}{\label{eq-constraint-p-q-3}}
        \sum_{i=N+1}^{D} ip_i + \sum_{\mathtt j=1}^{m} q_{\mathtt j} = |E(S)|-|E(H)| = p \mbox{ and } p_i \geq c_i \mbox{ for } N+1 \leq i \leq D \,.
    \end{equation}
    Thus, each valid choice of $p_{N+1}, \ldots, p_D$ and $q_1, \ldots, q_m$ satisfies that $\sum_{j=1}^m q_j \leq p$, implying that the enumeration of valid $q_1, \ldots, q_m$ is bounded by $(2D)^m$. For each valid $q_1, \ldots, q_m$, we fix $p_{N+1}, \ldots, p_D$ such that \eqref{eq-constraint-p-q-3} holds. Clearly, we have $(p_{N+1}, \ldots, p_D) \in \mathfrak P$. We now bound the enumeration of $S$ provided with fixed $p_{N+1}, \ldots, p_D$ and $q_1, \ldots, q_m$.
    Noting that $|\mathcal{L}(S)\setminus V(H)|=m-p+q$, we have at most $n^{m-p+q}$ choices for $\mathcal{L}(S)\setminus V(H)$. Given $\mathcal{L}(S)\setminus V(H)$, applying Lemma~\ref{lem-enu-cycle-path} with $\mathsf A=V(H) \cup (\mathcal{L}(S)\setminus V(H))$ we get that the enumeration of $\{ C_{\mathtt i} : 1 \leq \mathtt i \leq \mathtt t \}$ and $\{ P_{\mathtt j} : 1 \leq \mathtt j \leq m \}$ is bounded by
    \begin{align*}
        (2D)^{2m} n^{ q_1+\ldots+q_{m}+ (N+1)p_{N+1}+\ldots+Dp_D -m } \prod_{i=N+1}^{D} \frac{1}{p_i!} \,.
    \end{align*}
    Thus, (given $p_{N+1},\ldots,p_D$ and $q_{\mathtt 1}, \ldots, q_{m}$) the total enumeration of $\{ C_{\mathtt i} : 1 \leq \mathtt i \leq \mathtt t \}$ and $\{ P_{\mathtt j} : 1 \leq \mathtt j \leq m \}$ is bounded by
    \begin{align*}
        (2D)^{2m} n^{ m-p+q + q_1+\ldots+q_{m}+ (N+1)p_{N+1}+\ldots+Dp_D -m } \prod_{i=N+1}^{D} \frac{1}{p_i!} = (2D)^{2m} n^{q} \prod_{i=N+1}^{D} \frac{1}{p_i!} \,.
    \end{align*}
    Combined with preceding discussions on the enumeration for $q_1, \ldots, q_m$ and the requirement for $p_{N+1}, \ldots, p_D$, this implies the desired bound as in \eqref{eq-enu-general-subset-large-graph}.    
\end{proof}

\begin{lemma}{\label{lem-enu-general-subset-small-graph}}
    For an admissible $S$ with $|E(S)| \leq D$, we have
    \begin{equation}{\label{eq-enu-general-subset-small-graph}}
    \begin{aligned}
        \# \Big\{ & H: H \ltimes S, |\mathcal L(S) \setminus V(H)| + \tau(S)-\tau(H) = m ,  \\
        & \mathfrak C_j(S;H) = m_j, N+1 \leq j \leq D \Big\} \leq D^{15m} \prod_{j=N+1}^{D} \binom{ |\mathcal C_j(S)| }{ m_j } \,.
    \end{aligned}
    \end{equation}
\end{lemma}
\begin{proof}
    Take $H$ as an element in the set of \eqref{eq-enu-general-subset-small-graph}. Using Corollary~\ref{cor-revised-decomposition-H-Subset-S}, we have $S \doublesetminus H$ can be written as (for some $\mathtt t, \mathtt m\geq 0$)
    \[
        S \doublesetminus H = \big( \sqcup_{\mathtt i=1}^{\mathtt m} C_{\mathtt i} \big) \sqcup \big( \sqcup_{\mathtt j=1}^{\mathtt t} P_{\mathtt j} \big) \,,
    \]
    where $\{ C_{\mathtt i} : 1 \leq \mathtt i \leq \mathtt m \}$ is a collection of independent cycles of $S$ and $\{ P_{\mathtt j} : 1 \leq \mathtt j \leq \mathtt t \}$ is a collection of paths satisfying (i)--(iii) in Corollary~\ref{cor-revised-decomposition-H-Subset-S}. In addition, since $S$ is admissible, we have $|V(C_{\mathtt i})| \geq N+1$. Thus, the total enumeration of $\{ C_{\mathtt i} : 1 \leq \mathtt i \leq \mathtt m \}$ is bounded by $\prod_{j=N+1}^{D} \binom{| \mathcal C_j(S) |}{ m_j }$.
    Following the proof of Lemma~\ref{lem-enu-Subset-small-graph}, the total enumeration of $\{ P_{\mathtt j} : 1 \leq \mathtt j \leq \mathtt t \}$ is bounded by $D^{15m}$ in the following manner: for each $\mathtt j$,  we first bound the enumeration for $\operatorname{EndP}(P_{\mathtt j} )$ by $D^2$, and given this we bound the enumeration for $P_{\mathtt j}$ by $D$, and we finally sum over $\mathtt j$. Altogether, this completes the proof of the lemma.
\end{proof}

\section{Proof of Lemma~\ref{lem-trun-expectation-given-pi}}{\label{sec:Proof-Lem-4.10}}

Recall the definition of $G'$ in Definition~\ref{def-G'-P'}. Let $\mathcal{F}_G'=\sigma(\{G_e':e\in \operatorname{U}\})$ be the $\sigma$-field generated by the edge set of $G'$. It is important to note that $\mathcal{F}_G'$ is independent of $\pi_*$. The first step of our proof is to condition on $\mathcal F_G'$. Clearly we have
\begin{align*}
    \mathbb{E}_{\Pb_\pi'}[ A_e' \mid \mathcal F_G' ] = s G_e' \,, \ \mathbb{E}_{\Pb_\pi'}[ B_e' \mid \mathcal F_G' ] = s G_{\pi^{-1}(e)}' \,, \ \mathbb{E}_{\Pb_\pi'}[ A_e' B_{\pi(e)}' \mid \mathcal F_G' ] = s^2 G_{e}' \,.
\end{align*}
Thus,
\begin{align}
    & \mathbb{E}_{\Pb_\pi'} \Bigg[ \prod_{ e_1 \in E(S_1) } \frac{ A_{e_1}' - \frac{\lambda s}{n} }{ \sqrt{\lambda s/n} } \prod_{ e_2 \in E(S_2) } \frac{ B_{e_2}'-\frac{\lambda s}{n} }{ \sqrt{\lambda s/n} } \Bigg] \nonumber \\
    =& \mathbb{E}_{\Pb_\pi'} \Bigg[  \mathbb{E}_{\Pb_\pi'} \Big[ \prod_{ e_1 \in E(S_1) } \frac{ A'_{e_1} - \frac{\lambda s}{n} }{ \sqrt{\lambda s/n} } \prod_{ e_2 \in E(S_2) } \frac{ B'_{e_2}-\frac{\lambda s}{n} }{ \sqrt{\lambda s/n} } \mid \mathcal F_G' \Big]  \Bigg] \nonumber \\
    =& s^{\frac{1}{2}(|E(S_1)|+|E(S_2)|)} \cdot \mathbb{E}_{\Pb'_{\pi}}\Bigg[  \prod_{ e_1 \in E(S_1) } \frac{ G_{e_1}' - \frac{\lambda}{n} }{ \sqrt{\lambda/n} } \prod_{ e_2 \in E(\pi^{-1}(S_2)) } \frac{ G_{e_2}'-\frac{\lambda}{n} }{ \sqrt{\lambda/n} }  \Bigg] \,. \label{eq-B.2}
\end{align}
Thus, it suffices to show that for all admissible $S_1, S_2 \Subset \mathcal K_n$ and $H=S_1 \cap S_2$ we have (note that by replacing $S_2$ to $\pi^{-1}(S_2)$ the right-hand side of \eqref{eq-B.2} becomes the left-hand side of \eqref{eq-lem-technical-relax-1})
\begin{align}
    & s^{\frac{1}{2}(|E(S_1)|+|E(S_2)|)} \cdot \mathbb{E}_{\Pb'_*}\Bigg[  \prod_{ e_1 \in E(S_1) } \frac{ G_{e_1}' - \frac{\lambda}{n} }{ \sqrt{\lambda/n} } \prod_{ e_2 \in E(S_2) } \frac{ G_{e_2}'-\frac{\lambda}{n} }{ \sqrt{\lambda/n} }  \Bigg] \nonumber \\
    \leq\ & (\sqrt{\alpha}-\delta/2)^{|E(H)|} \sum_{H \ltimes K_1 \subset S_1} \sum_{H \ltimes K_2 \subset S_2} \tfrac{ \mathtt M(S_1,K_1) \mathtt M(S_2,K_2) \mathtt M(K_1,H) \mathtt M(K_2,H) }{ n^{ \frac{1}{2}( |V(S_1)|+|V(S_2)|-2|V(H)| ) } } \,. \label{eq-lem-technical-relax-1}
\end{align}
We estimate the left-hand side of \eqref{eq-lem-technical-relax-1} by the following two-step arguments. The first step is to deal with the special case $S_1=S_2=H$, where our strategy is to bound the left-hand side of \eqref{eq-lem-technical-relax-1} by its (slightly modified) first moment under $\Pb_*$.
\begin{lemma}{\label{lem-first-moment-bound}}
    For $H$ admissible with $|E(H)|\leq D$ we have
    \begin{align}
        & s^{|E(H)|} \mathbb{E}_{\Pb_*} \Bigg[ \prod_{e \in E(H)}\frac{ (G_{e}-\tfrac{\lambda}{n})^2 }{ \lambda/n } \Bigg] 
        \leq\ O(1) \cdot (\sqrt{\alpha}-\delta/4)^{|E(H)|} \,. \label{eq-lem-B.1}
    \end{align}
\end{lemma}
To show Lemma~\ref{lem-first-moment-bound}, we first prove a useful lemma regarding the conditional expectation of a certain product along a path, given its endpoints. Denote
\begin{equation}{\label{eq-def-omega-sigma-i,j}}
    \omega(\sigma_i,\sigma_j) = 
    \begin{cases}
    k-1 , & \sigma_i=\sigma_j \,; \\
    -1 , & \sigma_i \ne \sigma_j \,. 
    \end{cases}
\end{equation}
\begin{claim}{\label{claim-expectation-over-chain}}
    For a path $\mathcal{P}$ with $V(\mathcal P)= \{ v_0, \ldots, v_l \} $ and $\operatorname{EndP}(\mathcal P)=\{ v_0,v_l \}$, we have 
    \begin{equation}{\label{eq-expectation-over-chain}}
    \begin{aligned}
        \mathbb{E}_{\sigma \sim \nu} \Big[ \prod_{i=1}^{l} \Big( 1+\epsilon\omega(\sigma_{i-1},\sigma_i) \Big) \mid \sigma_0, \sigma_l \Big] 
        = 1 + \epsilon^l \cdot \omega(\sigma_0, \sigma_l)  \,.
    \end{aligned}
    \end{equation}
\end{claim}
\begin{proof}
    By independence, we see that $\mathbb E_{\sigma \sim \nu} \big[ \prod_{i \in I} \omega(\sigma_{i-1},\sigma_i) \mid \sigma_0, \sigma_l \big] = 0$ if $I \subsetneq [l]$. Thus,
    \begin{align*}
        \mathbb{E}_{\sigma \sim \nu} \Big[ \prod_{i=1}^{l} \Big( 1+\epsilon\omega(\sigma_{i-1},\sigma_i) \Big) \mid \sigma_0, \sigma_l \Big] = 1+ \epsilon^l\mathbb{E}_{\sigma \sim \nu} \Big[ \prod_{i=1}^{l} \omega(\sigma_{i-1},\sigma_i) \mid \sigma_0, \sigma_l \Big] \,.
    \end{align*}
    It remains to prove that
    \begin{equation}\label{eq-finalgoal1-sign-of-path-cycle-conditioned-on-endpoints}
        \mathbb{E}_{\sigma \sim \nu} \Big[ \prod_{i=1}^{l} \omega(\sigma_{i-1},\sigma_i) \mid \sigma_0, \sigma_l \Big] = \omega(\sigma_0, \sigma_l)\,.
    \end{equation}
    We shall show \eqref{eq-finalgoal1-sign-of-path-cycle-conditioned-on-endpoints} by induction. The case $l=1$ follows immediately. Now we assume that \eqref{eq-finalgoal1-sign-of-path-cycle-conditioned-on-endpoints} holds for $l$. Then we have
    \begin{align*}
        &\mathbb{E}_{\sigma \sim \nu} \Big[ \prod_{i=1}^{l+1} \omega(\sigma_{i-1},\sigma_i) \mid \sigma_0, \sigma_{l+1} \Big]\\
        = & \mathbb{E}_{\sigma \sim \nu}\Big[\omega(\sigma_l,\sigma_{l+1})\mathbb{E}_{\sigma \sim \nu} \Big[ \prod_{i=1}^{l} \omega(\sigma_{i-1},\sigma_i) \mid \sigma_0, \sigma_l,\sigma_{l+1} \Big] \mid \sigma_0,\sigma_{l+1}\Big]\\
        = & \mathbb{E}_{\sigma \sim \nu}\Big[ \omega(\sigma_l,\sigma_{l+1}) \omega(\sigma_0,\sigma_l) \mid \sigma_0,\sigma_{l+1}\Big] = \omega(\sigma_0,\sigma_{l+1})\,,
    \end{align*}
    which completes the induction procedure. 
\end{proof}
Based on Claim~\ref{claim-expectation-over-chain}, we can prove Lemma~\ref{lem-first-moment-bound} by a straightforward calculation, as incorporated in the upcoming Section~\ref{sec:supp-proofs-sec-4}. Now we estimate the expectation under $\Pb_*'$ in a more sophisticated way. Recall that $H = S_1 \cap S_2$. Firstly, by averaging over the conditioning on community labels we have (we write $\Pb'_{\sigma}=\Pb'_*(\cdot\mid\sigma_*=\sigma)$)
\begin{align}
    & s^{\frac{1}{2}(|E(S_1)|+|E(S_2)|)} \mathbb E_{\Pb_*'}\Bigg[ \prod_{e \in E(S_1) \triangle E(S_2)} \frac{ \big( G_{e}'-\tfrac{\lambda}{n} \big) }{ \sqrt{\lambda/n} } \prod_{e \in E(H)} \frac{ \big( G_{e}'-\tfrac{\lambda}{n} \big)^2 }{ \lambda/n }  \Bigg] \nonumber \\
    = \ & s^{\frac{1}{2}(|E(S_1)|+|E(S_2)|)} \mathbb E_{\sigma\sim\nu} \Bigg\{ \mathbb{E}_{\Pb'_{\sigma}} \Bigg[ \prod_{e \in E(S_1) \triangle E(S_2)} \frac{ \big( G_{e}'-\tfrac{\lambda}{n} \big) }{ \sqrt{\lambda/n} } \prod_{e \in E(H)} \frac{ \big( G_{e}'-\tfrac{\lambda}{n} \big)^2 }{ \lambda/n } \Bigg] \Bigg\} \,. \label{eq-B.3}
\end{align}
Note that given $\sigma_*=\sigma$, we have $G_{i,j} \sim \operatorname{Ber}\big( \frac{ (1+\epsilon \omega(\sigma_i,\sigma_j))\lambda }{n} \big)$ independently. This motivates us to write the above expression in the centered form as follows:
\begin{align*}
    & s^{\frac{1}{2}(|E(S_1)|+|E(S_2)|)} \prod_{(i,j) \in E(S_1) \triangle E(S_2)} \frac{ \big( G_{i,j}'-\tfrac{\lambda}{n} \big) }{ \sqrt{\lambda/n} } \prod_{(i,j) \in E(H)} \frac{ \big( G_{i,j}'-\tfrac{\lambda}{n} \big)^2 }{ \lambda/n }  \\
    = & s^{|E(H)|} \prod_{(i,j) \in E(S_1) \triangle E(S_2)} \frac{ \big( G_{i,j}'-\tfrac{(1+\epsilon\omega( \sigma_{i},\sigma_j))\lambda}{n} + \tfrac{\epsilon \omega(\sigma_{i},\sigma_j) \lambda}{n} \big) }{ \sqrt{\lambda/ns} } \prod_{(i,j) \in E(H)} \frac{ \big( G_{i,j}'-\tfrac{\lambda}{n} \big)^2 }{ \lambda/n } \\
    = & s^{|E(H)|} \sum_{H \ltimes K_1 \subset S_1} \sum_{H \ltimes K_2 \subset S_2} h_{\sigma}(S_1,S_2;K_1,K_2) \varphi_{\sigma}(K_1,K_2;H)  \,,
\end{align*}
where
\begin{align}
     & h_{\sigma}(S_1,S_2;K_1,K_2) = \prod_{ (i,j) \in (E(S_1) \cup E(S_2)) \setminus (E(K_1) \cup E(K_2)) } \frac{ \omega(\sigma_i,\sigma_j) \sqrt{\epsilon^2\lambda s} }{ \sqrt{n} } \,, \label{eq-def-h-sigma} \\
     & \varphi_{\sigma}(K_1,K_2;H) = \prod_{(i,j) \in (E(K_1) \cup E(K_2)) \setminus E(H)} \frac{ \big( {G}_{i,j}' - \tfrac{ (1+\epsilon \omega(\sigma_i,\sigma_j))\lambda }{n} \big) }{ \sqrt{\lambda/ns} } \prod_{(i,j) \in E(H)} \frac{ \big( G_{i,j}'-\tfrac{\lambda}{n} \big)^2 }{ \lambda/n }  \,. \label{eq-def-varphi}
\end{align}
In conclusion, we can write \eqref{eq-B.3} as (note that below the summation is over $K_1,K_2$)
\begin{align}
    \eqref{eq-B.3} = & \sum_{H \ltimes K_1 \subset S_1} \sum_{H \ltimes K_2 \subset S_2} s^{|E(H)|} \mathbb E_{\sigma\sim\nu} \Bigg\{ h_{\sigma}(S_1,S_2;K_1,K_2) \mathbb{E}_{\Pb_{\sigma}'} \Big[ \varphi_{\sigma}(K_1,K_2;H) \Big] \Bigg\} \,. \label{eq-B.6}
\end{align}
We now show the following bound on the summand in \eqref{eq-B.6}.
\begin{lemma}{\label{lem-extreme-technical}}
    Recall Equation~\eqref{eq-def-mathtt-M}. We have 
    \begin{align*}
        & \Bigg| \mathbb E_{\sigma\sim\nu} \Big\{ h_{\sigma}(S_1,S_2;K_1,K_2) \mathbb{E}_{\Pb_{\sigma}'} \big[ \varphi_{\sigma}(K_1,K_2;H) \big] \Big\} \Bigg| \\
        \leq\ & \mathbb E_{\Pb_*} \Bigg[ \prod_{(i,j) \in E(H)} \frac{ \big( G_{i,j}-\tfrac{\lambda}{n} \big)^2 }{ \lambda/n } \Bigg] \cdot \frac{ \mathtt M(S_1,K_1) \mathtt M(S_2,K_2) \mathtt M(K_1,H) \mathtt M(K_2,H) }{ n^{ \frac{1}{2}(|V(S_1)|+|V(S_2)|-2|V(H)|) } } \,.
    \end{align*}
\end{lemma}
We can now finish the proof of Lemma~\ref{lem-trun-expectation-given-pi}.
\begin{proof}[Proof of Lemma~\ref{lem-trun-expectation-given-pi}]
    Plugging the estimation of Lemma~\ref{lem-first-moment-bound} into the right-hand side of Lemma~\ref{lem-extreme-technical}, we see that 
    \begin{align*}
        & s^{|E(H)|} \cdot \Bigg| \mathbb E_{\sigma\sim\nu} \Big\{  h_{\sigma}(S_1,S_2;K_1,K_2) \mathbb{E}_{\Pb_{\sigma}'} \big[ \varphi_{\sigma}(K_1,K_2;H) \big] \Big\} \Bigg| \\
        \leq\ & O(1) \cdot (\sqrt{\alpha}-\delta/4)^{|E(H)|} \cdot \frac{ \mathtt M(S_1,K_1) \mathtt M(S_2,K_2) \mathtt M(K_1,H) \mathtt M(K_2,H) }{ n^{ \frac{1}{2}(|V(S_1)|+|V(S_2)|-2|V(H)|) } } \,.
    \end{align*}
    Combined with \eqref{eq-B.3} and \eqref{eq-B.6}, this yields \eqref{eq-lem-technical-relax-1}, leading to Lemma~\ref{lem-trun-expectation-given-pi}.
\end{proof}

The rest of this section is devoted to the proof of Lemma~\ref{lem-extreme-technical}. Recall that $H=S_1 \cap S_2$. Denote $\mathtt L=\mathtt L_1 \cup \mathtt L_2$, where
\begin{equation}{\label{eq-def-mathtt-L}}
\begin{aligned}
    \mathtt L_1 &= \big( \mathcal L(S_1) \setminus V(K_1) \big) \cup \big( \mathcal L(S_2) \setminus V(K_2) \big) \,; \\
    \mathtt L_2 &= \big( \mathcal L(K_1) \setminus V(H) \big) \cup \big( \mathcal L(K_2) \setminus V(H) \big) \,. 
\end{aligned}    
\end{equation}
In addition, denote
\begin{align}
    \mathtt V &= \big( V(S_1) \setminus V(K_1) \big) \cup \big( V(S_2) \setminus V(K_2) \big) \,; \label{eq-def-mathtt-V} \\
    \mathtt W &= \mathtt L_1 \cup \big( V(K_1) \setminus V(H) \big) \cup \big( V(K_2) \setminus V(H) \big) \,. \label{eq-def-mathtt-W} 
\end{align} 
We also define
\begin{equation}{\label{eq-def-Gamma}}
\begin{aligned}
    \Gamma_1 &= |\mathcal L(S_1) \setminus V(K_1)| + |\mathcal L(S_2) \setminus V(K_2)| + \tau(S_1)-\tau(K_1) + \tau(S_2)-\tau(K_2) \,; \\
    \Gamma_2 &= |\mathcal L(K_1) \setminus V(H)| + |\mathcal L(K_2) \setminus V(H)| + \tau(K_1) + \tau(K_2) - 2\tau(H) \,.
\end{aligned}
\end{equation}
For any $\sigma \in [k]^{n}$, denote by $\varkappa$ and $\gamma$ the restriction of $\sigma$ on $\mathtt V$ and $[n] \setminus \mathtt V$, respectively. We also write $\sigma=\varkappa\oplus\gamma$. Then
\begin{align*}
    & \mathbb E_{\sigma\sim\nu} \Bigg\{ h_{\sigma}(S_1,S_2;K_1,K_2) \mathbb E_{\Pb_{\sigma}'}\Big[ \varphi_{\sigma}(K_1,K_2;H) \Big] \Bigg\} \\
    =\ & \mathbb E_{\gamma\sim\nu_{ [n] \setminus \mathtt V }} \mathbb E_{\varkappa\sim\nu_{\mathtt V}} \Bigg\{ h_{\varkappa\oplus\gamma}(S_1,S_2;K_1,K_2) \mathbb E_{\Pb_{\varkappa\oplus\gamma}'}\Big[ \varphi_{\varkappa\oplus\gamma}(K_1,K_2;H) \Big] \Bigg\}\,.
\end{align*}
Clearly, it suffices to show that for all $\gamma$ we have the following estimates:
\begin{equation}{\label{eq-extreme-technical}}
    \begin{aligned}
        & \Bigg| \mathbb E_{\varkappa\sim \nu_{_{\mathtt V}}} \Big\{ h_{\varkappa\oplus\gamma}(S_1,S_2;K_1,K_2) \mathbb E_{\Pb_{\varkappa\oplus\gamma}'}\big[ \varphi_{\varkappa\oplus\gamma}(K_1,K_2;H) \big] \Big\} \Bigg| \\
        \leq\ & \mathbb E_{\Pb_{\gamma}} \Big[ \prod_{(i,j) \in E(H)} \frac{ (G_{i,j}-\lambda/n)^2 }{ \lambda/n } \Big] \cdot \frac{ \mathtt M(S_1,K_1) \mathtt M(S_2,K_2) \mathtt M(K_1,H) \mathtt M(K_2,H) }{ n^{ \frac{1}{2}(|V(S_1)|+|V(S_2)|-2|V(H)|) } } \,.
    \end{aligned}
\end{equation}
We begin our proof of \eqref{eq-extreme-technical} by constructing a probability measure (below $\operatorname{par}$ is the index of a special element to be defined)
\begin{align*}
    \widetilde \Pb = \widetilde \Pb_{\gamma} \mbox{ on } \Big( \widetilde{\Omega}, 2^{\widetilde{\Omega}} \Big) \,, \mbox{ where } \widetilde{\Omega} = \Big\{ \big( \chi(\varkappa) \big)_{\varkappa \in [k]^{\mathtt V} \cup \{ \operatorname{par} \} } : \chi(\varkappa) \in \{ 0,1 \}^{ \operatorname{U} }, \forall \varkappa \in [k]^{\mathtt V} \cup \{ \operatorname{par} \} \Big\} \,,
\end{align*} 
(we will write $\widetilde{\mathbb P}$ instead of $\widetilde{\mathbb P}_{\gamma}$ for simplicity when there is no ambiguity) such that for $(G'(\varkappa))_{\varkappa \in [k]^{\mathtt V}}$ sampled from $\widetilde{\mathbb P}$, we have $G'(\varkappa) \sim \Pb_{\varkappa \oplus \gamma}'$ for each $\varkappa \in [k]^{\mathtt V}$. This measure $\widetilde {\mathbb P}_{\gamma}$ is constructed as follows. First we generate a parent graph $G(\operatorname{par})$ such that $\{G(\operatorname{par})_{i,j}\}$ is a collection of independent Bernoulli variables which take value 1 with probability $\tfrac{ ( 1 + \epsilon \omega( \gamma_i, \gamma_j ) ) \lambda}{n}$ if $i,j \in [n] \setminus \mathtt V$ and with probability $\tfrac{(1+\epsilon(k-1))\lambda}{n}$ if $i \in \mathtt V \mbox{ or } j \in \mathtt V$. Let $\{J(\varkappa)_{i, j} : \varkappa\in [k]^{\mathtt V}, (i,j)\in \operatorname{U}\}$ be a collection of independent Bernoulli variables with parameter $\tfrac{1-\epsilon}{1+\epsilon (k-1)}$. Given $G(\operatorname{par})$, for each $\varkappa \in [k]^{\mathtt V}$ define
\begin{equation}{\label{eq-def-G-tau}}
    G(\varkappa)_{i,j} = 
    \begin{cases}
        G(\operatorname{par})_{i,j} \,, & i,j \in [n] \setminus \mathtt V \mbox{ or } (\varkappa\oplus\gamma)_i = (\varkappa\oplus\gamma)_j \,; \\
        G(\operatorname{par})_{i,j} J(\varkappa)_{i,j} \,, & \mbox{ otherwise} \,.
    \end{cases}
\end{equation}
Recall from Definition~\ref{def-G'-P'} that $\mathtt B_1, \ldots, \mathtt B_M$ is the collection of all cycles in $\mathcal K_n$ with lengths at most $N$ and all self-bad graphs in $\mathcal K_n$ with at most $D^3$ vertices.
For each $1 \leq i \leq M$, if $V(\mathtt B_i) \subset [n] \setminus \mathtt V$ and $\mathtt B_i \subset G(\operatorname{par})$, we choose a uniform edge from $\mathtt B_i$ and delete this edge in each $G(\varkappa)$; if $V(\mathtt B_i) \cap \mathtt V \neq \emptyset$, for each $\varkappa \in [k]^{\mathtt V}$, if $\mathtt B_i \subset G(\varkappa)$ we independently delete a uniform edge of $\mathtt B_i$ in $G(\varkappa)$. The remaining edges of $G(\varkappa)$ constitute $G'(\varkappa)$ and we let $\widetilde{\Pb}$ to be the joint measure of $\big( G(\operatorname{par}), (G'(\varkappa))_{\varkappa \in [k]^{\mathtt V}} \big)$ (so $\chi(\mathrm{par})$ represents the realization for $G(\mathrm{par})$ as hinted earlier). Clearly we have $G'(\varkappa) \sim \Pb_{\varkappa\oplus\gamma}'$ as we wished. Thus, we can write the left-hand side of \eqref{eq-extreme-technical} as
\begin{align}
    \Bigg| \mathbb E_{\widetilde \Pb} \Big[ \frac{1}{k^{|\mathtt V|}} \sum_{\varkappa \in [k]^{\mathtt V}} h_{\varkappa\oplus\gamma}(S_1,S_2;K_1,K_2) \varphi_{\gamma;K_1,K_2;H}(G'(\varkappa)) \Big] \Bigg| \,,\label{eq-extreme-technical-relax-1}
\end{align}
where for each $X \in \{ 0,1 \}^{\operatorname{U}}$, we used $\varphi_{\gamma;K_1,K_2;H}(X)$ to denote the formula obtained from replacing $G'_{i,j}$ by $X_{i,j}$ and replacing $\sigma$ from $\varkappa \oplus \gamma$ in \eqref{eq-def-varphi} for an arbitrary $\varkappa \in [k]^{\mathtt V}$ (note that $\varphi_{\gamma;K_1,K_2;H}(X)$ only depends on $\gamma$ and the values of $X$ on $\mathtt E= E(K_1) \cup E(K_2)$). 
To calculate \eqref{eq-extreme-technical-relax-1}, we will condition on 
\begin{align*}
    \mathcal F_{\operatorname{par}} = \sigma\Big\{ G(\operatorname{par})_{i,j}: (i,j) \in \operatorname{U} \setminus \mathtt E \Big\} \,.
\end{align*}
We will argue that unless the realization $\chi\in \{0, 1\}^{\operatorname{U}\setminus \mathtt E}$ satisfies a specific condition, we have that the conditional expectation (below $G(\operatorname{par})|_{\mathsf A}$ denotes the restriction of $G(\operatorname{par})$ on $\mathsf A$)
\begin{equation}{\label{eq-conditional-expectation}}
    \mathbb E_{\widetilde \Pb} \Big[ \frac{1}{k^{|\mathtt V|}} \sum_{\varkappa \in [k]^{\mathtt V}} h_{\varkappa\oplus\gamma}(S_1,S_2;K_1,K_2) \varphi_{\gamma;K_1,K_2;H}(G'(\varkappa)) \mid G(\operatorname{par})|_{\operatorname{U} \setminus \mathtt E} = \chi \Big]
\end{equation}
cancels to $0$. We need to introduce more notations before presenting our proofs.
\begin{defn}{\label{def-bad-vertex-set}}
    For $\chi \in \{ 0,1 \}^{ \operatorname{U} \setminus \mathtt E}$, we define the \emph{bad vertex set} with respect to $\chi$ as
    \begin{align*}
        \mathcal B(\chi) = \Big\{ & u \in (V(S_1) \cup V(S_2)) \setminus V(H) : \exists K \subset \chi \oplus \{ \mathtt{1}_{\mathtt E} \}, u \in V(K), K \mbox{ is a cycle with} \\
        & \mbox{length at most } N \mbox{ or a self-bad graph with at most } D^3 \mbox{ vertices} \Big\} \,.
    \end{align*}
\end{defn}
Clearly from this definition we see that $\mathcal B(G(\operatorname{par})|_{\operatorname{U} \setminus \mathtt E})$ is measurable with respect to $\mathcal F_{\operatorname{par}}$. Our proof will employ the following estimates. Recall \eqref{eq-def-mathtt-W} and \eqref{eq-def-Gamma}.
\begin{claim}{\label{claim-condition-on-G-par-case-1}}
    For any $\chi \in \{ 0,1 \}^{ \operatorname{U} \setminus \mathtt E }$ such that $\mathtt W \not \subset \mathcal B(\chi)$, we have $\eqref{eq-conditional-expectation} = 0$.
\end{claim}
\begin{claim}{\label{claim-condition-on-G-par-case-2}}
    For any $\chi \in \{ 0,1 \}^{ \operatorname{U} \setminus \mathtt E }$ such that $\mathtt W \subset \mathcal B(\chi)$ and $|\mathcal B(\chi) \setminus \mathtt W|=\ell$, we have
    \begin{align*}
        \big| \eqref{eq-conditional-expectation} \big| \leq k^{5\Gamma_1+5l} \frac{ (1-\delta/2)^{|E(S_1)|+|E(S_2)|-|E(K_1)|-|E(K_2)|} }{n^{\frac{1}{2}(|E(S_1)|+|E(S_2)|-|E(K_1)|-|E(K_2)|)}} \cdot \mathbb E_{\widetilde{\Pb}} \Big[ \big| \varphi_{\gamma;K_1,K_2;H} (G(\operatorname{par})|_{\mathtt E}) \big| \Big] \,.
    \end{align*}
\end{claim}
\begin{claim}{\label{claim-prob-bad-set-realization}}
    Suppose that $S_1,S_2\Subset \mathcal K_n$ are admissible, $H=S_1 \cap S_2$, $H \ltimes K_1 \subset S_1, H \ltimes K_2 \subset S_2$. For any $\mathtt B \subset V(S_1 \cup S_2 )$ such that $\mathtt W \subset \mathtt B, |\mathtt B \setminus \mathtt W|=\ell$, we have 
    \begin{equation*}
    \begin{aligned}
        \widetilde\Pb \Big( \mathcal B( G(\operatorname{par})|_{\operatorname{U}\setminus\mathtt E} ) = \mathtt B \Big) \leq \frac{ n^{ \frac{1}{4}( \tau(K_1)+\tau(K_2)-2\tau(H) ) } n^{ - \frac{1}{4}|\mathtt L_2| - \frac{1}{2}|\mathtt L_1| - \frac{1}{4}\ell } ( 2000\Tilde{\lambda}^{22}k^{22} )^{2N^2(\Gamma_1+\Gamma_2)} }{ (4\Tilde{\lambda}^{2} k^{2} )^{|E(K_1)|+|E(K_2)|-2|E(H)|} } \,.
    \end{aligned}
    \end{equation*}
\end{claim}
The proofs of Claims~\ref{claim-condition-on-G-par-case-1}, \ref{claim-condition-on-G-par-case-2} and \ref{claim-prob-bad-set-realization} are incorporated in Sections~\ref{subsec:Proof-Claim-B.5}, \ref{subsec:Proof-Claim-B.6} and \ref{subsec:Proof-Claim-B.7}, respectively.
Now we can present the proof of \eqref{eq-extreme-technical}, thus completing the proof of Lemma~\ref{lem-extreme-technical}.
\begin{proof}[Proof of \eqref{eq-extreme-technical}]
    Recall \eqref{eq-extreme-technical-relax-1}. We can write it as
    \begin{align*}
        \eqref{eq-extreme-technical-relax-1} = \Bigg| \mathbb E \Bigg\{ \mathbb E_{\widetilde \Pb} \Big[ \frac{1}{k^{|\mathtt V|}} \sum_{\varkappa \in [k]^{\mathtt V}} h_{\varkappa\oplus\gamma}(S_1,S_2;K_1,K_2) \varphi_{\gamma;K_1,K_2;H}(G'(\varkappa)) \mid \mathcal F_{\operatorname{par}} \Big] \Bigg\} \Bigg| \,.
    \end{align*}
    Combining Claims~\ref{claim-condition-on-G-par-case-1} and \ref{claim-condition-on-G-par-case-2}, we see that the above expression is bounded by the product of $\frac{(1-\delta/2)^{ |E(S_1)|+|E(S_2)|-|E(K_1)|-|E(K_2)| }}{n^{\frac{1}{2}(|E(S_1)|+|E(S_2)|-|E(K_1)|-|E(K_2)|)}} \cdot \mathbb E_{\widetilde \Pb} \Big[ \big| \varphi_{\gamma,K_1,K_2,H} \big( G(\operatorname{par})|_{\mathtt E} \big) \big| \Big]$ and the following term:
    \begin{align}
        & \sum_{ \ell \geq 0 } \sum_{ \substack{ \mathtt W \subset \mathtt B \subset V(S_1 \cup S_2 ) \\ |\mathtt B \setminus \mathtt W|=\ell } } k^{5\Gamma_1+5\ell} \cdot \widetilde{\mathbb P}\Big( \mathcal B\big( G(\operatorname{par})|_{\operatorname{U}\setminus\mathtt E} \big) \big) = \mathtt B \Big)  \,. \label{eq-C.49} 
    \end{align}
    Since $|V(S_1)|,|V(S_2)| \leq 2D$, we have
    \begin{align*}
        \# \big\{ \mathtt B: \mathtt W \subset \mathtt B \subset V(S_1 \cup S_2) , |\mathtt B \setminus \mathtt W|=\ell \big\} \leq (4D)^{\ell} \,.
    \end{align*}
    Combined with Claim~\ref{claim-prob-bad-set-realization}, it yields that
    \begin{align}
        \eqref{eq-C.49} &\leq \sum_{ \ell \geq 0 } \frac{ k^{5\ell} (4D)^{\ell} n^{\frac{1}{4}(\tau(K_1)+\tau(K_2) -2\tau(H))} n^{ -\frac{1}{4}|\mathtt L_2| - \frac{1}{2} |\mathtt L_1| -\frac{1}{4}\ell } (2000\Tilde{\lambda}^{22}k^{23})^{2N^2(\Gamma_1+\Gamma_2)} }{ (4\Tilde{\lambda}^{2}k^{2})^{ |E(K_1)|+|E(K_2)|-2|E(H)| } } \nonumber \\
        &\leq [1+o(1)] \cdot \frac{ n^{\frac{1}{4}(\tau(K_1)+\tau(K_2) -2\tau(H))} n^{ -\frac{1}{4}|\mathtt L_2| - \frac{1}{2} |\mathtt L_1| } (2000\Tilde{\lambda}^{22}k^{23})^{2N^2(\Gamma_1+\Gamma_2)} }{ (4\Tilde{\lambda}^{2}k^{2})^{ |E(K_1)|+|E(K_2)|-2|E(H)| } } \,. \label{eq-C.50}
    \end{align}
    Since the entries in $G(\operatorname{par})$ are stochastically dominated by a family of i.i.d.\ Bernoulli random variables with parameter $\frac{ (1+\epsilon k)\lambda }{n}$, we have that $\mathbb E_{ \widetilde{\Pb}_{\gamma} } \Big[ \big| \varphi_{\gamma,K_1,K_2,H} \big( G(\operatorname{par})|_{\mathtt E} \big) \big| \Big]$ is bounded by (note that below we used $s\leq 1$ for simplification)
    \begin{align}
        & \mathbb E_{ \widetilde{\Pb}_{\gamma} } \Bigg[ \prod_{(i,j) \in E(K_1 \cup K_2 ) \setminus E(H)} \frac{ \big| G(\operatorname{par})_{i,j} - \frac{ (1+\epsilon \omega(\gamma_i,\gamma_j) )\lambda }{n} \big| }{ \sqrt{\lambda/ns} } \prod_{(i,j) \in E(H)}  \frac{ \big( G(\operatorname{par})_{i,j} - \tfrac{\lambda}{n} \big)^2 }{ \lambda/n } \Bigg] \nonumber \\
        \leq\ & \big( \tfrac{2k\lambda}{n} \big)^{ \frac{1}{2}( |E(K_1)|+|E(K_2)|-2|E(H)| ) } \mathbb E_{ \widetilde{\Pb}_{\gamma} } \Bigg[ \prod_{(i,j) \in E(H)}  \frac{ \big( G(\operatorname{par})_{i,j} - \tfrac{\lambda}{n} \big)^2 }{ \lambda/n } \Bigg] \nonumber \\
        =\ &\big( \tfrac{2k\lambda}{n} \big)^{ \frac{1}{2}( |E(K_1)|+|E(K_2)|-2|E(H)| ) } \mathbb E_{ \Pb_{\gamma} } \Bigg[ \prod_{(i,j) \in E(H)} \frac{ \big( G_{i,j} - \tfrac{\lambda}{n} \big)^2 }{ \lambda/n } \Bigg] \,,  \label{eq-C.51}
    \end{align}
    where the equality follows from the fact that the distribution of $G(\operatorname{par})$ under $\widetilde{\Pb}_\gamma$ is equal to the distribution of $G$ under $\Pb_\gamma$. Plugging \eqref{eq-C.51} and \eqref{eq-C.50} into the bound surrounding \eqref{eq-C.49}, we obtain that \eqref{eq-extreme-technical-relax-1} is bounded by the product of $\mathbb E_{ \Pb_{\gamma} } \Big[ \prod_{(i,j) \in E(H)} \frac{ ( G_{i,j} - \tfrac{\lambda}{n} )^2 }{ \lambda/n } \Big]$ and
    \begin{align*}
        & \frac{ (1-\tfrac{\delta}{2})^{ |E(S_1)|+|E(S_2)|-2|E(H)|} n^{ \frac{1}{4}(\tau(K_1)+\tau(K_2)-2\tau(H)) } (2000\Tilde{\lambda}^{22}k^{23})^{2N^2(\Gamma_1+\Gamma_2)} }{ n^{\frac{1}{2}(|E(S_1)|+|E(S_2)|-2|E(H)|)} n^{ \frac{1}{2}|\mathtt L_1|+\frac{1}{4}|\mathtt L_2| } } \\
        =\ & \frac{ (1-\tfrac{\delta}{2})^{ |E(S_1)|+|E(S_2)|-2|E(H)|} (2000\Tilde{\lambda}^{22}k^{23})^{2N^2(\Gamma_1+\Gamma_2)} }{ n^{\frac{1}{2}(|V(S_1)|+|V(S_2)|-2|V(H)|)} n^{ \frac{1}{2}\Gamma_1 + \frac{1}{4}\Gamma_2 } }  \\
        \leq\ & \frac{ \mathtt M(S_1,K_1) \mathtt M(S_2,K_2) \mathtt M(K_1,H) \mathtt M(K_2,H) }{ n^{\frac{1}{2}(|V(S_1)|+|V(S_2)|-2|V(H)|)} }  \,,
    \end{align*}
    where the equality follows from \eqref{eq-def-mathtt-L} and \eqref{eq-def-Gamma}, and the inequality follows from Equation \eqref{eq-def-mathtt-M}. Thus we have shown \eqref{eq-extreme-technical}. 
\end{proof}

\section{Supplementary proofs in Section~\ref{sec:detection-lower-bound}}{\label{sec:supp-proofs-sec-4}}

\subsection{Proof of Lemma~\ref{lem-Gc-is-typical}}{\label{subsec:proof-lem-4.2}}

Recall Definition~\ref{def-addmisible}. Note that $G$ is a stochastic block model with average degree $\lambda=O(1)$, and $G$ is independent of $\pi_*$. Hence, it suffices to show that with positive probability such a stochastic block model contains no ``undesirable'' subgraph (as described when defining $\mathcal E$). To this end, it suffices to prove the following two items:
\begin{enumerate}
    \item[(i)] With probability $1-o(1)$, $G$ does not contain a subgraph $H$ such that $|V(H)| \leq D^3$ and $\Phi(H) < (\log n)^{-1}$.
    \item[(ii)] With probability at least $c$, $G$ contains no cycle with length no more than $N$.
\end{enumerate}
Denoting by $C_l(G)$ the number of $l$-cycles in $G$, it was known in \cite[Theorem~3.1]{MNS15} that
\begin{equation}{\label{eq-Poisson-convergence}}
    \Big( C_3(G), \ldots, C_N(G) \Big) \Longrightarrow \Big( \operatorname{Pois}(c_3), \ldots, \operatorname{Pois}(c_N) \Big) \,,
\end{equation}
where $\{ \operatorname{Pois}(c_j):3 \leq j \leq N \}$ is a collection of independent Poisson variables with parameters $c_j = \frac{ (1+(k-1)\epsilon^j)\lambda^j }{ 2j }$. Thus, we have Item (ii) holds. We now verify Item (i) via a union bound. For each $1 \leq j \leq D^3$, define
\begin{equation}{\label{eq-def-E(k)}}
    \kappa(j) = \min\Big\{ j' \geq 0: \big( \tfrac{2k^2 \Tilde{\lambda}^2 n}{D^{50}} \big)^{j} \big( \tfrac{ 1000 k^{20} \Tilde{\lambda}^{20} D^{50} }{ n } \big)^{j'} < (\log n)^{-1} \Big\} \,.
\end{equation}
A simple calculation yields $\kappa(j)>j$. In order to prove Item (i), it suffices to upper-bound the probability (by $o(1)$) that there exists $W \subset [n]$ with $|W|=j \leq D^3$ and $|E(G_W)| \geq \kappa(j)$. By a union bound, the aforementioned probability is upper-bounded by
\begin{align}{ \label{eq-upper-bound-prob-G^c} }
    \sum_{j=1}^{D^3} \sum_{W\subset[n],|W|=j} \mathbb{P}_{*}\Big( |E(G_W)| \ge \kappa(j) \Big) \leq \sum_{j=1}^{D^3} \binom{n}{j} \mathbb{P} \Big( \mathbf{B}\Big(\binom{j}{2},\frac{k\lambda}{n} \Big) \geq \kappa(j) \Big) \,,
\end{align}
where $\mathbf{B}\big(\binom{j}{2}, \frac{k\lambda}{n} \big)$ is a binomial variable with parameters $\big(\binom{j}{2},\frac{k\lambda}{n}\big)$, and the inequality holds since this binomial variable stochastically dominates $|E(G_W)|$ for any $W\subset [n]$ with $|W|=j$. For $j\le D^3$, we have $\binom{j}{2}k\lambda/n=o(1)$ and thus by Poisson approximation we get that
\begin{align}{ \label{eq-upper-bound-prob-G^c-.2} }
    \binom{n}{j} \mathbb{P}\Big( \mathbf{B} \Big(\binom{j}{2},\frac{k\lambda}{n} \Big) \geq \kappa(j) \Big) &\leq \frac{n^j}{j!} \cdot \frac{(j^2 \lambda k/n)^{\kappa(j)}}{(\kappa(j))!} \leq n^{j-\kappa(j)} (10\lambda k)^{\kappa(j)} j^{\kappa(j)-j} \nonumber \\
    &\leq 2^{-j} \big( \tfrac{2\Tilde{\lambda}^2 k^2 n}{D^{50}} \big)^{j} \big( \tfrac{ 1000 \Tilde{\lambda}^{20} k^{20} D^{50} }{ n } \big)^{\kappa(j)} \overset{\eqref{eq-def-E(k)}}{\leq} 2^{-j} (\log n)^{-1} \,,
\end{align}
where the third inequality follows from the fact that $\Tilde{\lambda}\geq \lambda,\kappa(j)>j$ and $j \le D^3$. Plugging this estimation into \eqref{eq-upper-bound-prob-G^c}, we get that the right-hand side of \eqref{eq-upper-bound-prob-G^c} is further bounded by
\begin{align}{ \label{eq-upper-bound-prob-G^c-.3} }
    (\log n)^{-1} \sum_{j=1}^{D^3} 2^{-j}=o(1) \,.
\end{align}
This gives Item (i), thereby completing the proof of the lemma.

\subsection{Proof of Lemma~\ref{lem-TV-Pb-Pb'}}{\label{subsec:proof-lem-4.4}}

We first introduce some notation for convenience. Denote $\mathcal P_j(\chi)$ the set of $j$-paths (i.e., paths with $j$ vertices) of $\chi$, and denote 
\begin{equation}{\label{eq-def-CAND}}
\begin{aligned}
    \operatorname{CAND}^{=}_j(\chi) = \big\{ (u,v) \in \operatorname{U} : \chi_{u,v}=0, \sigma_u = \sigma_v, \exists P \in \mathcal P_j(\chi), \operatorname{EndP}(P)=\{u,v\} \big\} \,, \\ 
    \operatorname{CAND}_j^{\neq}(\chi) = \big\{ (u,v) \in \operatorname{U} : \chi_{u,v}=0, \sigma_u \neq \sigma_v, \exists P \in \mathcal P_j(\chi), \operatorname{EndP}(P)=\{u,v\} \big\} \,,
\end{aligned}
\end{equation}
as the sets of non-neighboring pairs $(u,v)$ for which there exists a $j$-path connecting this pair. These sets are candidates for the edges in $E(G) \setminus E(G')$. For a fixed labeling $\sigma \in [k]^n$, we say that $\sigma$ is \emph{typical} if (in what follows, $\chi'$ is the random edge vector corresponding to $G'$ as in Definition~\ref{def-G'-P'})
\begin{align}
    & \big| \#\{ u \in [n] : \sigma_u = i \} - n/k \big| \leq n^{0.9} \mbox{ for all } i \in [k] \,; \label{eq-def-typical-labeling-1} \\
    & \Pb'_{\sigma}\Big( \#\big(\operatorname{CAND}^{=}_j(\chi') \cap \operatorname{CAND}^{=}_l(\chi') \big) \leq 3n^{0.1} \Big) = 1-o(1)\mbox{ for }2 \leq j \neq l \leq N\,; \label{eq-def-typical-labeling-3} \\
    & \Pb'_{\sigma}\Big( \#\big(\operatorname{CAND}^{\neq}_j(\chi') \cap \operatorname{CAND}^{\neq}_l(\chi') \big) \leq 3n^{0.1} \Big) = 1-o(1) \mbox{ for }2 \leq j \neq l \leq N\,; \label{eq-def-typical-labeling-4} \\
    & \Pb'_{\sigma}\Big( \big| \#\operatorname{CAND}^{=}_j(\chi') - \tfrac{n \lambda^{j-1}(1+(k-1)\epsilon^{j-1})}{2k} \big| \leq 2n^{0.9}, \forall\ 2 \leq j \leq N \Big) = 1-o(1) \,; \label{eq-def-typical-labeling-5} \\
    & \Pb'_{\sigma} \Big( \big| \#\operatorname{CAND}^{\neq}_j(\chi') - \tfrac{n \lambda^{j-1}(k-1)(1-\epsilon^{j-1})}{2k} \big| \leq 2n^{0.9}, \forall\ 2 \leq j \leq N \Big) = 1-o(1) \,. \label{eq-def-typical-labeling-6}
\end{align}
\begin{claim}{\label{claim-typical-labeling}}
    We have $\nu( \{ \sigma \in [k]^{n} : \sigma \mbox{ is typical} \} )=1-o(1)$.
\end{claim}
\begin{proof}
    Clearly, we have that $\sigma$ satisfies \eqref{eq-def-typical-labeling-1} with probability $1-o(1)$. Suppose $\chi' \sim \Pb_{\sigma}'$ is subsampled from $\chi \sim \Pb_{\sigma}$ (recalling Definition~\ref{def-G'-P'}, this means that $\chi$ is the random edge vector according to $G \sim \mathbb P_\sigma$, and $\chi'$ is the random edge vector according to $G'$ which is obtained from $G$ after appropriate edge removal). Denote 
    \begin{align*}
        \Check{\mathcal P}^{=}_j(\chi) &= \Big\{ P \in \mathcal P_j(\chi):   \operatorname{EndP}(P) = \{ u,v \} \text{ for some }(u,v)\mbox{ in }\operatorname{U}\text{ with } \chi_{u,v}=0, \sigma_u=\sigma_v \Big\} \,; \\
        \Check{\mathcal P}^{\neq}_j(\chi) &= \Big\{ P \in \mathcal P_j(\chi):   \operatorname{EndP}(P) = \{ u,v \} \text{ for some }(u,v)\mbox{ in }\operatorname{U}\text{ with } \chi_{u,v}=0, \sigma_u\neq \sigma_v \Big\} \,.
    \end{align*}
    (Note that it is possible that $\# \Check{\mathcal P}^{=}_j \neq \# \operatorname{CAND}^{=}_j$ since pairs in $\operatorname{CAND}^{=}_j$ may correspond to multiple paths in $\Check{\mathcal P}^{=}_j$.)
    Recalling Definition~\ref{def-G'-P'} and applying a union bound (over all $B_i$'s in Definition~\ref{def-G'-P'}), we have for all $\sigma\in[k]^n$
    \begin{align}
        &\Pb_{\sigma}'\Big( \#\big\{ e \in \operatorname{U} : \chi_e > \chi'_{e} \big\} < n^{0.1} \Big),\  \Pb_{\sigma}'\Big( \#\big( \mathcal P_j(\chi) \setminus \mathcal P_j(\chi') \big) < n^{0.1} \Big) =1-o(1)\,,\label{eq-diff-chi-chi'}\\
        &\Pb_{\sigma}\Big( |\#\Check{\mathcal P}^{=}_j(\chi) - \#\operatorname{CAND}^{=}_j(\chi)| \leq n^{0.1} \Big) =1-o(1)\label{eq-diff-cyc-cand} \,.
    \end{align}
    In addition, it can be shown by Markov inequality that for all $\sigma \in [k]^n$ and $2\leq j\neq l\leq N$, with $\mathbb P'_\sigma$-probability $1-o(1)$ we have that 
    \begin{align*}
        & \#\big\{ (P_1,P_2) : P_1 \in \Check{\mathcal P}^{=}_j(\chi'), P_2 \in \Check{\mathcal P}^{=}_l(\chi'), \operatorname{EndP}(P_1) = \operatorname{EndP}(P_2) \big\}  \leq n^{0.1}  \,;  \\
        & \#\big\{ (P_1,P_2) : P_1 \in \Check{\mathcal P}^{\neq}_j(\chi'), P_2 \in \Check{\mathcal P}^{\neq}_l(\chi'), \operatorname{EndP}(P_1) = \operatorname{EndP}(P_2) \big\} \leq n^{0.1}  \,.
    \end{align*}
    Thus \eqref{eq-def-typical-labeling-3} and \eqref{eq-def-typical-labeling-4} hold for all $\sigma \in [k]^n$.
    We next deal with \eqref{eq-def-typical-labeling-5}. To this end, by \eqref{eq-diff-cyc-cand} and \eqref{eq-diff-chi-chi'}, it suffices to show that the measure of $\sigma \in [k]^n$ such that
    \begin{align*}
        \Pb_{\sigma} \Big( \big| \#\Check{\mathcal P}^{=}_j(\chi) - \tfrac{n \lambda^{j-1}(1+(k-1)\epsilon^{j-1})}{2k} \big| \leq n^{0.9}, \forall\ 3 \leq j \leq N \Big) = 1-o(1)
    \end{align*}
    is $1-o(1)$. Note that
    \begin{align*}
        & \mathbb E_{\sigma\sim\nu} \Big[ \mathbb E_{\Pb_{\sigma}}\big[ \#\Check{\mathcal P}^{=}_j \big] \Big] = \mathbb E_{\Pb}\Big[ \#\Check{\mathcal P}^{=}_j \Big] = \sum_{ P \in \mathcal P_j (1_{\operatorname{U}})} \Pb\Big( P \in \Check{\mathcal P}^{=}_j \Big) \\
        =\ & \sum_{ \substack{P \in \mathcal P_j (1_{\operatorname{U}})\\ (u,v) \in \operatorname{U}}}\Pb \Big( \operatorname{EndP}(P) = \{ u,v \}, \sigma_u = \sigma_v, \chi_{u,v} = 0, \chi_e =1 \mbox{ for all } e \in E(P) \Big) \\
        =\ & \sum_{ P \in \mathcal P_j (1_{\operatorname{U}})} \tfrac{1}{k} \cdot ( 1- \tfrac{(1+\epsilon(k-1))\lambda}{n} ) \cdot \tfrac{(1+(k-1) \epsilon^{j-1} ) \lambda^{j-1} }{n^{j-1}} = [1+O(\tfrac{1}{n})] \cdot \tfrac{n\lambda^{j-1}(1+(k-1)\epsilon^{j-1})}{2k} \,,
    \end{align*}
    where the fourth equality follows from Claim~\ref{claim-expectation-over-chain}. We now estimate the second moment. We have
    \begin{align*}
        \mathbb E_{\Pb}\Big[ \big( \#\Check{\mathcal P}^{=}_j \big)^2 \Big] = \sum_{ P_1 , P_2 \in \mathcal P_j (1_{\operatorname{U}}) } \Pb( P_1,P_2 \in \Check{\mathcal P}^{=}_j) \,.
    \end{align*}
    For $|V(P_1) \cap V(P_2)|=m \geq 1$, we have (note that $|E(P_1) \cap E(P_2)| \leq m-1$)
    \begin{align*}
        \Pb( P_1,P_2 \in \Check{\mathcal P}^{=}_j ) \leq \Pb( \chi_e=1 \mbox{ for all } e \in E(P_1) \cup E(P_2) ) \leq \big( \tfrac{k\lambda}{n} \big)^{ 2(j-1)-m+1 } \,.
    \end{align*}
    Since the number of pairs $(P_1,P_2)$ with $|V(P_2) \cap V(P_2)|=m$ is at most $j^m n^{2j-m}$, we have that 
    \begin{align*}
        \sum_{ V(P_1) \cap V(P_2) \neq \emptyset } \Pb( P_1,P_2 \in \Check{\mathcal P}^{=}_j) &\leq \sum_{m=1}^{j} j^m n^{2j-m} \big( \tfrac{k\lambda}{n} \big)^{ 2(j-1)-m+1 } \\ &\leq n^{-0.4} \Big( \tfrac{n\lambda^{j-1}(1+(k-1)\epsilon^{j-1})}{2k} \Big)^2 \,.
    \end{align*} 
    In addition, for $V(P_1) \cap V(P_2) = \emptyset$, we have $\Pb( P_1,P_2 \in \Check{\mathcal P}^{=}_j ) = \Pb( P_1 \in \Check{\mathcal P}^{=}_j )^2$. Thus, we have 
    \begin{align*}
        \mathbb E_{\sigma\sim\nu} \Big[ \big( \#\Check{\mathcal P}^{=}_j - \tfrac{n\lambda^{j-1}(1+(k-1)\epsilon^{j-1})}{2k} \big)^2 \Big] \leq n^{-0.4} \cdot  \Big( \tfrac{n\lambda^{j-1}(1+(k-1)\epsilon^{j-1})}{2k} \Big)^2 \,,
    \end{align*}
    and we can deduce \eqref{eq-def-typical-labeling-5} by Chebyshev inequality. The requirement \eqref{eq-def-typical-labeling-6} can be dealt with in a similar manner.
\end{proof}

\begin{lemma}{\label{lem-Poisson-approx}}
    Fix a typical $\sigma\in [k]^n$. For $G^{\sigma} \sim \Pb_{\sigma}$ we get
    $$ \Big( C_3(G^{\sigma}), \ldots, C_N(G^{\sigma}) \Big) \Longrightarrow \Big( \operatorname{Pois}(c_3), \ldots, \operatorname{Pois}(c_N) \Big) \,,$$
    where $\{ \operatorname{Pois}(c_j):3 \leq j \leq N \}$ is a collection of independent Poisson variables with parameters $ \tfrac{ (1+(k-1)\epsilon^j)\lambda^j }{ 2j }$. 
\end{lemma}
\begin{proof}
    Let $G \sim \Pb$ and recall that $C_j(G)$ is the number of $j$-cycles in $G$. Recalling \eqref{eq-Poisson-convergence}, it suffices to show that for typical $\sigma$ we have
    \begin{equation}{\label{eq-TV-cycle-distribution}}
        \operatorname{TV}\Big( (C_3(G),\ldots,C_N(G)), (C_3(G^{\sigma}),\ldots,C_N(G^{\sigma})) \Big) = o(1) \,.
    \end{equation}
    From Claim~\ref{claim-typical-labeling}, it suffices to show that for arbitrary typical $\sigma'$, there exists a coupling $\check{\Pb}$ of $G^{\sigma} \sim \Pb_{\sigma}$ and $G^{\sigma'} \sim \Pb_{\sigma'}$ such that
    \begin{equation}\label{eq-dist-cycles-identical-on-typical-sigma}
        \check{\Pb}\Big( \big( C_3(G^\sigma), \ldots, C_N(G^\sigma ) \big) \neq \big( C_3(G^{\sigma'}), \ldots, C_N(G^{\sigma'}) \big) \Big) = o(1).
    \end{equation}
    Define $\operatorname{DIF}(\sigma,\sigma') = \{ i \in [n] : \sigma_i \neq \sigma'_i \}$. Since the distribution of $\big( C_3(G^\sigma ), \ldots, C_N(G^\sigma ) \big)$ is invariant under any permutation of $\sigma$, by \eqref{eq-def-typical-labeling-1} it suffices to show \eqref{eq-dist-cycles-identical-on-typical-sigma} assuming that $| \operatorname{DIF}(\sigma, \sigma')| \leq 2kn^{0.9}$. We couple $G^{\sigma} \sim \Pb_{\sigma}$ and $G^{\sigma'} \sim \Pb_{\sigma'}$ as follows: for each $(i,j) \in \operatorname{U}$, we independently sample a random variable $x_{i,j} \sim \mathcal U[0,1]$, and then for $y \in \{ \sigma , \sigma' \}$ take $G^{y}_{i,j} = 1$ if and only if $x_{i,j} \leq \frac{1+\epsilon\omega(y_i ,y_j)}{n}$. Let $\check{\Pb}$ be the law of $(G^{\sigma},G^{\sigma'})$. Since for any $3\le j\le N$, $y \in \{ \sigma,\sigma'\}$ and $u \in \operatorname{DIF}(\sigma , \sigma')$, we have
    \begin{align*}
        \check{\Pb}\Big(u\in V(\mathtt C_j(G^y)) \Big) \leq n^{j-1} (\tfrac{\lambda k}{n})^{j} = o(n^{-0.9}) \,,
    \end{align*}
    it follows that
    \begin{align*}
    &\check{\Pb}\Big( \big( C_3(G^\sigma), \ldots, C_N(G^\sigma) \big) \neq \big( C_3(G^{\sigma'}), \ldots, C_N(G^{\sigma'}) \big) \Big)\\
    \leq &\ \check{\Pb}\Big( \operatorname{DIF}(\sigma , \sigma') \cap \big( \cup_{3\le j\le N , y \in \{ \sigma, \sigma'\} } V(\mathtt C_j(G^{y})) \big) \neq \emptyset \Big) \\
    \leq &\  2N|\operatorname{DIF}(\sigma,\sigma')| \cdot  o(n^{-0.9})=o(1).
    \end{align*}
    Therefore we have verified \eqref{eq-dist-cycles-identical-on-typical-sigma}, finishing our proof of  Lemma~\ref{lem-Poisson-approx}.
\end{proof}
Let $\Pb_{*,\sigma}=\Pb_*(\cdot \mid \sigma_*=\sigma)$ and define $\Pb_{*,\sigma}'$ in the similar manner.
Based on Claim~\ref{claim-typical-labeling}, it suffices to show that $\operatorname{TV}\big( \Pb_{*,\sigma}((A,B)\in \cdot \mid \mathcal E ), \Pb'_{*,\sigma}((A,B)\in\cdot) \big)=o(1)$ for all typical $\sigma$. Let $G=G({\sigma})$ and $G'=G'(\sigma)$ be two parent graphs sampled from $\Pb_{*,\sigma}$ and $\Pb'_{*,\sigma}$ respectively (and coupled naturally via the mechanism in Definition~\ref{def-G'-P'}). From the data processing inequality, it suffices to show that $\operatorname{TV}\big(\mathbb P_{*,\sigma} (G\in\cdot\mid\mathcal E) ,\mathbb{P}_{*,\sigma}' (G'\in\cdot)\big)=o(1)$. Denote $\Gc$ the event that $G$ does not contain any self-bad subgraph $H$ such that $|V(H)| \le D^3$ and that the number of cycles of length at most $N$ is at most $\log n$ (note that sometimes we also view $\mathcal G$ as a collection of vectors that correspond to edges in $G$ satisfying $\mathcal G$). It is known from \eqref{eq-upper-bound-prob-G^c} that $\Pb_{*,\sigma}'(\Gc) =1-o(1)$ (the label $\sigma$ does not matter here since the stochastic domination employed in \eqref{eq-upper-bound-prob-G^c} holds for all $\sigma$). By the triangle inequality 
\begin{equation*}
    \operatorname{TV}\big(\mathbb P_{*,\sigma} (G\in\cdot\mid\mathcal E) ,\mathbb{P}_{*,\sigma}'(G'\in\cdot)\big) \leq \operatorname{TV}\big(\mathbb P_{*,\sigma}(G \in \cdot \mid \mathcal E), \mathbb P_{*,\sigma}'(G'\in\cdot \mid \Gc)\big) + \mathbb P_{*,\sigma}'(\Gc^c) \,,
\end{equation*}
in order to prove Lemma~\ref{lem-TV-Pb-Pb'} it suffices to show  
\begin{align}\label{ref-for-end-of-lem-TV-Pb'-Pb}
    \operatorname{TV}\big(\mathbb P_{*,\sigma}(G\in\cdot\mid\mathcal{E}),\mathbb P_{*,\sigma}'(G'\in\cdot\mid\mathcal{G}))\big)=o(1) \,.
\end{align}
Denote $\mathtt p = \frac{ (1+\epsilon(k-1))\lambda }{ n }$ and $\mathtt q = \frac{ (1-\epsilon)\lambda }{ n }$. For any $\chi \in \{ 0,1 \}^{\operatorname{U}}$, denote 
\begin{align*}
    \mathtt E_{1,=}(\chi) = \# \big\{ (i,j) \in \operatorname{U}: \chi_{i,j}=1, \sigma_i=\sigma_j \big\} \,; \quad
    \mathtt E_{1,\neq}(\chi) = \#\big\{ (i,j) \in \operatorname{U}: \chi_{i,j}=1, \sigma_i\neq\sigma_j \big\} \,; \\
    \mathtt E_{0,=}(\chi) = \# \big\{ (i,j) \in \operatorname{U}: \chi_{i,j}=0, \sigma_i=\sigma_j \big\} \,; \quad
    \mathtt E_{0,\neq}(\chi) = \# \big\{ (i,j) \in \operatorname{U}: \chi_{i,j}=0, \sigma_i\neq\sigma_j \big\} \,. 
\end{align*}
Note that for $\chi \in \{0, 1\}^{\operatorname{U}}$ such that $\mathcal E$ holds, we have
\begin{align}
    & \Pb_{*,\sigma}( G=\chi \mid \mathcal E ) = \frac{ \Pb_{*,\sigma}(G=\chi) }{ \Pb_{*,\sigma}(\mathcal E) } = \frac{ \mathtt p^{ \mathtt E_{1,=}(\chi) } \mathtt q^{ \mathtt E_{1,\neq}(\chi) } (1-\mathtt p)^{ \mathtt E_{0,=}(\chi) } (1-\mathtt q)^{ \mathtt E_{0,\neq}(\chi) } }{ \mathbb P_{*,\sigma}(\mathcal E) }  \,. \label{eq-prob-realization-Pb}
\end{align}
In addition, for any $\sigma \in [k]^{n}$, by applying \eqref{eq-upper-bound-prob-G^c}, \eqref{eq-upper-bound-prob-G^c-.2} and \eqref{eq-upper-bound-prob-G^c-.3} in the proof of Lemma~\ref{lem-Gc-is-typical} (the label $\sigma$ does not matter here for the same reason as explained earlier), we have 
\begin{align}\label{detailed-approx-mathcalE-1}
    \Pb_{*,\sigma}(( \mathcal E ^{(1)})^c ) \le o(1). 
\end{align}
Also, using Lemma~\ref{lem-Poisson-approx} we have 
\begin{equation}\label{detailed-approx-mathcalE-2}
    \Pb_{*,\sigma}(\mathcal E^{(2)}) \circeq \mathbb P\Big( (\operatorname{Pois}(c_3),\ldots,\operatorname{Pois}(c_N))=(0,\ldots,0) \Big) \,.
\end{equation}
Combining \eqref{detailed-approx-mathcalE-1} and \eqref{detailed-approx-mathcalE-2}, for a typical $\sigma \in [k]^n$, we have
\begin{align}{\label{eq-prob-mathcal-E}}
    \Pb_{*,\sigma}(\mathcal E) &\circeq \mathbb P\Big( (\operatorname{Pois}(c_3),\ldots,\operatorname{Pois}(c_N))=(0,\ldots,0) \Big) \circeq \prod_{ j=3 }^{N} e^{ - \frac{ (1+(k-1)\epsilon^j)\lambda^j }{ 2j } } \,. 
\end{align}
We now estimate $\Pb'_{*,\sigma}(G'=\chi' \mid \Gc)$. Since $\Pb_{*,\sigma} (\mathcal G )=1-o(1)$, we have
\begin{align*}
    \Pb_{*,\sigma}'(G'=\chi' \mid \Gc) & \circeq \sum_{ \chi \in \Gc } \Pb_{*,\sigma}(G=\chi) \cdot \Pb_{*,\sigma}'(G'=\chi' \mid G=\chi) \\
    &\circeq \sum_{ \chi \in \Gc } \mathtt p^{ \mathtt E_{1,=}(\chi) } \mathtt q^{ \mathtt E_{1,\neq}(\chi) } (1-\mathtt p)^{ \mathtt E_{0,=}(\chi) } (1-\mathtt q)^{ \mathtt E_{0,\neq}(\chi) } \cdot \Pb_{*,\sigma}'(G'=\chi' \mid G=\chi) \,.
\end{align*}
And it remains to estimate $\Pb_{*,\sigma}'(G'=\chi' \mid G=\chi)$. For $\chi' \leq \chi$, denote 
\begin{align*}
    \Upsilon(\chi';\chi) = \Big\{ e \in \operatorname{U} : \chi_e=1, \chi_e'=0 \Big\} \mbox{ and } \Xi_j(\chi) = \bigcup_{C \in \mathtt C_j(\chi)}E(C) \,.
\end{align*}
Recall Definition~\ref{def-addmisible}. For $\chi \in \mathcal G$, we have $\Xi_j (\chi)\cap \Xi_l (\chi)= \emptyset$ for $3 \leq j < l \leq N$. Denote $\chi' \lhd \chi$ when $\Upsilon(\chi';\chi) \subset \cup_{j=3}^{N} \Xi_j(\chi)$ and $|\Upsilon(\chi';\chi) \cap E(C)| \leq 1$ for $C \in \cup_{3 \le j \le N} \mathtt C_j(\chi)$ (note the cycles in $\cup_{3 \le j \le N} \mathtt C_j(\chi)$ cannot intersect for $\chi\in \mathcal G$). Then for $\chi \in \mathcal G$ we have
\begin{align*}
    \Pb_{*,\sigma}'(G'=\chi' \mid G=\chi) = \mathbf{1}_{ \{ \chi' \lhd \chi \} } \cdot \prod_{j=3}^{N} (1/j)^{ | \Upsilon (\chi';\chi) \cap \Xi_j(\chi) | }  \,.
\end{align*}
Thus, we have (recall $\Pb_{*,\sigma}(\mathcal G) = 1-o(1)$)
\begin{align*}
    \Pb_{*,\sigma}'(G'=\chi'\mid \Gc) \circeq \sum_{ \substack{ \chi: \chi \in \Gc \\ \chi' \lhd \chi } } \frac{ \mathtt p^{ \mathtt E_{1,=}(\chi) } \mathtt q^{ \mathtt E_{1,\neq}(\chi) } (1-\mathtt p)^{ \mathtt E_{0,=}(\chi) } (1-\mathtt q)^{ \mathtt E_{0,\neq}(\chi) } }{ \prod_{j=3}^{N} j^{ | \Upsilon (\chi';\chi) \cap \Xi_j | } } \,. 
\end{align*}
Denote
\begin{align*}
    \Upsilon_{=}(\chi';\chi)= \{ (i,j) \in \Upsilon(\chi';\chi) : \sigma_i = \sigma_j \} \mbox{ and } \Upsilon_{\neq}(\chi';\chi) = \Upsilon(\chi';\chi) \setminus \Upsilon_{=}(\chi';\chi) \,.
\end{align*}
We then have that  
\begin{align}\label{eq-prob-realization-Pb'}
    & \Pb_{*,\sigma}'(G'=\chi'\mid \Gc) \circeq \mathtt p^{ \mathtt E_{1,=}(\chi') } \mathtt q^{ \mathtt E_{1,\neq}(\chi') } (1-\mathtt p)^{ \mathtt E_{0,=}(\chi')} (1-\mathtt q)^{ \mathtt E_{0,\neq}(\chi') } \sum_{\mathbf{m},\mathbf{n}=0}^{\log n} \Bigg( \prod_{j=3}^{N} \frac{ \mathtt p^{m_j} \mathtt q^{n_j} }{ j^{m_j+n_j} } * \nonumber \\
    & \# \Big\{ \chi \in \Gc : \chi'\leq \chi, | \Upsilon_{=}(\chi';\chi) \cap \Xi_j(\chi) | = m_j, | \Upsilon_{\neq}(\chi';\chi) \cap \Xi_j(\chi) | = n_j \Big\} \Bigg) \,,
\end{align}
where $\mathbf{m}=(m_1,\ldots,m_N)$, $\mathbf{n}=(n_1,\ldots,n_N)$, and the summation indicates summing each entry in $\mathbf{m}$ and $\mathbf{n}$ from $0$ to $\log n$.
We next bound the cardinality for the set in \eqref{eq-prob-realization-Pb'}. To this end, we have (we denote by $\mathsf{CAN}^{=}_j=\#( \operatorname{CAND}^{=}_j(\chi') \setminus \cup_{l \leq N, l \neq j} \operatorname{CAND}^{\neq}_l(\chi') )$ and $\mathsf{CAN}^{\neq}_j=\#( \operatorname{CAND}^{\neq}_j(\chi') \setminus \cup_{l \leq N, l \neq j} \operatorname{CAND}^{\neq}_l(\chi') )$ below)
\begin{align*}
    & \# \Big\{ \chi \in \Gc : \chi'\leq \chi, | \Upsilon_=(\chi';\chi) \cap \Xi_j(\chi) | = m_j, | \Upsilon_{\neq}(\chi';\chi) \cap \Xi_j(\chi) | = n_j \Big\} \\
    \geq \ & \prod_{j=3}^{N} \binom{ \mathsf{CAN}^{=}_j }{ m_j } \prod_{j=3}^{N} \binom{ \mathsf{CAN}^{\neq}_j }{ n_j }  \geq [1+o(1)] \cdot \prod_{j=3}^{N} \frac{ ( \mathsf{CAN}^{=}_j )^{m_j} ( \mathsf{CAN}^{\neq}_j )^{n_j} }{ m_j! n_j! } 
\end{align*}
and 
\begin{align*}
    & \# \Big\{ \chi \in \Gc : \chi'\leq \chi, | \Upsilon_=(\chi';\chi) \cap \Xi_j(\chi) | = m_j, | \Upsilon_{\neq}(\chi';\chi) \cap \Xi_j(\chi) | = n_j \Big\} \\
    \leq \ & \prod_{j=3}^{N} \binom{ \# \operatorname{CAND}^{=}_j(\chi') }{ m_j } \prod_{j=3}^{N} \binom{ \# \operatorname{CAND}^{\neq}_j(\chi') }{ n_j } \\
    \leq \ & \prod_{j=3}^{N} \frac{ ( \# \operatorname{CAND}^{=}_j(\chi') )^{m_j} ( \# \operatorname{CAND}^{\neq}_j(\chi') )^{n_j} }{ m_j! n_j! } \,.
\end{align*}
Denote  
\begin{equation}\label{def-A-CANDs}
\begin{aligned}
    \mathcal{A} = &\Big\{\chi' : \Big| \# \operatorname{CAND}^{=}_j(\chi') - \tfrac{n \lambda^{j-1}(1+(k-1)\epsilon^{j-1})}{2k} \Big| \leq 2n^{0.9} \mbox{ for } j=3,\cdots,N \Big\} \\
    &\bigcap \Big\{\chi' : \Big| \# \operatorname{CAND}^{\neq}_j(\chi') - \tfrac{n \lambda^{j-1}(1-\epsilon^{j-1})}{2k} \Big| \leq 2n^{0.9} \mbox{ for } j=3,\cdots,N \Big\} \\
    &\bigcap \Big\{ \chi': \#\big( \operatorname{CAND}^{i}_j(\chi') \cap \operatorname{CAND}^{i}_{l}(\chi') \big) \leq 3n^{0.1} \mbox{ for } i\in \{ =,\neq \}; j \neq l \Big\} \,.
\end{aligned}
\end{equation}
Since $\sigma$ is typical, we see that $\Pb_{\sigma}'(\mathcal A)=1-o(1)$. In addition, for $\chi' \in \mathcal A$, we have
\begin{align*}
    \mathsf{CAN}^{=}_j,\ \# \operatorname{CAND}^{=}_j(\chi') &\circeq \tfrac{n \lambda^{j-1}(1+(k-1)\epsilon^{j-1})}{2k} \,; \\
    \mathsf{CAN}^{\neq}_j,\ \# \operatorname{CAND}^{\neq}_j(\chi') &\circeq \tfrac{n(k-1) \lambda^{j-1}(1-\epsilon^{j-1})}{2k} \,.
\end{align*}
Thus, for such $\chi'$ we have
\begin{align*}
    & \# \Big\{\chi \in \Gc : \chi'\leq \chi, | \Upsilon_=(\chi';\chi) \cap \Xi_j(\chi) | = m_j, | \Upsilon_{\neq}(\chi';\chi) \cap \Xi_j(\chi) | = n_j \Big\} \\
    \circeq\ & \prod_{j=3}^{N} \frac{ 1 }{ m_j! } \Big( \frac{n \lambda^{j-1}(1+(k-1)\epsilon^{j-1})}{2k} \Big)^{m_j} \prod_{j=3}^{N} \frac{ 1 }{ n_j! } \Big( \frac{n (k-1) \lambda^{j-1}(1-\epsilon^{j-1})}{2k} \Big)^{n_j} \,.
\end{align*}
Plugging this estimation into \eqref{eq-prob-realization-Pb'}, we get that for $\chi'\in \mathcal A$
\begin{align*}
    \Pb_{*,\sigma}'(G'=\chi'\mid &\Gc) \circeq\ \mathtt p^{ \mathtt E_{1,=}(\chi') } \mathtt q^{ \mathtt E_{1,\neq}(\chi') } (1-\mathtt p)^{ \mathtt E_{0,=}(\chi')} (1-\mathtt q)^{ \mathtt E_{0,\neq}(\chi') }  * \\
    &\qquad \sum_{\mathbf{m},\mathbf{n}=0}^{\log n} \prod_{j=3}^{N} \frac{ \lambda^{j(m_j + n_j )} (\frac{(1+(k-1)\epsilon )(1+(k-1)\epsilon^{j-1})}{k})^{m_j} (\frac{(k-1)(1-\epsilon)(1-\epsilon^{j-1})}{k})^{n_j} }{ {(2j)}^{m_j + n_j} m_j! n_j! } \\
    &\ \overset{\eqref{eq-prob-mathcal-E}}{\circeq}\frac{ \mathtt p^{ \mathtt E_{1,=}(\chi') } \mathtt q^{ \mathtt E_{1,\neq}(\chi') } (1-\mathtt p)^{ \mathtt E_{0,=}(\chi')} (1-\mathtt q)^{ \mathtt E_{0,\neq}(\chi') } }{ \mathbb{P}(\mathcal E) } \\
    &\quad \circeq \Pb_{*,\sigma}(G=\chi' \mid \mathcal E) \,.
\end{align*}
Thus, we have for all typical $\sigma$
\begin{align*}
     \operatorname{TV}\big(\mathbb P_{*,\sigma}(G\in\cdot\mid\mathcal{E}),\mathbb P_{*,\sigma}'(G'\in\cdot\mid\mathcal{G}))\big)\leq \Pb_{*,\sigma}'(G' \in \mathcal A^c) + \max_{\chi' \in \mathcal A} \Bigg\{ \Big|  \frac{\Pb_{*,\sigma}'(G'=\chi'\mid \Gc)}{\Pb_{*,\sigma}(G=\chi' \mid \mathcal E)} - 1 \Big| \Bigg\} \,,
\end{align*}
which vanishes, thereby yielding \eqref{ref-for-end-of-lem-TV-Pb'-Pb} as desired.

\subsection{Proof of Claim~\ref{claim-relation-mathtt-M-N}}{\label{subsec:Proof-Claim-4.11}}
Note that
\begin{align*}
    &\sum_{K:H \ltimes K \subset S} \mathtt M(S,K) \mathtt M(K,H) \\
    \overset{\eqref{eq-def-mathtt-M}}{=}&\  \mathtt M(S,H) \sum_{K:H \ltimes K \subset S} \big( \tfrac{D^8}{n^{0.1}} \big)^{\frac{1}{2} ( |\mathcal L(S) \setminus V(K)| + |\mathcal L(K) \setminus V(H)| - |\mathcal L(S) \setminus V(H)| ) }  \,.
\end{align*}
In light of \eqref{eq-def-mathtt-N}, in order to prove Claim~\ref{claim-relation-mathtt-M-N}, it suffices to prove
\begin{equation}{\label{eq-estimates-on-mathtt-P}}
\begin{aligned}
    & \sum_{K:H \ltimes K \subset S} n^{ -0.04( |\mathcal L(S) \setminus V(K)| + |\mathcal L(K) \setminus V(H)| - |\mathcal L(S) \setminus V(H)| ) } \\
    \leq\ & [1+o(1)] \cdot 2^{|\mathfrak{C}(S,H)|} D^{ 10(|\mathcal L(S) \setminus V(H)| + \tau(S) - \tau(H)) } \,.
\end{aligned}
\end{equation}
\begin{figure}[!ht]
    \centering
    \vspace{0cm}
    \includegraphics[height=5.5cm,width=12cm]{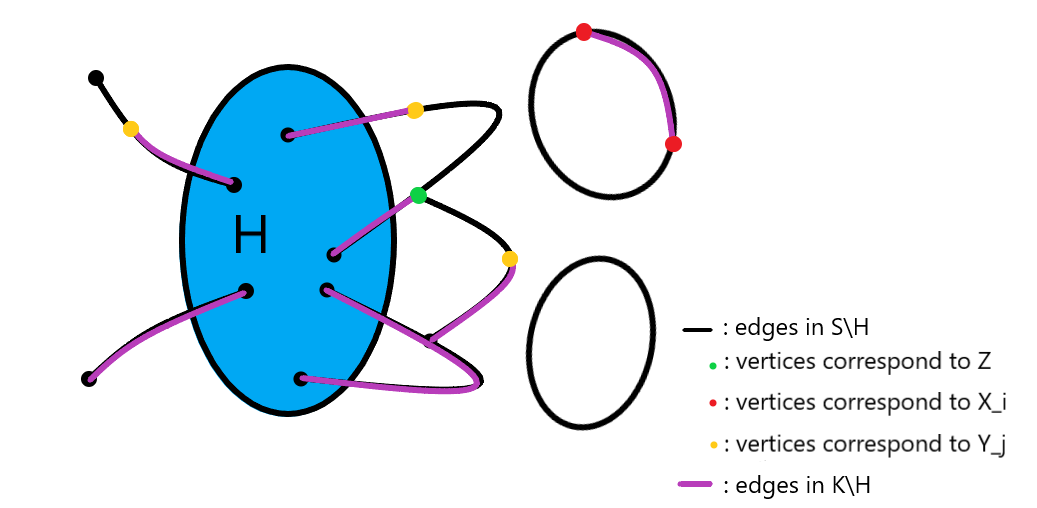}
    \caption{Illustration of the decomposition}
    \label{fig:1}
\end{figure}
To this end, we consider the decomposition of $E(S)\setminus E(H)$ given by Corollary~\ref{cor-revised-decomposition-H-Subset-S}: for $\mathtt m = |\mathfrak{C}(S,H)|$ and some $0 \leq \mathtt t \leq 5(|\mathcal{L}(S)\setminus V(H)|+\tau(H)-\tau(S))$ we can decompose $E(S) \setminus E(H)$ into $\mathtt m$ independent cycles $C_{1}, \ldots, C_{\mathtt m}$ and $\mathtt t$ paths $P_{1}, \ldots, P_{\mathtt t}$. Denote 
\begin{align*}
    X_{\mathtt i}(K) =\ & \#\Big\{ u \in V(C_{\mathtt i}): u \in \mathcal L(K) \setminus V(H)  \Big\} \,, \mathtt 1 \leq \mathtt i \leq \mathtt m \,; \\
    Y_{\mathtt j}(K) =\ & \#\Big\{ u \in V(P_{\mathtt j}) \setminus \operatorname{EndP}(P_{\mathtt j}) : u \in  \mathcal L(K) \setminus V(H) \Big\} \,, \mathtt 1 \leq \mathtt j \leq \mathtt t \,; \\
    Z(K) =\ & \# \Big( \big\{ u \in \cup_{\mathtt j=\mathtt 1}^{\mathtt t} \operatorname{EndP}(P_{\mathtt j}): u \in \big( (\mathcal L(S) \setminus V(K)) \cup (\mathcal L(K) \setminus V(H)) \big) \big\} \setminus \\
     & \big\{ u \in \cup_{\mathtt j=\mathtt 1}^{\mathtt t} \operatorname{EndP}(P_{\mathtt j}): u \in \mathcal L(S) \setminus V(H) \big\} \Big) \,. 
\end{align*}
(See Figure~\ref{fig:1} for an illustration.) Clearly we have $X_{\mathtt i}, Y_{\mathtt j} \geq 0$. We argue that
\begin{equation}{\label{eq-inclusion-leaves}}
    \big( \mathcal L(S) \setminus V(H) \big) \cap \operatorname{EndP}(P_{\mathtt j}) \subset \big( (\mathcal L(S) \setminus V(K)) \cup (\mathcal L(K) \setminus V(H)) \big) \cap \operatorname{EndP}(P_{\mathtt j})  
\end{equation}
for $1 \leq \mathtt j \leq \mathtt t$ and thus we also have $Z \geq 0$. Assume that $u \in (\mathcal L(S) \setminus V(H)) \cap \operatorname{EndP}(P_{\mathtt j})$. If $u \not \in V(K)$, then $u \in \mathcal L(S) \setminus V(K)$. If $u \in V(K) \setminus V(H)$, we see that $u \not \in \mathcal I(K)$ since $H\ltimes K$, which together with $u\in \mathcal L(S)$ implies that $u\in \mathcal L(K)$ (and thus $u\in \mathcal L(K)\setminus V(H)$). Combining the above two arguments leads to \eqref{eq-inclusion-leaves}. 

Note that the vertices in $C_{\mathtt i}$ and $V(P_{\mathtt j}) \setminus \operatorname{EndP}(P_{\mathtt j})$ have degree at least 2 in $S$ and thus they do not belong to $\mathcal L(S)$. By the definition of $X_{\mathtt i}, Y_{\mathtt j}$ and $Z$, we have
\[
    |\mathcal L(S) \setminus V(K)| + |\mathcal L(K) \setminus V(H)| - |\mathcal L(S) \setminus V(H)|= Z+ \sum_{\mathtt i=\mathtt 1}^{\mathtt m} X_{\mathtt i} + \sum_{\mathtt j=\mathtt 1}^{\mathtt t} Y_{\mathtt j} \,.
\]
Thus, the left-hand side of \eqref{eq-estimates-on-mathtt-P} equals (below we write $\mathbf{X}(K) = (X_{\mathtt 1}(K), \ldots, X_{\mathtt m}(K))$ and the same applies to $\mathbf{Y}(K)$, $\mathbf{x}$ and $\mathbf{y}$)
\begin{align*}
    & \sum_{ K: H \ltimes K \subset S } n^{ -0.04( Z(K)+ \sum_{\mathtt i=\mathtt 1}^{\mathtt m} X_{\mathtt i}(K) + \sum_{\mathtt j=\mathtt 1}^{\mathtt t} Y_{\mathtt j}(K) )} \\
    = & \sum_{ z,\mathbf{x},\mathbf{y} \geq 0 } n^{ - 0.04(z+\sum_{\mathtt i=\mathtt 1}^{\mathtt m} x_{\mathtt i} + \sum_{\mathtt j=\mathtt 1}^{\mathtt t} y_{\mathtt j} )} \#\big\{ K: H \ltimes K \subset S, Z(K)=z, \mathbf{X}(K)=\mathbf{x}, \mathbf{Y}(K)=\mathbf{y} \big\} \,,
\end{align*}
where  $\mathbf{x}\geq 0$ means $x_{\mathtt i} \geq 0$ for all $1 \leq \mathtt i \leq \mathtt m$ (and similarly for $\mathbf{y}\geq 0$).
We now bound $\mathrm{CARD}$, the cardinality of the above set as follows. First note that the enumeration of $(\mathcal L(S) \setminus V(K)) \cup (\mathcal L(K) \setminus V(H))$ is bounded by $D^{z+\sum_{\mathtt i=\mathtt 1}^{\mathtt m} x_{\mathtt i} + \sum_{\mathtt j=\mathtt 1}^{\mathtt t} y_{\mathtt j}}$. Since the vertices in $(\mathcal L(S) \setminus V(K)) \cup (\mathcal L(K) \setminus V(H))$ split $C_{\mathtt i}$'s and $P_{\mathtt j}$'s into $\mathtt m' \leq \mathtt m$ independent cycles and at most $\sum_{\mathtt i} X_{\mathtt i} + \sum_{\mathtt j} (Y_{\mathtt j}+1)$ new paths, where each paths/cycles belong to either $K$ or $S \setminus K$, leading to a bound of $2^{\mathtt m+\sum_{\mathtt i=\mathtt 1}^{\mathtt m} x_{\mathtt i} + \sum_{\mathtt j=\mathtt 1}^{\mathtt t} (y_{\mathtt j}+1)}$ on the enumeration. Thus for $|E(S)| \leq D$ we have
\begin{align*}
    \mathrm{CARD} \leq D^{ z+ \sum_{\mathtt i=\mathtt 1}^{\mathtt m} x_{\mathtt i} + \sum_{\mathtt j=\mathtt 1}^{\mathtt t} y_{\mathtt j} } 2^{\mathtt m+\sum_{\mathtt i=\mathtt 1}^{\mathtt m} x_{\mathtt i} + \sum_{\mathtt j=\mathtt 1}^{\mathtt t} (y_{\mathtt j}+1)}  \,.
\end{align*}
Therefore, the left-hand side of \eqref{eq-estimates-on-mathtt-P} is bounded by
\begin{align*}
    2^{\mathtt m} (2D)^{\mathtt t} \sum_{ z,x,y \geq 0 } \big( \tfrac{2D}{n^{0.04}} \big)^{ z+ \sum_{\mathtt i=\mathtt 1}^{\mathtt m} x_{\mathtt i} + \sum_{\mathtt j=\mathtt 1}^{\mathtt t} y_{\mathtt j} } \circeq 2^{\mathtt m} (2D)^{\mathtt t} \,,
\end{align*}
concluding \eqref{eq-estimates-on-mathtt-P} by recalling that $\mathtt m=|\mathfrak{C}(S,H)|$ and $\mathtt{t} \le 5 (\tau (S)- \tau (H)) $.

\subsection{Proof of Lemma~\ref{lem-first-moment-bound}}{\label{subsec:Proof-Lem-B.1}}

To prove Lemma~\ref{lem-first-moment-bound}, by averaging over the conditioning of community labels we have that the left-hand side of \eqref{eq-lem-B.1} is bounded by 
\begin{align*}
    s^{|E(H)|} \mathbb{E}_{\sigma\sim\nu} \Bigg[ \prod_{e \in E(H)} \mathbb E_{\Pb_{\sigma}}\Big[ \frac{ (G_{e}-\tfrac{\lambda}{n})^2 }{ \lambda/n } \Big] \Bigg] = s^{|E(H)|} \mathbb{E}_{\sigma\sim\nu} \Bigg[ \prod_{(i,j) \in E(H)} \big( 1+\epsilon \omega(\sigma_i,\sigma_j) \big) \Bigg] \,.
\end{align*}
Thus, it suffices to show that for any admissible $H$ we have
\begin{equation}{\label{eq-important-inequality-first-moment}}
    s^{|E(H)|} \mathbb{E}_{\sigma\sim\nu} \Bigg[ \prod_{(i,j) \in E(H)} \big( 1+\epsilon \omega(\sigma_i,\sigma_j) \big) \Bigg] \leq O(1)\cdot(\sqrt{\alpha}-\delta/4)^{|E(H)|} \,.
\end{equation}
Now we provide the proof of \eqref{eq-important-inequality-first-moment}, thus finishing the proof of Lemma~\ref{lem-first-moment-bound}.

\begin{proof}[Proof of \eqref{eq-important-inequality-first-moment}]
    Denoting $\mathsf{Core}(H)$ the $2$-core of $H$, we can write $H$ as $H=\mathsf{Core}(H) \cup \big( \cup_{\mathtt i=1}^{\mathtt t} T_{\mathtt i} \big)$, where $\{ T_{\mathtt i}: 1 \leq \mathtt i \leq \mathtt t \}$ are disjoint rooted trees such that $V(T_{\mathtt i}) \cap V( \mathsf{Core}(H) )$ is (the singleton of) the root of $T_{\mathtt i}$, denoted as $\mathfrak R(T_{\mathtt i})$. Clearly, conditioned on $\{ \sigma_{u} : u \in V(\mathsf{Core}(H)) \}$, we have that
    \[
    \Bigg\{  \prod_{(i,j) \in E(T_{\mathtt i})} \big( 1+\epsilon \omega(\sigma_i,\sigma_j) \big) : 1 \leq \mathtt i \leq \mathtt t  \Bigg\}
    \]
    are conditionally independent. In addition, since for any tree $T$ we have
    \begin{align*}
        \mathbb{E} \Big[  \prod_{(i,j) \in E(T)} \big( 1+\epsilon \omega(\sigma_i,\sigma_j) \big) \mid \sigma_{\mathfrak R(T)} \Big] = 1 \,,
    \end{align*}
    we then get that
    \begin{align*}
        \mathbb{E} \Big[ \prod_{\mathtt i=1}^{\mathtt t}  \prod_{(i,j) \in E(T_{\mathtt i})} \big( 1+\epsilon \omega(\sigma_i,\sigma_j) \big) \mid \{ \sigma_u : u \in V(\mathsf{Core}(H)) \} \Big] = 1 \,.
    \end{align*}
    Therefore (noting that the product over edges outside of the 2-core can be decomposed as product over edges in $T_{\mathtt i}$ for $1\leq \mathtt i\leq \mathtt t$),
    \begin{align*}
        & \mathbb{E} \Big[ \prod_{(i,j) \in E(H)} \big( 1+\epsilon \omega(\sigma_i,\sigma_j) \big) \Big] \\
        =\ & \mathbb E \Bigg[ \mathbb{E} \Big[ \prod_{(i,j) \in E(H)} \big( 1+\epsilon \omega(\sigma_i,\sigma_j) \big) \mid \{ \sigma_u : u \in V(\mathsf{Core}(H)) \} \Big] \Bigg] \\
        =\ & \mathbb E \Bigg[ \prod_{(i,j) \in E(\mathsf{Core}(H))} \big( 1+\epsilon \omega(\sigma_i,\sigma_j) \big) \prod_{\mathtt i=1}^{\mathtt t} \mathbb{E} \Big[ \prod_{(i,j) \in E(T_{\mathtt i})} \big( 1+\epsilon \omega(\sigma_i,\sigma_j) \big) \mid \sigma_{\mathfrak R(T_\mathtt i)} \Big] \Bigg] \\
        =\ & \mathbb E \Bigg[ \prod_{(i,j) \in E(\mathsf{Core}(H))} \big( 1+\epsilon \omega(\sigma_i,\sigma_j) \big) \Bigg] \,.
    \end{align*}
    Since in addition, $s<\sqrt{\alpha}-\delta$, it suffices to show that for any admissible $H$ with at most $D$ edges and with minimum degree at least 2, we have \eqref{eq-important-inequality-first-moment} holds for $H$. For any such graph $H$ (note that in this case $\mathcal I(H)=\emptyset$), by applying Corollary~\ref{cor-revised-decomposition-H-Subset-S} with $\emptyset \ltimes H$ (in place of $H\ltimes S$ as in the corollary-statement), we see that $H$ can be decomposed into $\mathtt m$ independent cycles $C_1, \ldots C_{\mathtt m}$ and $\mathtt t$ paths $P_{1}, \ldots, P_{\mathtt t}$ satisfying Item (i)--(iii) in Corollary~\ref{cor-revised-decomposition-H-Subset-S}. In particular, since $H$ is admissible, recalling Definition 4.1 we have $\mathtt t \leq 5\tau(H) = O(1)$ and $|E(C_{\mathtt i})| \geq N$. Keeping this in mind, we now proceed to show \eqref{eq-important-inequality-first-moment}. Denoting $\mathsf{End}=\cup_{\mathtt i=\mathtt 1}^{\mathtt t} \operatorname{EndP}(P_{\mathtt i})$, conditioned on $\{ \sigma_{\mathsf u}: \mathsf u \in \mathsf{End} \}$ we have 
    \begin{align*}
        \Big\{ \prod_{(i,j) \in E(C_{\mathtt i})} \big( 1+\epsilon \omega(\sigma_i,\sigma_j) \big), \prod_{(i,j) \in E(P_{\mathtt j})} \big( 1+\epsilon \omega(\sigma_i,\sigma_j) \big) : \mathtt 1 \leq \mathtt i \leq \mathtt m, \mathtt 1 \leq \mathtt j \leq \mathtt t \Big\}
    \end{align*}
    are conditionally independent. In addition, by Claim~\ref{claim-expectation-over-chain} we have
    \begin{align*}
        & s^{|E(C_{\mathtt i})|} \mathbb{E}_{\sigma\sim\nu} \Big[ \prod_{(i,j) \in E(C_{\mathtt i})} \big( 1+\epsilon \omega(\sigma_i,\sigma_j) \big) \mid \{ \sigma_{\mathsf u} : \mathsf u \in \mathsf{End} \} \Big] \\
        =\ & s^{|E(C_{\mathtt i})|} (1+(k-1)\epsilon^{|E(C_{\mathtt i})|}) \overset{(12)}{\leq} (\sqrt{\alpha}-\delta/2)^{|E(C_{\mathtt i})|} \,,
    \end{align*}
    where in the last inequality we also used $|E(C_{\mathtt i})| \geq N$. Furthermore, using Claim~\ref{claim-expectation-over-chain} and the fact that $|\omega(\sigma_u, \sigma_v)| \leq k-1$ we get that
    \begin{align*}
        & s^{|E(P_{\mathtt j})|} \mathbb{E}_{\sigma\sim\nu} \Big[ \prod_{(i,j) \in E(P_{\mathtt j})} \big( 1+\epsilon \omega(\sigma_i,\sigma_j) \big) \mid \{ \sigma_{\mathsf u} : \mathsf u \in \mathsf{End} \} \Big] \\
        \leq\ & s^{|E(P_{\mathtt j})|} (1+(k-1)\epsilon^{|E(P_{\mathtt j})|}) \leq k (\sqrt{\alpha}-\delta/2)^{|E(P_{\mathtt j})|} \,,
    \end{align*}
    where the last inequality follows from $s<\sqrt{\alpha}-\delta/2$ and $\epsilon<1$. Putting these together, we have (note that $\tau(H) \leq o(1) \cdot |E(H)|$ as $H$ is admissible)
    \begin{align}
        & s^{|E(H)|} \mathbb{E}_{\sigma\sim\nu} \Big[ \prod_{(i,j) \in E(H)} \big( 1+\epsilon \omega(\sigma_i,\sigma_j) \big) \Big] \nonumber \\ 
        \leq\ & \prod_{\mathtt i=\mathtt 1}^{\mathtt m} (\sqrt{\alpha}-\delta/2)^{|E(C_{\mathtt i})|} \prod_{\mathtt j=\mathtt 1}^{\mathtt t} k (\sqrt{\alpha}-\delta/2)^{|E(P_{\mathtt j})|} \nonumber \\
        \overset{\mathtt t \leq 5\tau(H)}{\leq} & k^{5\tau(H)} \cdot (\sqrt{\alpha}-\delta/2)^{|E(H)|} \leq (\sqrt{\alpha}-\delta/4)^{|E(H)|} \label{admissible-product-reduction-1} \,,
    \end{align}
    which yields the desired result.
\end{proof}

\subsection{Proof of Claim~\ref{claim-condition-on-G-par-case-1}}{\label{subsec:Proof-Claim-B.5}}

Recall \eqref{eq-def-mathtt-W}. We divide the assumption $\mathtt W \not \subset \mathcal B(\chi)$ into two cases. 

$\textbf{Case 1}$: There exists $u \in \mathcal L(S_1) \setminus V(K_1) \subset \mathtt V$ such that $u \not \in \mathcal B(\chi)$. For each $\varkappa \in [k]^{\mathtt V}$ and $i \in [k]$, define $\varkappa_{ i(u) } \in [k]^{\mathtt V}$ such that $\varkappa_{ i(u) } (v) = \varkappa(v)$ for $v \in \mathtt V \setminus \{ u \}$ and $\varkappa_{ i(u) }(u)=i$. Since $u \not \in \mathcal B(\chi)$, we know that in $\chi \oplus \{ \mathtt 1_{\mathtt E} \}$, there is neither small cycle nor self-bad graph containing $u$. Thus, given $G(\operatorname{par})|_{\operatorname{U} \setminus \mathtt E} = \chi$, in each $G(\varkappa)$ for $\varkappa \in [k]^{\mathtt V}$ there is no small cycle nor self-bad graph containing $u$. Thus we have given $G(\operatorname{par})|_{\operatorname{U} \setminus \mathtt E} = \chi$,
\[
    G'(\varkappa_{i(u)}) \mbox{ is equal in distribution with } G'(\varkappa_{j(u)}) \mbox{ for all } \varkappa \in [k]^{\mathtt V}, i,j \in [k] \,.
\]
Thus, the left-hand side of \eqref{eq-conditional-expectation} equals
\begin{align*}
    \frac{1}{k^{|\mathtt V|+1}} \sum_{i \in [k]} \sum_{\varkappa \in [k]^{\mathtt V}} h_{\varkappa_{i(u)}\oplus\gamma}(S_1,S_2;K_1,K_2) \mathbb E_{\widetilde \Pb} \Big[  \varphi_{\gamma;K_1,K_2;H}(G'(\varkappa_{i(u)})) \mid G(\operatorname{par})|_{\operatorname{U} \setminus \mathtt E} = \chi \Big] \,.
\end{align*}
Noticing from \eqref{eq-def-h-sigma} that 
\begin{align*}
    \sum_{i \in [k]} h_{\varkappa_{i(u)}\oplus\gamma}(S_1,S_2;K_1,K_2) = 0 \,,
\end{align*}
we have that the left-hand side of \eqref{eq-conditional-expectation} must cancel to $0$. The result follows similarly if there exists $u \in \mathcal L(S_2) \setminus V(K_2)$ such that $u \not \in \mathcal B(\chi)$.

$\textbf{Case 2}$: There exists $u \in V(K_1) \setminus V(H)$ such that $u \not \in \mathcal B(\chi)$. We argue that now we must have
\begin{align}
    \mathbb E_{\widetilde \Pb} \Big[  \varphi_{\gamma;K_1,K_2;H}(G'(\varkappa)) \mid G(\operatorname{par})|_{\operatorname{U} \setminus \mathtt E} = \chi \Big] =0 \,, \forall \varkappa \in [k]^{\mathtt V} \,. \label{eq-cancelling-conditional-expectation}
\end{align}
In fact, since $H \ltimes K_1$, there exists $e \in E(K_1) \setminus E(H)$ such that $e=(u,v)$. Since in $\chi \oplus \{ \mathtt 1_{\mathtt E} \}$ there is no small cycle nor self-bad graph containing $u$, we know that for any realization $G(\operatorname{par})$ such that $G(\operatorname{par})|_{\operatorname{U} \setminus \mathtt E} = \chi$, for each $\varkappa \in [k]^{\mathtt V}$ there is no small cycle nor self-bad graph containing $u$ in $G(\varkappa)$. Therefore 
\[
    G'(\varkappa)_{u,v} = G(\varkappa)_{u,v} \mbox{ for all } \varkappa \in [k]^{\mathtt V} 
\]
and thus $G'(\varkappa)_{u,v}$ is conditionally independent with $\{ G'(\varkappa)_{i,j}: (i,j) \in \operatorname{U} \setminus (u,v) \}$. Thus recalling \eqref{eq-def-varphi} we see that the conditional expectation in \eqref{eq-cancelling-conditional-expectation} cancels to $0$. The result follows similarly if there exists $u \in V(K_2) \setminus V(H)$ such that $u \not \in \mathcal B(\chi)$.

\subsection{Proof of Claim~\ref{claim-condition-on-G-par-case-2}}{\label{subsec:Proof-Claim-B.6}}

Consider the graph $\check{K}_1, \check{K}_2$ such that $E(\check{K}_i)=E(K_i)$ and $V(\check{K}_i)=V(K_i) \cup (\mathcal{B}(\chi) \setminus \mathtt W)$. Recalling \eqref{eq-def-mathtt-W} and Definition~\ref{def-bad-vertex-set}, we have
\begin{align*}
    \mathcal B(\chi) \setminus \mathtt W \subset (V(S_1) \setminus V(K_1)) \cup (V(S_2)\setminus V(K_2)) \,,
\end{align*}
and thus $\tau(\check{K}_i)= \tau(K_i)-|\mathcal B(\chi) \setminus \mathtt W|=\tau(K_i)-\ell$. Noting that $\check K_1 \ltimes S_1$ and $\check K_2 \ltimes S_2$ since $\mathcal I(S_1)=\mathcal I(S_2)=\emptyset$, we can decompose $E(S_1) \setminus E(\check K_1)$ into $\mathtt t$ paths $P_{1}, \ldots, P_{\mathtt t}$ and $\mathtt x$ independent cycles $C_{\mathtt 1},\ldots,C_{\mathtt x}$ satisfying Items (i)--(iii) in Corollary~\ref{cor-revised-decomposition-H-Subset-S}, and similarly we can decompose $E(S_2) \setminus E(\check K_2)$ into into $\mathtt r$ paths $Q_1, \ldots, Q_{\mathtt r}$ and $\mathtt y$ independent cycles $D_{\mathtt 1},\ldots,D_{\mathtt y}$. What's more, since $C_{\mathtt i} \in \mathfrak{C}(S_1,K_1)$ we have $V(C_{\mathtt i}) \cap V(H) \subset V(C_{\mathtt i}) \cap V(K_1)=\emptyset$, and thus from $S_1 \cap S_2 = H$ we have $V(C_{\mathtt i}) \cap V(Q_{\mathtt j})= V(C_{\mathtt i}) \cap V(D_{\mathtt j'})=\emptyset$. This yields that
\begin{equation}{\label{eq-disjointness-cycles}}
    V(C_{\mathtt i}) \cap \Big( \big(\cup_{ \mathtt m \neq \mathtt i } V(C_{\mathtt m}) \big) \cup \big( \cup_{ \mathtt i' } V(P_{\mathtt i'}) \big) \cup \big( \cup_{\mathtt j} V(D_{\mathtt j}) \big) \cup \big( \cup_{\mathtt j'} V(Q_{\mathtt j'}) \big) \Big) = \emptyset 
\end{equation}
and similar results hold for $D_{\mathtt j}$. Also from Item (ii) in Corollary~\ref{cor-revised-decomposition-H-Subset-S} we have $V(P_{\mathtt i'}) \cap V(H) \subset V(P_{\mathtt i'}) \cap V(K_1) \subset \operatorname{EndP}(P_{\mathtt i'})$, thus 
\begin{equation}{\label{eq-almost-disjointness-paths}}
    V(P_{\mathtt i'}) \cap \Big( \big( \cup_{ \mathtt i } V(C_{\mathtt i})\big) \cup \big( \cup_{ \mathtt m' \neq \mathtt i' } V(P_{\mathtt m'})\big) \cup \big( \cup_{\mathtt j} V(D_{\mathtt j}) \big) \cup \big( \cup_{\mathtt j'} V(Q_{\mathtt j'}) \big) \Big) \subset \operatorname{EndP}(P_{\mathtt i'})  
\end{equation}
and similar results hold for $Q_{\mathtt j'}$. In addition, since $S_1,S_2$ are admissible, we must have $|V(C_{\mathtt i})|,|V(D_{\mathtt j})| \geq N$ for all $1\le \mathtt i \le \mathtt x$ and $1 \le \mathtt j \le \mathtt y$. Denote 
\begin{align*}
    \mathtt S= \Big( \big( \cup_{\mathtt i=1}^{\mathtt t} \operatorname{EndP}(P_{\mathtt i}) \big) \cup \big( \cup_{\mathtt j=1}^{\mathtt r} \operatorname{EndP}(Q_{\mathtt j}) \big) \Big) \setminus V(K_1 \cup K_2) \,.
\end{align*}
By \eqref{eq-def-mathtt-W}, \eqref{eq-def-mathtt-L} and Definition~\ref{def-bad-vertex-set}, we have $((\mathcal B(\chi) \setminus \mathtt W )\cup \mathtt L_1) \cap V(K_1 \cup K_2)=\emptyset$. In addition, we can see that each vertex in $(\mathcal B(\chi) \setminus \mathtt W )\cup \mathtt L_1$ has degree at least $1$ in $S_1\cup S_2$ but has degree $0$ in $\Check{K}_1,\Check{K}_2$. Thus, from Item (ii) in Corollary~\ref{cor-revised-decomposition-H-Subset-S} we have
\begin{align*}
    (\mathcal B(\chi) \setminus \mathtt W )\cup \mathtt L_1 \subset \big( \cup_{\mathtt i=1}^{\mathtt t} \operatorname{EndP}(P_{\mathtt i}) \big) \cup \big( \cup_{\mathtt j=1}^{\mathtt r} \operatorname{EndP}(Q_{\mathtt j}) \,.
\end{align*}
In conclusion, we have that $(\mathcal B(\chi) \setminus \mathtt W )\cup \mathtt L_1 \subset \mathtt S \subset \mathtt V$ and
\begin{align*}
    |\mathtt S| \leq 2(\mathtt t + \mathtt r) \leq 10 \sum_{i=1,2} \big( |\mathcal L(S_i) \setminus V(\check K_i)| + \tau(S_i) -\tau(K_i) +\ell \big) \overset{\eqref{eq-def-Gamma}}{\leq} 10 (\Gamma_1 + 2\ell) \,,
\end{align*}
where the second inequality follows from Item (iii) in Corollary~\ref{cor-revised-decomposition-H-Subset-S}. Recall that we use $\varkappa$ and $\gamma$ to denote the community labeling restricted on $\mathtt V$ and $[n] \setminus \mathtt V$, respectively. Also recall the definition of $\varphi_{\gamma;K_1,K_2;H}(G'(\varkappa))$ in Claim~\ref{claim-condition-on-G-par-case-2} (recall that it was defined for the partial label $\gamma$ since it only depends on the community labeling on $[n]\setminus \mathtt V$). We have that for all $\varkappa|_{\mathtt S} = \eta|_{\mathtt S}$
\begin{align}
    & \Big| \mathbb E_{\widetilde \Pb} \Big[  \varphi_{\gamma;K_1,K_2;H}(G'(\varkappa)) \mid G(\operatorname{par})_{i,j} = \chi_{i,j} , (i,j) \in \operatorname{U} \setminus \mathtt E \Big] \Big| \nonumber \\
    =\ & \Big| \mathbb E_{\widetilde \Pb} \Big[ \varphi_{\gamma;K_1,K_2;H}(G'(\eta)) \mid G(\operatorname{par})_{i,j} = \chi_{i,j} , (i,j) \in \operatorname{U} \setminus \mathtt E \Big] \Big| \label{eq-equality-given-mathtt-S} \\
    \leq\ & \mathbb E \Big[ \big|\varphi_{\gamma;K_1,K_2;H}\big( G(\operatorname{par})|_{\mathtt E} \big) \big| \Big] \,, \label{eq-inequality-given-mathtt-S}
\end{align}
where the last inequality follows from $G'(\varkappa)_{i,j} \leq G(\operatorname{par})_{i,j}$ for all $\varkappa \in [k]^{\mathtt V}$ and $(i,j) \in \operatorname{U}$, and the fact that $|\varphi_{\gamma;K_1,K_2;H}|$ (as a function on $\{ 0,1 \}^{\mathtt E}$) is increasing.
In addition, for all $\zeta \in [k]^{\mathtt S},\eta \in [k]^{\mathtt V \setminus \mathtt S}$ we can write $h_{\eta\oplus\zeta\oplus\gamma}(S_1,S_2;K_1,K_2)$ as (recall \eqref{eq-def-h-sigma})
\begin{align*}
    &\frac{ (1-\delta)^{|E(S_1)|+|E(S_2)|-|E(K_1)|-|E(K_2)|} }{n^{\frac{1}{2}(|E(S_1)|+|E(S_2)|-|E(K_1)|-|E(K_2)|)}} \\
*\ & \prod_{\mathtt i=\mathtt 1}^{\mathtt t} \mathtt h_{ \eta\oplus\zeta } (P_{\mathtt i}) \prod_{\mathtt j=\mathtt 1}^{\mathtt r} \mathtt h_{ \eta\oplus\zeta } (Q_{\mathtt j}) \prod_{\mathtt i'=\mathtt 1}^{\mathtt x} \mathtt h_{ \eta\oplus\zeta } (C_{\mathtt i'}) \prod_{\mathtt j'=\mathtt 1}^{\mathtt y} \mathtt h_{ \eta\oplus\zeta } (D_{\mathtt j'}) \,,
\end{align*}
where for each $V(P_{\mathtt i})=\{u_0,\ldots,u_l\}$ with $\operatorname{EndP}(P_{\mathtt i})=\{u_0,u_l\}$ and $V(C_{\mathtt i'})=\{ v_0,\ldots,v_{l'} \}$
\begin{align*}
    \mathtt h_{ \eta\oplus\zeta } (P_{\mathtt i}) = \omega(\zeta_{u_0},\eta_{u_1}) \omega(\eta_{u_{l-1}},\zeta_{u_{l}}) \prod_{m=1}^{l-2} \omega(\eta_{u_m},\eta_{u_{m+1}}), \quad \mathtt h_{ \eta\oplus\zeta } (C_{\mathtt i'}) = \prod_{m=0}^{l'} \omega(\eta_{u_m},\eta_{u_{m+1}}) \,,
\end{align*}
and $\mathtt h_{ \eta\oplus\zeta } (Q_{\mathtt j}), \mathtt h_{ \eta\oplus\zeta } (D_{\mathtt j'})$ are defined in the similar manner. By \eqref{eq-disjointness-cycles} and \eqref{eq-almost-disjointness-paths}, we have that given a fixed $\zeta \in [k]^{\mathtt S}$,
$\big\{ \mathtt h_{ \eta\oplus\zeta } (P_{\mathtt i}), \mathtt h_{ \eta\oplus\zeta } (C_{\mathtt i'}), \mathtt h_{ \eta\oplus\zeta } (Q_{\mathtt j}), \mathtt h_{ \eta\oplus\zeta } (D_{\mathtt j'}) \big\}$ (where $\eta\sim\nu_{\mathtt V \setminus \mathtt S}$) are conditionally independent. In addition, from \eqref{eq-finalgoal1-sign-of-path-cycle-conditioned-on-endpoints} we have that for each $P_{\mathtt i}$ 
\begin{align*}
    \Big| \mathbb E_{\eta\sim\nu_{\mathtt V \setminus \mathtt S}} \big[ \mathtt h_{ \eta\oplus\zeta } (P_{\mathtt i}) \mid \zeta \big] \Big| \leq k 
\end{align*}
and the same bound holds for each $Q_{\mathtt j}$, $C_{\mathtt i'}$ and $D_{\mathtt j'}$.
Thus for all $\zeta \in [k]^{\mathtt S}$ we have 
\begin{align}
    & \frac{(1-\delta)^{|E(S_1)|+|E(S_2)|-|E(K_1)|-|E(K_2)|}}{ k^{|\mathtt V \setminus \mathtt S|} } \Big| \sum_{\eta \in [k]^{\mathtt V \setminus \mathtt S}} h_{\eta\oplus\zeta\oplus\gamma}(S_1,S_2;K_1,K_2) \Big| \nonumber \\
    \leq\ & \frac{ \prod_{ \mathtt i=\mathtt 1 }^{ \mathtt t } k(1-\delta)^{|E(P_{\mathtt i})|} \prod_{ \mathtt j=\mathtt 1 }^{ \mathtt r } k(1-\delta)^{|E(D_{\mathtt j})|} \prod_{ \mathtt i'=\mathtt 1 }^{ \mathtt x } k(1-\delta)^{|E(C_{\mathtt i'})|} \prod_{ \mathtt j'=\mathtt 1 }^{ \mathtt y } k(1-\delta)^{|E(D_{\mathtt j'})|}  }{ n^{\frac{1}{2}(|E(S_1)|+|E(S_2)|-|E(K_1)|-|E(K_2)|)} }  \nonumber \\
    \leq\ & k^{\mathtt t+\mathtt r} \frac{ (1-\delta/2)^{|E(S_1)|+|E(S_2)|-|E(K_1)|-|E(K_2)|} }{n^{\frac{1}{2}(|E(S_1)|+|E(S_2)|-|E(K_1)|-|E(K_2)|)}}  \,, \label{eq-h-averaging}
\end{align}
where the second inequality follows from \eqref{eq-def-N} and $|E(C_{\mathtt i})|,|E(D_{\mathtt j})|\geq N$. 
Thus, by \eqref{eq-equality-given-mathtt-S} and \eqref{eq-inequality-given-mathtt-S} we have that \eqref{eq-conditional-expectation} is bounded by 
\begin{align*}
    &\mathbb E \Big[ \big| \varphi_{\gamma;K_1,K_2;H}\big( G(\operatorname{par})|_{\mathtt E} \big) \big| \Big] \cdot \frac{1}{k^{|\mathtt S|}} \sum_{ \zeta \in [k]^{\mathtt S} } \frac{1}{k^{|\mathtt V \setminus \mathtt S |}} \Big| \sum_{\eta \in [k]^{\mathtt V \setminus \mathtt S}} h_{\eta\oplus\zeta\oplus\gamma}(S_1,S_2;K_1,K_2) \Big| \\
    \overset{\eqref{eq-h-averaging}}{\leq}\ & \mathbb E \Big[ \big| \varphi_{\gamma;K_1,K_2;H}\big( G(\operatorname{par})|_{\mathtt E} \big) \big| \Big] \cdot k^{ \mathtt t + \mathtt r } \frac{ (1-\delta/2)^{|E(S_1)|+|E(S_2)|-|E(K_1)|-|E(K_2)|} }{n^{\frac{1}{2}(|E(S_1)|+|E(S_2)|-|E(K_1)|-|E(K_2)|)}} \,,
\end{align*} 
as desired.

\subsection{Proof of Claim~\ref{claim-prob-bad-set-realization}}{\label{subsec:Proof-Claim-B.7}}

This subsection is devoted to the proof of Claim~\ref{claim-prob-bad-set-realization}. We first outline our strategy for bounding $\widetilde{\Pb}(\mathcal B(G(\operatorname{par})|_{\operatorname{U} \setminus \mathtt E}) = \mathtt B )$. Roughly speaking, we will show that for all $\chi \in \{ 0,1 \}^{\operatorname{U} \setminus \mathtt E}$ such that $\mathcal B(\chi)=\mathtt B$, there exists a subgraph $\mathtt G \subset \chi \oplus \mathtt 1_{\mathtt E}$ such that $V(K_1) \cup V(K_2) \cup \mathtt B \subset V(\mathtt G)$ and $\mathtt G$ has high edge density (or equivalently, $\Phi(\mathtt G)$ is small). With this observation, we can reduce the problem to bounding the probability that in the graph $G(\operatorname{par})|_{\operatorname{U} \setminus \mathtt E} \oplus \mathtt 1_{\mathtt E}$ there exists a subgraph with high edge density and it turns out that a union bound suffices for this, though some further delicacy in bounding enumerations of such subgraphs also arise.

Intuitively, the existence of $\mathtt G$ follows from the fact that there exists $\mathtt I\subset \mathtt B$ such that for each $u\in \mathtt I$ there exists a self-bad graph $B_u$ containing $u$, and for each $u\in \mathtt B\setminus \mathtt I$ there exists a small cycle $C_u$ containing $u$. We expect the graph $H \cup (\cup_{u \in \mathtt I} B_u) \cup (\cup_{u \in \mathtt B \setminus \mathtt I} C_u)$ to have high edge density (and it contains all the vertices in $V(K_1) \cup V(K_2) \cup \mathtt B$, as desired). To verify this, we list $\mathtt B$ as $\{ u_1, \ldots, u_{\mathtt M} \}$ in an arbitrary order and we define $\mathtt G_{i}$ to be the subgraph in $\chi\oplus\mathtt 1_{\mathtt E}$ induced by
\[
V(H) \cup \big( \cup_{j \leq i,u_j \in \mathtt I} V(B_{u_j}) \big) \cup \big( \cup_{j \leq i, u_j \in \mathtt B \setminus \mathtt I} V(C_{u_j}) \big)\,.
\]
We will track the change of $\Phi(\mathtt G_{i})$ and we will show that: (a) for each $u_j \in \mathtt I$ we have $\Phi(\mathtt G_{j}) \leq \Phi(\mathtt G_{j-1})$; (b) for each $u_j \in \mathtt B \setminus \mathtt I$ such that $V(C_{u_j})$ or the neighborhood of $V(C_{u_j})$ in $K_1 \cup K_2$ intersect with $\mathtt G_{j-1}$, we also have $\Phi(\mathtt G_{j}) \leq \Phi(\mathtt G_{j-1})$; (c) for each $u_j \in \mathtt B \setminus \mathtt I$ such that neither $V(C_{u_j})$ nor the neighborhood of $V(C_{u_j})$ in $K_1 \cup K_2$ intersects with $\mathtt G_{j-1}$, we have $\Phi(\mathtt G_j) \leq (2000 \Tilde{\lambda}^{22} k^{22})^{N} \Phi(\mathtt G_{j-1})$. Thus, to control $\Phi(\mathtt G_{\mathtt M})$, it suffices to bound the number of ``undesired'' vertices that fall into category (c). 

Now we present our proof formally. For each $\chi \in \{ 0,1 \}^{ \operatorname{U} \setminus \mathtt E }$, we write $\chi \sim \mathtt E$ if in the realization $\chi \oplus \mathtt 1_{\mathtt E}$, there is no self-bad graph $K$ with $D^3-N \leq |V(K)| \leq D^3$. Similar to \eqref{eq-upper-bound-prob-G^c-.2} and \eqref{eq-upper-bound-prob-G^c-.3} in Lemma~\ref{lem-Gc-is-typical}, we can show that $\Pb( \chi \not\sim \mathtt E ) \leq n^{-D^2}$ (recall we have assumed $D\geq 2\log_2 n$ at the beginning of Section~\ref{sec:detection-lower-bound}, and in \eqref{eq-upper-bound-prob-G^c-.2} we get an extra factor $2^{-|V(K)|} \leq 2^{-D^3 /2} \leq n^{-D^2}$). Thus, it suffices to bound the probability that $\chi \sim \mathtt E$ and $\mathcal B(\chi)=\mathtt B$. For each $\chi \sim \mathtt E$ and $\mathcal B(\chi)=\mathtt B$, denote 
\begin{align*}
    \mathcal B_{\operatorname{dense}}(\chi) = \Big\{ u \in \mathtt B : \exists K \subset \chi \oplus \mathtt 1_{\mathtt E}, u \in V(K), K \mbox{ is self-bad}, |V(K)| \leq D^3 \Big\} \,.
\end{align*}
Since $\chi\sim\mathtt E$, we also know that for all $u \in \mathcal B_{\operatorname{dense}}(\chi)$, there exists a self-bad graph $K=K(u)$ such that $|V(K)| \leq D^3-N$ (by $\chi \sim \mathtt E$, we have excluded self-bad graphs with the number of vertices in $[D^3 -N,D^3]$). In addition, for all $u \in \mathcal B(\chi) \setminus \mathcal B_{\operatorname{dense}}(\chi)$, there must exist a cycle $C_u \subset \chi \oplus \mathtt 1_{\mathtt E}$ with length at most $N$ such that $u \in V(C_u)$. Clearly, we have either $C_u=C_w$ or $V(C_u) \cap V(C_w)=\emptyset$ for all $u,w \in \mathcal B(\chi) \setminus \mathcal B_{\operatorname{dense}}(\chi)$, since otherwise $C_u \cup C_w$ is a self-bad graph containing $u$ and $w$ (leading to $u, w\in \mathcal B_{\operatorname{dense}}(\chi)$). This also implies that the cycle $C_u$ is unique for each $u \in \mathcal B(\chi) \setminus \mathcal B_{\mathrm{dense}}(\chi)$. Define
\begin{align*}
    \mathcal B_{\operatorname{cyc}}(\chi)= \Big\{ u \in \mathtt B \setminus \mathcal B_{\operatorname{dense}}(\chi) : C_u \text{ and its neighbors in }K_1\cup K_2\text{ do not intersect }H\Big\} \,.
\end{align*}
The set $\mathcal B_{\operatorname{cyc}}$ is the set of ``undesired'' vertices as we discussed at the beginning of this subsection. Our proof will follow the following three steps, as shown in the boldface font below.

\

\noindent{\bf Control the number of undesired vertices.} We first show that (recall \eqref{eq-def-Gamma})
\begin{equation}{\label{eq-bound-card-mathcal-B-cyc}}
    |\mathcal B_{\operatorname{cyc}}(\chi)| \leq 2N(\Gamma_1 + \Gamma_2 + \ell) \,.
\end{equation}
Recalling that $\mathcal B_{\operatorname{cyc}}(\chi) \subset \mathcal B(\chi)$ and our assumption that $|\mathcal B(\chi) \setminus \mathtt W|=\ell$, we have (recall \eqref{eq-def-mathtt-W})
\begin{align}
     & \# \Big( \mathcal B_{\operatorname{cyc}}(\chi) \cap \big( (V(S_1) \setminus V(K_1)) \cup (V(S_2) \setminus V(K_2)) \big) \Big) \nonumber \\
     \leq\ & |\mathtt L_1| + \ell \overset{\eqref{eq-def-mathtt-L}}{\leq} |\mathcal L(S_1) \setminus V(K_1)| + |\mathcal L(S_2) \setminus V(K_2)| + \ell \leq 2 \Gamma_1 + \ell \,, \label{ref-L-leq-double-gamma}
\end{align}
where the third inequality follows from applying Lemma~\ref{lem-property-H-Subset-S} to $K_1 \ltimes S_1$ and $K_2 \ltimes S_2$ (note that $\mathcal I (S_1) = \mathcal I(S_2) = \emptyset $) respectively. 
Clearly, it suffices to show that 
\begin{equation*}
    \# \Big( \mathcal B_{\operatorname{cyc}}(\chi) \cap \big( V(K_1) \cup V(K_2) \big) \Big) \leq 2N\Gamma_2 \,.
\end{equation*}
For all $u \in \mathcal B_{\operatorname{cyc}}(\chi) \cap \big( V(K_1) \cup V(K_2) \big)$, note that $u \not \in V(H)$, and thus we have $u \not \in V(K_1) \cap V(K_2)$. We may assume that $u \in V(K_1) \setminus V(K_2) \subset V(K_1) \setminus V(H)$. Since $H \ltimes K_1$ (which implies $\mathcal I(K_1) \subset \mathcal I (H) \subset V(H)$) we see that $u \not \in \mathcal I(K_1)$. Recall that $C_u$ is a cycle with length at most $N$.  Also recall that $V(C_u) \cap V(H)=\emptyset$, which implies (recall that $V(H) = V(S_1) \cap V(S_2)$)
$$
V(C_u) \cap \big( V(K_1) \cap V(K_2) \big) \subset V(C_u) \cap \big( V(S_1 \cap S_2) \big) = V(C_u) \cap V(H) = \emptyset \,.
$$
Also, from the fact that $S_1$ is admissible we have $C_u \not\subset K_1$. We now claim that
\begin{align}
    V(C_u) \cap (\mathcal L(K_1) \setminus V(H)) \neq \emptyset \,. \label{eq-char-C_u}
\end{align}
Indeed, suppose on the contrary that $V(C_u) \cap (\mathcal L(K_1) \setminus V(H)) = \emptyset$. Let $I_u$ be the connected component containing $u$ in $K_1 \cap C_u$. Since we have that $I_u \neq C_u$ from $C_u \not \subset K_1$, it must hold that either $V(I_u)=\{u\}$ or $\mathcal L(I_u ) \neq \emptyset$. If $V(I_u)=\{u\}$, since $u \not \in \mathcal I(K_1)$ there must exist a neighbor of $u$ (denoted as $y$) in $K_1$, and by $y \not\in V(I_u)$ we have $(u,y) \not \in E(C_u)$. If $\mathcal L(I_u) \neq \emptyset$, take an arbitrary $x \in \mathcal L(I_u)$. By the definition of $\mathcal B_{\operatorname{cyc}}(\chi)$ we have $x \not \in V(H)$ and thus $x \not \in \mathcal L(K_1)$ (by our assumption that $V(C_u) \cap (\mathcal L(K_1) \setminus V(H)) = \emptyset$). Thus there must exist a neighbor of $x$ (denoted as $y$) in $K_1$ such that $(x,y) \not \in E(C_u)$ (otherwise $x$ has two neighbors in $C_u \cap K_1$, which contradicts to $x \in \mathcal L(I_u)$). In conclusion, in both cases we have shown that there exists an $x \in V(C_u)$ (whereas $x=u$ in the first case) such that $x$ has a neighbor $y \in V(K_1)$ and $(x,y) \not \in E(C_u)$. Then using the definition of $u \in \mathcal B_{\operatorname{cyc}}(\chi)$ we see that $y \not \in V(H)$, which gives $y \in \mathtt W \subset \mathcal B (\chi)$ (recall \eqref{eq-def-mathtt-W} and recall our assumption that $\mathtt W \subset \mathcal B(\chi)$). Therefore, there exists either a self-bad graph $B_y$ with $|V(B_y)| \leq D^3$ (which further implies that $|V(B_y)| \leq D^3-N$ since $\chi\sim\mathtt E$) or a cycle $C_y$ with $|V(C_y)| \leq N$. In addition, $y \not \in V(C_u)$ since otherwise the graph $\mathring{C}_u$ with $V(\mathring{C}_u)=V(C_u)$ and $E(\mathring{C}_u)=E(C_u)\cup\{ (x,y) \}$ is a self-bad subgraph of $\chi\oplus\mathtt 1_{\mathtt E}$ containing $u$, contradicting to $u \not \in \mathcal B_{\operatorname{dense}}(\chi)$. Therefore, neither $B_y$ nor $C_y$ is identical to $C_u$. Since $(x,y)$ is an edge in $K_1 \cup K_2$, then accordingly the graph $B$ induced by $V(B_y) \cup V(C_u)$ or $V(C_y) \cup V(C_u)$ is a self-bad graph in $\chi\oplus\{\mathtt 1_{\mathtt E}\}$ with $|V(B)| \leq D^3$ and $u \in V(B)$, contradicting the fact that $u \not\in \mathcal B_{\operatorname{dense}}(\chi)$. This completes the proof of \eqref{eq-char-C_u}.
Using \eqref{eq-char-C_u} and the fact that either $C_u=C_w$ or $V(C_u) \cap V(C_w)=\emptyset$ for all $u,w \in \mathcal B_{\operatorname{cyc}}$, we see that $|\mathcal B_{\operatorname{cyc}}| \leq N|\mathtt L_2| \leq 2N\Gamma_2$, where the second inequality follows from Lemma~\ref{lem-property-H-Subset-S} by an argument similar to \eqref{ref-L-leq-double-gamma}.

\

\noindent{\bf Construct the dense graph $\mathtt G$.} Now based on \eqref{eq-bound-card-mathcal-B-cyc}, we construct a graph $\mathtt G \subset \chi \oplus \mathtt 1_{\mathtt E}$ as follows. Recall that for each $u \in \mathcal B_{\operatorname{dense}}(\chi)$, there exists a self-bad graph $B_u$ such that $u \in V(B_u)$ and $|V(B_u)| \leq D^3-N$. Thus, we have
\begin{align}
    \Phi(B_u \cup J) \overset{\text{Lemma~\ref{lemma-facts-graphs}(ii)}}{\leq} \frac{ \Phi(B_u) \Phi(J) }{ \Phi(B_u \Cap J) } \leq \Phi(J) \mbox{ for all } J \subset \mathcal K_n \,. \label{eq-increase-Phi-adding-dense}
\end{align}
For each $u \in \mathcal B(\chi) \setminus \mathcal B_{\operatorname{dense}}(\chi)$, if $V(C_u)$ intersect with $V(H)$, it is straightforward to check that 
\begin{align}
    \Phi(C_u \cup J) \leq \Phi(J) \mbox{ for all } J \supset H \,. \label{eq-increase-Phi-adding-cycle-1}
\end{align}
Similarly, if the neighborhood of $V(C_u)$ in $K_1 \cup K_2$ intersect with $V(H)$ (i.e. there exists $x \in V(C_u)$ and $y \in V(H)$ such that $(x,y) \in E(K_1) \cup E(K_2)$), it is straightforward to check that 
\begin{align}
    \Phi(C_u \cup J \cup \{ (x,y) \}) \leq \Phi(J) \mbox{ for all } J \supset H  \,. \label{eq-increase-Phi-adding-cycle-2}
\end{align}
Finally, if $u \in \mathcal B_{\operatorname{cyc}}(\chi)$, it is straightforward to check that 
\begin{align}
    \Phi(C_u \cup J) \leq ( 2000\Tilde{\lambda}^{22}k^{22} )^N \cdot \Phi(J) \mbox{ for all } J \supset H  \,. \label{eq-increase-Phi-adding-cycle-3}
\end{align}
Now we take $\mathtt G$ to be the subgraph in $\chi\oplus\mathtt 1_{\mathtt E}$ induced by 
\begin{align*}
    V(H) \cup \big( \cup_{u \in \mathcal B_{\operatorname{dense}}(\chi)} V(B_u) \big) \cup \big( \cup_{u \in \mathtt B \setminus \mathcal B_{\operatorname{dense}}(\chi)} V(C_u) \big) \,.
\end{align*}
We claim that $\mathtt G$ satisfies the following conditions:
\begin{enumerate}
    \item[(i)] $V(K_1) \cup V(K_2) \cup \mathtt B \subset V(\mathtt G)$ and $|V(\mathtt G)| \leq D^4$;
    \item[(ii)] $\mathcal I(\mathtt G), \mathcal L(\mathtt G) \subset V(H)$;
    \item[(iii)] All the independent cycles of $\mathtt G$ must intersect with $V(K_1) \cup V(K_2) \cup \mathtt B$; 
    \item[(iv)] $\Phi(\mathtt G) \leq (2000 \Tilde{\lambda}^{22} k^{22})^{2N^2(\Gamma_1+\Gamma_2+\ell)} \cdot \Phi(H)$.
\end{enumerate}
We check these four conditions one by one. Condition (i) is straightforward since (recall \eqref{eq-def-mathtt-W})
$$
V(K_1) \cup V(K_2) \cup \mathtt B = V(H) \cup \mathtt B
$$
for all $\mathtt W \subset \mathtt B$. Condition (ii) follows from the fact that $\mathcal I(B_u), \mathcal I(C_u), \mathcal L(B_u), \mathcal L(C_u)=\emptyset$. Condition (iv) directly follows from \eqref{eq-increase-Phi-adding-dense}, \eqref{eq-increase-Phi-adding-cycle-1}, \eqref{eq-increase-Phi-adding-cycle-2} and \eqref{eq-increase-Phi-adding-cycle-3}. As for Condition (iii), suppose on the contrary that there exists an independent cycle $C$ of $\mathtt G$ such that $V(C) \cap (V(K_1) \cup V(K_2) \cup \mathtt B)=\emptyset$. For $u \in \mathtt B \setminus \mathcal B_{\operatorname{dense}}(\chi)$ we have $u \not \in V(C)$, and we must have $V(C) \cap V(C_u)=\emptyset$ since otherwise $C$ is connected to $u$ in $\mathtt G$, contradicting to our assumption that $C$ is an independent cycle. In addition, for $u \in \mathcal B_{\operatorname{dense}}(\chi)$, we must have $V(C)\cap V(B_u)=\emptyset$, since otherwise we have $|V(B_u)|-|V(B_u \doublesetminus C)|=|V(B_u) \cap V(C)| >0$ and $|E(B_u)|-|E(B_u \doublesetminus C)| = |E(B_u) \cap E(C)| \leq |V(B_u)|-|V(B_u \doublesetminus C)|$, leading to $\Phi(B_u)> \Phi(B_u \doublesetminus C)$ and contradicting to the assumption that $B_u$ is self-bad. Altogether, we have 
\begin{align*}
    V(C) &\subset V(\mathtt G) \setminus \big( ( \cup_{u \in \mathcal B_{\operatorname{dense}}(\chi)} V(B_u) ) \cup ( \cup_{u \in \mathtt B \setminus \mathcal B_{\operatorname{dense}}(\chi)} V(C_u) ) \big) \\
    &\subset V(H) \subset V(K_1) \cup V(K_2) \cup \mathtt B \,,
\end{align*}
which contradicts to our assumption and thus verifies Condition (iii). In conclusion, we show that for all $\chi\in\{0,1\}^{\operatorname{U} \setminus \mathtt E}$ such that $\mathcal B(\chi)=\mathtt B$, there exists a graph $\mathtt G=\mathtt G(\chi) \subset \chi\oplus\mathtt 1_{\mathtt E}$ such that Conditions (i)--(iv) hold. Thus we have
\begin{align}
    \mathbb P_{\chi \sim G(\operatorname{par})|_{\operatorname{U} \setminus \mathtt E}}( \mathcal B(\chi)=\mathtt B ) \leq \mathbb P_{\chi \sim G(\operatorname{par})|_{\operatorname{U} \setminus \mathtt E}} ( \exists \mathtt G \subset \chi \oplus \mathtt 1_{\mathtt E} : \mbox{ Conditions (i)--(iv) hold} ) \,. \label{eq-bad-set-prob-relaxation-1}
\end{align}

\

\noindent{\bf Bound the probability in \eqref{eq-bad-set-prob-relaxation-1}.} Denote $\mathtt G_0$ the graph such that $V(\mathtt G_0)=\mathtt B \cup V(K_1) \cup V(K_2)$ and $E(\mathtt G_0)=E(K_1) \cup E(K_2)$. Now, applying Corollary~\ref{cor-revised-decomposition-H-Subset-S} with $\mathtt G_0 \ltimes \mathtt G$ (note that Condition (3) yields $\mathfrak{C}(\mathtt G,\mathtt G_0)=\emptyset$), we see that $E(\mathtt G) \setminus E(\mathtt{G}_0)$ can be decomposed into $\mathtt t$ paths $P_1,\ldots,P_{\mathtt t} \subset \chi$ such that
\begin{enumerate} 
    \item[(I)] $|V(P_{\mathtt i})| \leq D^4$ and $\mathtt t \leq 5D^4$;
    \item[(II)] $V(P_{\mathtt i}) \cap \big( \mathtt B \cup V(K_1 \cup K_2) \cup ( \cup_{\mathtt j \neq \mathtt i} V(P_{\mathtt j}) ) \big) = \operatorname{EndP}(P_{\mathtt i})$.
\end{enumerate}
In addition, we claim that
\begin{enumerate}
    \item[(III)] $\mathtt L \cup (\mathtt B \setminus \mathtt W) \subset \cup_{\mathtt i=1}^{\mathtt t} \operatorname{EndP}(P_{\mathtt i})$; 
    \item[(IV)] $\Phi\big( (\cup_{\mathtt i=\mathtt 1}^{\mathtt t} P_{\mathtt i}) \cup \mathtt G_0 \big) \leq (2000 \Tilde{\lambda}^{22} k^{22})^{2N^2(\Gamma_1+\Gamma_2+\ell)} \Phi(H)$.
\end{enumerate}
Note that Item (IV) follows directly from Condition (iv) above. We next verify Item (III). Since $\mathcal I(\mathtt G), \mathcal L(\mathtt G) \subset V(H)$, each vertex in $\mathtt L \cup (\mathtt B \setminus \mathtt W)$ has degree at least 2 in $\mathtt G$ but has degree at most 1 in $\mathtt G_0$ (recall \eqref{eq-def-mathtt-L} and \eqref{eq-def-mathtt-W}). Thus we have $\mathtt L \cup (\mathtt B \setminus \mathtt W) \subset \cup_{\mathtt i=1}^{\mathtt t}V(P_{\mathtt i})$, implying Item (III) together with Item (II). Now we can apply the union bound to conclude that
\begin{align}
    \eqref{eq-bad-set-prob-relaxation-1} \leq\ & \mathbb P_{\chi \sim G(\operatorname{par})|_{\operatorname{U} \setminus \mathtt E}}( \exists \mbox{ paths } P_1,\ldots,P_{\mathtt t} \subset \chi \mbox{ satisfying Item (I)--(IV)} ) \nonumber \\ 
    \leq\ & \sum_{ (P_1,\ldots,P_{\mathtt t}) \text{ satisfying (I)--(IV) } } \mathbb P_{\chi \sim G(\operatorname{par})|_{\operatorname{U} \setminus \mathtt E}}( P_1,\ldots,P_{\mathtt t} \subset \chi ) \nonumber \\
    \leq\ & \sum_{\mathtt t = 0}^{5D^4} \sum_{ X_1, \ldots, X_{\mathtt t} \leq D^4 } \big( \tfrac{k\lambda}{n} \big)^{ X_1 + \ldots + X_{\mathtt t} } \cdot \mathsf{NumPath}(X_1,\ldots,X_\mathtt t) \,, \label{eq-bad-set-prob-relaxation-2}
\end{align}
where $\mathsf{NumPath}(X_1,\ldots,X_\mathtt t)$ is defined to be
\begin{equation}{\label{eq-def-numpath}}
    \# \Big\{ ( P_1, \ldots P_{\mathtt t} ) \mbox{ satisfying (I)--(IV)} : |E(P_{\mathtt i})|=X_{\mathtt i} , \forall \mathtt i \leq \mathtt t \Big\} \,.
\end{equation}
Denote $p=\# \big( \cup_{\mathtt i=1}^{\mathtt t} \operatorname{EndP}(P_{\mathtt i}) \big) \setminus \big( \mathtt B \cup V(H) \big)$. Note that according to Remark~\ref{rmk-endpoints-have-degree-geq-3}, we may assume without loss of generality that for each $u \in \big( \cup_{\mathtt i=1}^{\mathtt t} \operatorname{EndP}(P_{\mathtt i}) \big) \setminus \big(\mathtt B \cup V(H)\big)$, $u$ belongs to at least $3$ different $P_{\mathtt j}$'s. Thus from Item~(III) we must have 
$$
\mathtt t \geq (|\mathtt L \cup (\mathtt B \setminus \mathtt W)| +3p)/2 = (|\mathtt L|+\ell+3p)/2 \,,
$$
where the equality follows from $\mathtt L \cap (\mathtt B \setminus \mathtt W) = \emptyset$, implied by $\mathtt L \subset \mathtt W$. In addition, from Item~(IV) we see that 
\begin{align*}
    \Phi\big( (\cup_{\mathtt i=\mathtt 1}^{\mathtt t} P_{\mathtt i}) \cup \mathtt G_0) 
    \leq (2000 \Tilde{\lambda}^{22} k^{22})^{ 2N^2(\Gamma_1+\Gamma_2+\ell)} \cdot \Phi(H) \,.
\end{align*}
Note that $|V(\mathtt G_0)|=|V(K_1)|+|V(K_2)|+|\mathtt L_1|+\ell-|V(H)|$ and $|E(\mathtt G_0)|=|E(K_1)|+|E(K_2)|-|E(H)|$. Recalling (13), we see that $\Phi\big( (\cup_{\mathtt i=\mathtt 1}^{\mathtt t} P_{\mathtt i}) \cup \mathtt G_0)$ equals to
\begin{align*}
    \big( \tfrac{2k^2 \Tilde{\lambda}^2 n}{D^{50}} \big)^{ X_{\mathtt 1} + \ldots + X_{\mathtt t}-\mathtt t + |\mathtt L_1|+\ell+p } \big( \tfrac{1000 \Tilde{\lambda}^{20} k^{20} D^{50}}{n} \big)^{ X_{\mathtt 1} + \ldots + X_{\mathtt t} } \cdot \Phi(K_1) \Phi(K_2)/\Phi(H) \,. 
\end{align*}
Combining  the preceding two displays, we obtain that
\begin{align}
    & \big( \tfrac{2k^2 \Tilde{\lambda}^2 n}{D^{50}} \big)^{ X_{\mathtt 1} + \ldots + X_{\mathtt t}-\mathtt t } \big( \tfrac{1000 \Tilde{\lambda}^{20} k^{20} D^{50}}{n} \big)^{ X_{\mathtt 1} + \ldots + X_{\mathtt t} } \nonumber \\ 
    \leq\ & \big( \tfrac{2k^2 \Tilde{\lambda}^2 n}{D^{50}} \big)^{ -|\mathtt L_1|-\ell-p } (2000 \Tilde{\lambda}^{22} k^{22})^{ 2N^2(\Gamma_1+\Gamma_2+\ell)} * \frac{ \Phi(H)^2 }{ \Phi(K_1) \Phi(K_2) } \nonumber \\
    \leq\ & \big( \tfrac{n}{D^{50}} \big)^{ \tau(K_1)+\tau(K_2)-2\tau(H)-|\mathtt L_1|-\ell-p } \cdot \frac{ (2000 \Tilde{\lambda}^{22} k^{22})^{ 2N^2(\Gamma_1+\Gamma_2+\ell)} }{ (1000 \Tilde{\lambda}^{20} k^{20})^{|E(K_1)|+|E(K_2)|-2|E(H)|} }  \,, \label{eq-requirement-X}
\end{align}
where the last inequality follows from (noticing that $H \subset K_1,K_2$ and $\Tilde{\lambda}\geq 1$)
\begin{align*}
    \tfrac{ \Phi(H)^2 }{ \Phi(K_1)\Phi(K_2) } &\overset{(13)}{=} \big( \tfrac{2\Tilde{\lambda}^2 k^2 n}{D^{50}} \big)^{ -|V(K_1)|-|V(K_2)|+2|V(H)| } \big( \tfrac{1000 \Tilde{\lambda}^{20} k^{20} D^{50}}{n} \big)^{ -|E(K_1)|-|E(K_2)|+2|E(H)| } \\
    &\leq \big( \tfrac{n}{D^{50}} \big)^{ -\tau(K_1)-\tau(K_2)+2\tau(H) } \cdot (1000 \Tilde{\lambda}^{20} k^{20})^{-(|E(K_1)|+|E(K_2)|-2|E(H)|)} \,.
\end{align*}
It remains to bound $\mathsf{NumPath}(X_1,\ldots,X_{\mathtt t})$. Firstly, we have at most $n^{p}$ possible choices for the set $\big( \cup_{\mathtt i=1}^{\mathtt t} \operatorname{EndP}(P_{\mathtt i}) \big) \setminus \big( \mathtt B \cup V(H) \big)$. Given this, we have at most $D^{8}$ possible choices of each $\operatorname{EndP}(P_{\mathtt i})$. Given the endpoints of $P_{\mathtt i}$, we have at most $n^{X_{\mathtt i}-1}$ possible choices for the remaining vertices of $P_{\mathtt i}$. Thus, we have
\begin{align*}
    \mathsf{NumPath}(X_1,\ldots,X_{\mathtt t}) \leq n^p D^{8 \mathtt t} n^{ X_1+ \ldots + X_{\mathtt t} - \mathtt t } \,.
\end{align*}
Plugging this estimation into \eqref{eq-bad-set-prob-relaxation-2} we obtain that
\begin{align}
    \eqref{eq-bad-set-prob-relaxation-2} &\leq \sum_{p \leq 5D^4} \sum_{\mathtt t \geq (|\mathtt L|+\ell+3p)/2} \sum_{ (X_1,\ldots,X_{\mathtt t}) \textup{ satisfying } \eqref{eq-requirement-X} } \big( \tfrac{k\lambda}{n} \big)^{ X_1 + \ldots + X_{\mathtt t} } n^p D^{8 \mathtt t} n^{ X_1+ \ldots + X_{\mathtt t} - \mathtt t } \nonumber \\
    &= \sum_{p \leq 5D^4} \sum_{\mathtt t \geq (|\mathtt L|+\ell+3p)/2} \sum_{ (X_1,\ldots,X_{\mathtt t}) \textup{ satisfying } \eqref{eq-requirement-X} } n^{- \mathtt t+p} D^{8 \mathtt t} (k\lambda)^{X_{\mathtt 1}+\ldots+X_{\mathtt t}} \,.  \label{eq-final-relaxation}
\end{align}
Recall \eqref{eq-requirement-X}. For $(X_1,\ldots,X_{\mathtt t})$ satisfying \eqref{eq-requirement-X}, we can then get that the quantity $n^{- \mathtt t} D^{8 \mathtt t} (k\lambda)^{X_{\mathtt 1}+\ldots+X_{\mathtt t}}$ is bounded by
\begin{align*}
    \frac{ (\tfrac{n}{D^{50}})^{ \tau(K_1)+\tau(K_2)-2\tau(H)-|\mathtt L_1|-\ell-p } (2000 \Tilde{\lambda}^{22} k^{22})^{ 2N^2(\Gamma_1+\Gamma_2+\ell)} }{ D^{42 \mathtt t} (1000 \Tilde{\lambda}^{19} k^{19})^{ (X_{\mathtt 1}+\ldots+X_{\mathtt t}) } (1000\Tilde{\lambda}^{20} k^{20})^{|E(K_1)|+|E(K_2)|-2|E(H)| } } \,, 
\end{align*}
which is in turn bounded by (denote $\iota \in (\tfrac{1}{4}, \tfrac{1}{3})$ such that $(\tfrac{n}{D^{50}})^{\iota}=n^{\frac{1}{4}}$)
\begin{align*}
    &\Big( n^{- \mathtt t} D^{8 \mathtt t} \lambda^{X_{\mathtt 1}+\ldots+X_{\mathtt t}} \Big)^{1-\iota} \\
    *\ &\Big( \frac{ (\tfrac{n}{D^{50}})^{ \tau(K_1)+\tau(K_2)-2\tau(H)-|\mathtt L_1|-\ell-p } (2000 \Tilde{\lambda}^{12} k^{12})^{ 2N^2(\Gamma_1+\Gamma_2+\ell)} }{ D^{42\mathtt t} (1000\Tilde{\lambda}^{19} k^{19})^{ (X_{\mathtt 1}+\ldots+X_{\mathtt t}) } (1000\Tilde{\lambda}^{20} k^{20})^{|E(K_1)|+|E(K_2)|-2|E(H)| } } \Big)^{\iota} \\
    \leq\ &  n^{- 2\mathtt t/3} D^{-8\mathtt t}(2k^2)^{ -(X_{\mathtt 1}+\ldots+X_{\mathtt t}) } \\
    *\ &\frac{ n^{ \frac{1}{4}( \tau(K_1)+\tau(K_2)-2\tau(H)-|\mathtt L_1|-\ell-p ) } (2000 \Tilde{\lambda}^{22} k^{22})^{2N^2(\Gamma_1+\Gamma_2+\ell)} }{ (4\Tilde{\lambda}^2 k^2)^{|E(K_1)|+|E(K_2)|-2|E(H)|} } \,.
\end{align*}
Thus, we have that \eqref{eq-final-relaxation} is bounded by (note that $\sum_{ X_1,\ldots,X_{\mathtt t} \geq 1 } (2k^2)^{-(X_{\mathtt 1}+\ldots+X_{\mathtt t})} D^{-8\mathtt t}=1+o(1)$)
\begin{align*}
    & \sum_{p \leq 5D^4} \sum_{\mathtt t \geq (|\mathtt L|+\ell+3p)/2} \frac{ n^{ \frac{1}{4}( \tau(K_1)+\tau(K_2)-2\tau(H)-|\mathtt L_1|-\ell-p ) } (2000 \Tilde{\lambda}^{22} k^{22})^{ 2N^2(\Gamma_1+\Gamma_2+\ell)} n^{-\frac{2}{3}\mathtt t+p} }{ (4\Tilde{\lambda}^2 k^2)^{|E(K_1)|+|E(K_2)|-2|E(H)|} }  \\  
    \leq\ & [1+o(1)] \cdot \frac{ n^{ \frac{1}{4}( \tau(K_1)+\tau(K_2)-2\tau(H)-2|\mathtt L_1|- |\mathtt L_2|-2\ell ) } (2000 \Tilde{\lambda}^{22} k^{22})^{ 2N^2(\Gamma_1+\Gamma_2+\ell)} }{ (4\Tilde{\lambda}^2 k^2)^{|E(K_1)|+|E(K_2)|-2|E(H)|} } \,.
\end{align*}
Combined with \eqref{eq-bad-set-prob-relaxation-1} and \eqref{eq-bad-set-prob-relaxation-2}, this yields the desired bound on $\widetilde{\Pb}(\mathcal B(G(\operatorname{par})|_{\operatorname{U} \setminus \mathtt E}) = \mathtt B )$.

\bibliographystyle{plain}

\end{document}